\UseRawInputEncoding
\documentclass[12pt,reqno]{article}

\usepackage{amsfonts}
\usepackage{mathrsfs}
\usepackage{graphicx}
\usepackage{amsmath,amsthm,amsfonts,amssymb,latexsym,mathrsfs,color,bm}

\usepackage[colorlinks=true,
linkcolor=webblue, filecolor=webbrown, citecolor=webred ]{hyperref}
\definecolor{webblue}{rgb}{0,.5,0}
\definecolor{webred}{rgb}{0,.5,0}
\definecolor{webbrown}{rgb}{.6,0,0}

\newtheorem{thm}{Theorem}[section]

\newtheorem{cor}[thm]{Corollary}
\newtheorem{prop}[thm]{Proposition}
\newtheorem{conj}[thm]{Conjecture}

\theoremstyle{definition}
\newtheorem{definition}[thm]{Definition}
\newtheorem{ex}[thm]{Example}
\newtheorem{ques}[thm]{Question}
\newtheorem{rem}[thm]{Remark}

\numberwithin{equation}{section}

\def\JS{\text{JS}}

\def\eop{\hbox{\kern1pt\vrule height6pt width4pt
depth1pt\kern1pt}\medskip}

\newcommand{\textbfit}[1]{\textbf{\textit{#1}}}

\newcommand{\be}{\begin{equation}}
\newcommand{\ee}{\end{equation}}

\def\reff#1{(\protect\ref{#1})}
\def\spose#1{\hbox to 0pt{#1\hss}}
\def\ltapprox{\mathrel{\spose{\lower 3pt\hbox{$\mathchar"218$}}
    \raise 2.0pt\hbox{$\mathchar"13C$}}}
\def\gtapprox{\mathrel{\spose{\lower 3pt\hbox{$\mathchar"218$}}
    \raise 2.0pt\hbox{$\mathchar"13E$}}}

\renewcommand{\qed}{ $\square$ \bigskip}

\newcommand{\diag}{\mathop{\rm diag}\nolimits}

\newcommand{\scro}{{\mathcal{O}}}

\newcommand{\scrp}{{\mathcal{P}}}

\newcommand{\bfr}{{\mathbf{r}}}

\newcommand{\bfw}{{\mathbf{w}}}
\newcommand{\bfx}{{\mathbf{x}}}

\newcommand{\x}{{\mathbf{x}}}

\newcommand{\Z}{{\mathbb Z}}
\newcommand{\N}{{\mathbb N}}

\newcommand{\ba}{{\bm{a}}}

\newcommand{\balpha}{{\bm{\alpha}}}

\newcommand{\sn}{{\rm sn}}
\newcommand{\cn}{{\rm cn}}
\newcommand{\dn}{{\rm dn}}

\newcommand{\euler}[2]{\genfrac{\langle}{\rangle}{0pt}{}{#1}{#2}}

\title{Coefficientwise Hankel-total positivity of the row-generating
polynomials for the output matrices of certain production matrices
\thanks{Supported partially by the National Natural
Science Foundation of China (No. 12022105) and the Natural Science
Foundation for Distinguished Young Scholars of Jiangsu Province (No.
BK20200048).
\newline\hspace*{5mm}
   {\it Email address:} bxzhu@jsnu.edu.cn (B.-X. Zhu)}}
\author{Bao-Xuan Zhu}
\date{\footnotesize School of Mathematics and Statistics,
         Jiangsu Normal University,
         Xuzhou 221116, PR China}

\begin{document}

\maketitle

\begin{abstract}
Total positivity of matrices is deeply studied and plays an
important role in various branches of mathematics. The aim of this
paper is to study the criteria for coefficientwise Hankel-total
positivity of the row-generating polynomials of generalized
$m$-Jacobi-Rogers triangles and their applications.

Using the theory of production matrices, we present the criteria for
coefficientwise Hankel-total positivity of the row-generating
polynomials of the output matrices of certain production matrices.
In particular, we gain a criterion for coefficientwise Hankel-total
positivity of the row-generating polynomial sequence of the
generalized $m$-Jacobi-Rogers triangle. This immediately implies
that the corresponding generalized $m$-Jacobi-Rogers triangular
convolution preserves the Stieltjes moment property of sequences and
its zeroth column sequence is coefficientwise Hankel-totally
positive and log-convex of higher order in all the indeterminates.
In consequence, for $m=1$, we immediately obtain some results on
Hankel-total positivity for the Catalan-Stieltjes matrices. In
particular, we in a unified manner apply our results to some
combinatorial triangles or polynomials including the generalized
Jacobi Stirling triangle, a generalized elliptic polynomial, a
refined Stirling cycle polynomial and a refined Eulerian polynomial.
For the general $m$, combining our criterion and a function
satisfying an autonomous differential equation, we present different
criteria for coefficientwise Hankel-total positivity of the
row-generating polynomial sequence of exponential Rirodan arrays. In
addition, we also derive some results for coefficientwise
Hankel-total positivity in terms of compositional functions and
$m$-branched Stieltjes-type continued fractions. Finally, we apply
our criteria to: (1) rook polynomials and signless Laguerre
polynomials (confirming a conjecture of Sokal on coefficientwise
Hankel-total positivity of rook polynomials), (2) labeled trees and
forests (proving some conjectures of Sokal on total positivity and
Hankel-total positivity), (3) $r$th-order Eulerian polynomials
(giving a new proof for the coefficientwise Hankel-total positivity
of $r$th-order Eulerian polynomials, which in particular implies the
conjecture of Sokal on the coefficientwise Hankel-total positivity
of reversed $2$th-order Eulerian polynomials), (4) multivariate Ward
polynomials, labeled series-parallel networks and nondegenerate
fanout-free functions, (5) an array from the Lambert function and a
generalization of Lah numbers and associated triangles, and so on.
\bigskip\\
{\sl \textbf{MSC}:}\quad 05A20; 05A15; 15B48; 30B70; 44A60; 15B05;
30E05
\bigskip\\
{\sl \textbf{Keywords}:}\quad Total positivity; Hankel-total
positivity; Toeplitz-total-positivity; Stieltjes moment sequences;
P\'olya frequency; $3$-$\textbf{x}$-log-convexity; Convolutions;
Production matrices; Exponential Riordan arrays; Stieltjes-type
continued fractions; Jacobi-type continued fractions; Branched
continued fractions; Stieltjes-Rogers polynomials; Jacobi-Rogers
polynomials; Laguerre polynomials; $r$th-order Eulerian polynomials;
Multivariate Ward polynomials; Elliptic functions; Lambert functions
\end{abstract}

\tableofcontents
\section{Introduction}
Recently, in \cite{PSZ18}, based on lattice paths and branched
continued fractions, P\'{e}tr\'{e}olle, Sokal and Zhu developed the
theory for coefficientwise Hankel-total positivity of the
$m$-Stieltjes-Rogers polynomial sequence in all the indeterminates.
Note that an $m$-Stieltjes-Rogers polynomial sequence coincides with
the zeroth column of certain $m$-Jacobi-Rogers triangle. The purpose
of this paper is to consider a more generalized problem whether the
row-generating polynomial sequence of the generalized
$m$-Jacobi-Rogers triangles is coefficientwise Hankel-totally
positive. Since the generalized $m$-Jacobi-Rogers triangle can be
viewed as the output matrix of certain lower-Hessenberg matrix, in
terms of the theory of production matrices, in general, total
positivity of production matrices implies coefficientwise
Hankel-total positivity of the zeroth column of the output matrices,
but does not deduce that of its row-generating polynomials. This
motivates us to consider the stronger properties of production
matrices and propose a stronger concept: binomial-total positivity.
We develop some criteria for coefficientwise Hankel-total positivity
of the row-generating polynomials of output matrices and then
present many applications for the generalized $m$-Jacobi-Rogers
triangle.

The organization of this paper is as follows. We introduce the
background and definitions from total positivity in subsection
\ref{section+the background+Definitions} and give the definitions of
the $m$-Stieltjes-Rogers triangle, the $m$-Jacobi-Rogers triangle
and the generalized $m$-Jacobi-Rogers triangle in terms of lattice
paths in subsection \ref{subsec+m+SR+triangle}. In Section
\ref{sec+TP+production}, we introduce the general theory for total
positivity from production matrices and present our criteria to
coefficientwise Hankel-total positivity of the row-generating
polynomials of output matrices. We provide some kinds of
binomial-totally positive matrices in Section \ref{sec+Bino-TP} and
use such matrices to give the general criteria for coefficientwise
Hankel-total positivity of the row-generating polynomials of the
generalized $m$-Jacobi-Rogers triangle in Section
\ref{sec+generalized $m$-Jacobi--Rogers triangle}. We apply our
criteria to a variety of combinatorial triangles from
Catalan-Stieltjes matrices in Section \ref{sec+row-Cat-Stie} and
exponential Riordan arrays in Section \ref{section+EEA} in a unified
viewpoint.

\subsection{Total positivity}\label{section+the background+Definitions}
Total positivity of matrices has been deeply studied and plays an
important role in various branches of mathematics. The literature on
this fascinating subject is extensive such as classical analysis
\cite{Sc30}, representation theory
\cite{Lus94,Lusztig_98,Lusztig_08,Rie03}, network analysis
\cite{Pos06}, cluster algebra \cite{BFZ96,FZ99}, positive
Grassmannians and integrable systems \cite{KW11,KW14}, combinatorics
\cite{Bre95,GV85}. We also refer the reader to monographs
\cite{Kar68,Pin10} for more details.

A matrix of real numbers is {\it \textbf{totally positive}} if all
its minors are nonnegative and is {\it \textbf{totally positive of
order $r$}} if all minors of order $k\le r$ are nonnegative. For a
sequence $\ba = (a_n)_{n \ge 0}$, denote by
$\Gamma(\ba)=[a_{i-j}]_{i,j\ge 0}$ its {\it \textbf{Toeplitz
matrix}} and by $H(\ba)=[a_{i+j}]_{i,j\ge 0}$ its {\it
\textbf{Hankel matrix}}. Their total positivity plays an important
role in different fields.

Total positivity of $\Gamma(\ba)$ can characterize the P\'olya
frequency of $\ba$. In classical analysis, the sequence $\ba$ with
$a_0=1$ of nonnegative real numbers is called a {\it \textbf{P\'olya
frequency}} sequence if its generating function has the form
\begin{eqnarray}\label{PF+repesentation}
\sum_{n\ge 0}a_nz^n
=\frac{\prod_{j\ge 1}(1+\alpha_jz)}{\prod_{j\ge
1}(1-\beta_jz)}e^{\gamma z}
\end{eqnarray} in some open disk
centered at the origin, where $\alpha_j,\beta_j,\gamma\ge 0$ and
$\sum_{j\ge 1}(\alpha_j+\beta_j)<+\infty$, see
Karlin~\cite[pp.~412]{Kar68} for instance. We say that
$\sum_{n\geq0}a_nt^n$ is a {\it \textbf{P\'olya frequency ogf}} in
$\mathbb{R}[[t]]$ if the sequence $\ba$ is a P\'olya frequency
sequence. The P\'olya frequency is closely related to polynomials
with only real zeros. In fact, from the function representation
(\ref{PF+repesentation}) for the P\'olya frequency, it in particular
implies that $a_0,a_1,\ldots,a_n,0,0,\ldots$ is a P\'olya frequency
sequence if and only if $\sum_{k= 0}^na_kz^k$ has only real zeros
(\cite[pp.~399]{Kar68}). It is well known that the sequence $\ba$ is
a P\'olya frequency sequence if and only if its infinite Toeplitz
matrix $\Gamma(\ba)$ is totally positive. We refer the reader to
Br\"and\'en \cite{Bra06}, Brenti~\cite{Bre89,Bre94,Bre95} and
Wang-Yeh \cite{WYjcta05} for the P\'olya frequency in combinatorics.

Total positivity of $H(\ba)$ is closely related to the Stieltjes
moment property of $\ba$. The sequence $\ba$ of real numbers is a
{\it \textbf{Stieltjes moment sequence }}if it has the form
\begin{equation}\label{i-e}
a_k=\int_0^{+\infty}x^kd\mu(x),
\end{equation}
where $\mu$ is a non-negative measure on $[0,+\infty)$ (see
\cite[Theorem 4.4]{Pin10} for instance). The Stieltjes moment
problem is one of classical moment problems and arises naturally in
many branches of mathematics \cite{ST43,Wid41}. Stieltjes proved
that $\ba$ is a Stieltjes moment sequence if and only if there exist
nonnegative numbers $\alpha_0,\alpha_1,\ldots$ such that
$\sum_{n\geq0}a_nz^n$ has the Stieltjes-type continued fraction
\begin{eqnarray}\label{eq+SCF}
\sum_{n\geq0}a_nz^n=\frac{1}{1-\cfrac{\alpha_0z}{1-\cfrac{\alpha_1z}{1-\cdots}}}
\end{eqnarray} in the sense of formal power series. It is well known
that an equivalent characterization for $\ba$ being a Stieltjes
moment sequence is that its Hankel matrix $H(\ba)$ is totally
positive \cite{Pin10}. The sequence $\ba$ is {\it log-convex} if
$a_{k-1}a_{k+1}\ge a_k^2$ for all $k\ge 1$. Clearly, a sequence of
positive numbers is log-convex if and only if $H(\ba)$ is totally
positive of order two. Thus, log-convexity is implied by its
Stieltjes moment property. In fact, log-convexity of many
combinatorial sequences can be extended to Stieltjes moment
property. See Liu and Wang \cite{LW07} and Zhu \cite{Zhu13} for
log-convexity and Liang et al. \cite{LMW16}, Wang and Zhu
\cite{WZ16} and Zhu\cite{Zhu19,Zhu191} for Stieltjes moment property
in combinatorics.

In fact, this is only the beginning of the story. Because some
combinatorial objects with respect to one or more statistics often
generate multivariate polynomials, our more interest is to study
sequences and matrices of polynomials in one or more indeterminates
$\textbf{x}$. Let $\mathbb{R}$ denote the set of all real numbers
and $\textbf{x}=\{x_i\}_{i\in{I}}$ be a set of indeterminates. A
matrix $M$ with entries in $\mathbb{R}[\textbf{x}]$ is
\textbf{coefficientwise totally positive} in $\textbf{x}$ (we also
call it $\textbf{x}$-totally positive) if all its minors are
polynomials with nonnegative coefficients in the indeterminates
$\textbf{x}$ and is \textbf{coefficientwise totally positive of
order $r$} in $\textbf{x}$ (we also call it an $\textbf{x}$-totally
positive of order $r$) if all its minors of order $k\le r$ are
polynomials with nonnegative coefficients in the indeterminates
$\textbf{x}$. A sequence $(\alpha_n(\textbf{x}))_{n\geq0}$ in
$\mathbb{R}[\textbf{x}]$ is \textbf{coefficientwise Toeplitz totally
positive} in $\textbf{x}$ (we also call it an
$\textbf{x}$\textbf{-P\'olya frequency} sequence) if its associated
infinite Toeplitz matrix is coefficientwise totally positive in
$\textbf{x}$. A sequence $(\alpha_n(\textbf{x}))_{n\geq0}$ in
$\mathbb{R}[\textbf{x}]$ is \textbf{coefficientwise Hankel totally
positive} in $\textbf{x}$ (we also call it
$\textbf{x}$-\textbf{Stieltjes moment}) if its associated infinite
Hankel matrix is coefficientwise totally positive. Similarly, we
have \textbf{coefficientwise Hankel totally positive of order $r$}
and \textbf{coefficientwise Toeplitz totally positive of order $r$}
in $\textbf{x}$. It is \textbf{x-log-convex } if all coefficients of
$\alpha_{n+1}(\textbf{x})\alpha_{n-1}(\textbf{x})-\alpha_n(\textbf{x})^2$
are nonnegative for all $n\geq1$. Clearly, an $\textbf{x}$-Stieltjes
moment sequence is $\textbf{x}$-log-convex. Define an operator
$\mathcal {L}$ by
$$\mathcal {L}[\alpha_i(\textbf{x})]:=\alpha_{i-1}(\textbf{x})\alpha_{i+1}(\textbf{x})-\alpha_i(\textbf{x})^2$$
for $i\geq1$. In general, we say that $(\alpha_i(\textbf{x}))_{i\geq
0}$ is {\it $\textbf{k}$-\textbf{x-log-convex}} if the coefficients
of $\mathcal {L}^m[\alpha_i(\textbf{x})]$ are nonnegative for all
$m\leq k$, where $\mathcal {L}^m=\mathcal {L}(\mathcal {L}^{m-1})$.
If $\textbf{x}$ contains a unique indeterminate $q$, then they
reduce to be $q$-log-convex, $k$-$q$-log-convex and $q$-Stieltjes
moment, respectively. More and more combinatorial polynomials were
proved to have such properties. For example, the Bell polynomials,
the classical Eulerian polynomials, the Narayana polynomials of type
$A$ and $B$, Dowling polynomials, Jacobi-Stirling polynomials, and
so on, are $q$-log-convex (see Chen {\it et al.}
\cite{CTWY10,CWY11}, Liu and Wang \cite{LW07}, Zhu
\cite{Zhu13,Zhu14,Zhu17,Zhu182}, Zhu and Sun \cite{ZS15} for
instance), $3$-$q$-log-convex (see \cite{Zhu19}) and $q$-Stieltjes
moment (see \cite{WZ16,Zhu19,Zhu201}). We refer the reader to
\cite{PS19,PSZ18,Sok21,Sok,Zhu21,Zhu212,Zhu213} for coefficientwise
Hankel-total positivity in more indeterminates.

\subsection{The $m$-Stieltjes--Rogers triangle and $m$-Jacobi--Rogers
triangle} \label{subsec+m+SR+triangle}

The $m$-Stieltjes-Rogers polynomial was the zeroth column of the
$m$-Stieltjes--Rogers triangle. In \cite{PSZ18}, the
\textbfit{$\bm{m}$-Stieltjes--Rogers triangle} was defined in terms
of $m$-Dyck paths. Recall the definition of the $m$-Dyck path as
follows \cite{Aval_08,Cameron_16,Prodinger_16}:

\begin{definition} \label{def+mdyck}
\rm Fix an integer $m \ge 1$. A  \textbfit{partial $\bm{m}$-Dyck
path}\/ is a path in the upper half-plane $\Z \times \N$, starting
on the horizontal axis but ending anywhere, using steps $(1,1)$
(``rise'') and $(1,-m)$ (``$m$-fall''). In particular, a partial
$m$-Dyck path ending on the horizontal axis is called the
\textbfit{$\bm{m}$-Dyck path}: see Figure~$1$ for an example.
\end{definition}

\begin{center}
\setlength{\unitlength}{1cm}
\begin{picture}(22,4.5)(-2,-0.5)
\thicklines \put(0,0){\line(1,0){13}}
\thicklines\put(0,0){\line(0,1){4}}
\thicklines \put(0,0){\line(1,1){1}}
\thicklines\put(1,1){\line(1,-1){1}}
\thicklines\put(2,0){\line(1,1){3}}
\thicklines\put(5,3){\line(1,-2){1}}

\thicklines\put(6,1){\line(1,1){3}}
\thicklines\put(9,4){\line(1,-3){1}}
\thicklines\put(10,1){\line(1,-1){1}}
\put(0,0){\circle*{0.2}}\put(1,0){\circle*{0.2}}
\put(2,0){\circle*{0.2}}\put(3,0){\circle*{0.2}}
\put(4,0){\circle*{0.2}}\put(5,0){\circle*{0.2}}
\put(6,0){\circle*{0.2}}\put(7,0){\circle*{0.2}}
\put(8,0){\circle*{0.2}}\put(9,0){\circle*{0.2}}
\put(10,0){\circle*{0.2}}\put(11,0){\circle*{0.2}}
\put(12,0){\circle*{0.2}}\put(12.6,-0.1){$\rightarrow$}
\put(-0.095,3.75){$\uparrow$}
\put(1,1){\circle*{0.2}}\put(3,1){\circle*{0.2}}\put(4,2){\circle*{0.2}}\put(5,3){\circle*{0.2}}
\put(6,1){\circle*{0.2}}\put(7,2){\circle*{0.2}}\put(8,3){\circle*{0.2}}\put(9,4){\circle*{0.2}}
\put(10,1){\circle*{0.2}}

\end{picture}
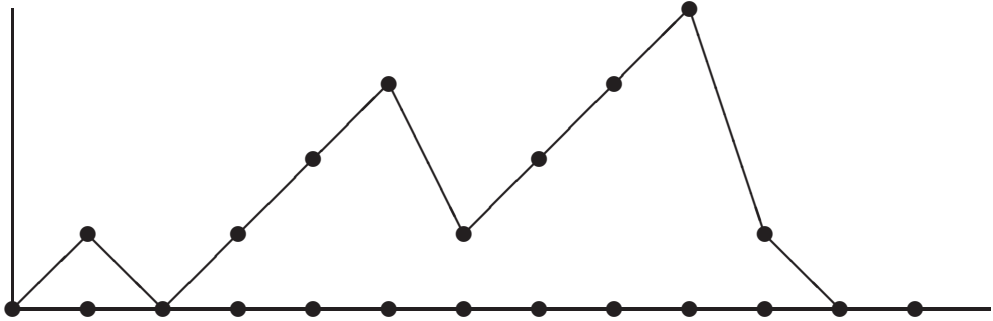
Figure~1. A $3$-Dyck path of length $10$.
\end{center}

In particular, for $m=1$, the $1$-Dyck path is the well-known Dyck
path. The definitions of the \textbfit{$\bm{m}$-Stieltjes--Rogers
triangle} and the \textbfit{$\bm{m}$-Stieltjes--Rogers polynomial}
 \cite{PSZ18} can be stated as follows:
\begin{definition} \label{def+m+S+R+CF}
\rm Fix an integer $m \ge 1$, and let $\balpha = (\alpha_i)_{i \ge
m}$ be an infinite set of indeterminates. Let
$S^{(m)}_{n,k}(\bm\alpha)$ be the generating polynomial for partial
$m$-Dyck paths from $(0,0)$ to ${((m+1)n,(m+1)k)}$ in~which each
rise gets weight~1 and each $m$-fall from height~$i$ gets weight
$\alpha_i$. We call the infinite unit-lower-triangular array
$\textbf{S}^{(m)}=[S^{(m)}_{n,k}(\bm\alpha)]_{n,k}$ the
\textbfit{$\bm{m}$-Stieltjes--Rogers triangle}. In particular, the
zeroth column $S^{(m)}_{n,0}(\bm\alpha)$ is called the
\textbfit{$\bm{m}$-Stieltjes--Rogers polynomial} of order~$n$,
denoted by $S^{(m)}_n(\balpha)$.
\end{definition}

 Generally,
$S_n^{(m)}(\balpha)$ is a homogeneous polynomial of degree~$n$ with
nonnegative integer coefficients. When $\alpha_i=1$ for all $i\geq
m$, the $m$-Stieltjes--Rogers polynomials $S^{(m)}_n(\textbf{1})$
reduce to the famous Fuss-Catalan numbers of order $p=m+1$:
$\frac{1}{pn+1}\binom{pn+1}{n}$. Let $f_0(t) = \sum_{n=0}^\infty
S^{(m)}_n(\balpha) \, t^n$ be the ordinary generating function for
$m$-Dyck paths with these weights. It is known that $f_0$ is given
by the \textbfit{$\bm{m}$-branched Stieltjes-type continued
fraction} \cite{PSZ18}:
\begin{eqnarray} \label{eq+f0+mSfrac}
   f_0(t)
   & = &
   \cfrac{1}
         {1 \,-\, \alpha_{m} t
            \prod\limits_{i_1=1}^{m}
                 \cfrac{1}
            {1 \,-\, \alpha_{m+i_1} t
               \prod\limits_{i_2=1}^{m}
               \cfrac{1}
            {1 \,-\, \alpha_{m+i_1+i_2} t
               \prod\limits_{i_3=1}^{m}
               \cfrac{1}{1 - \cdots}
            }
           }
         },
\end{eqnarray}
which for $m=1$ in particular reduces to the classical
Stieltjes-type continued fraction (\ref{eq+SCF}).

The following result on Hankel-total positivity for
$m$-Stieltjes--Rogers polynomials was one of main results in
\cite{PSZ18}.

\begin{thm}[Hankel-total positivity for $m$-Stieltjes--Rogers polynomials]
   \label{thm+BSCF+Hankel}
For each integer $m \ge 1$, the sequence $ ( S^{(m)}_n(\balpha) )_{n
\ge 0}$ of $m$-Stieltjes--Rogers polynomials is a coefficientwise
Hankel-totally positive sequence in the polynomial ring
$\Z[\balpha]$.
\end{thm}

From \cite{PSZ18}, we know that the $m$-Stieltjes--Rogers triangle
can be viewed as a special $m$-Jacobi--Rogers triangle defined by
the following well known $m$-\L{}ukasiewicz path.

\begin{definition}   \label{definition.lukasiewicz}
Fix $1 \le m \le \infty$. A {\em \textbf{partial m-\L{}ukasiewicz
path}}\/ is a path in the upper half-plane $\Z \times \N$, starting
on the horizontal axis but ending anywhere, using steps $(1,r)$ with
$-m \le r \le 1$: the allowed steps are thus $r=1$ (``rise''), $r=0$
(``level step''), and $r = -\ell$ for any $\ell > 0$
(``$\ell$-fall''). In particular, a partial $m$-\L{}ukasiewicz path
starting and ending on the horizontal axis is called the
$\bm{m}$-\textbfit{\L{}ukasiewicz path}. See Figure~$2$ for an
example.
\end{definition}

\begin{center}
\setlength{\unitlength}{1cm}
\begin{picture}(22,4.5)(-2,-0.5)
\thicklines \put(0,0){\line(1,0){13}}
\thicklines\put(0,0){\line(0,1){4}}
\thicklines\put(1,1){\line(1,0){1}}
\thicklines \put(0,0){\line(1,1){1}}
\thicklines\put(2,1){\line(1,1){2}}
\thicklines\put(4,3){\line(1,-2){1}}
\thicklines\put(5,1){\line(1,1){3}}
\thicklines\put(8,4){\line(1,0){1}}
\thicklines\put(9,4){\line(1,-3){1}}
\thicklines\put(10,1){\line(1,1){1}}
\thicklines\put(11,2){\line(1,-2){1}}
\put(0,0){\circle*{0.2}}\put(1,0){\circle*{0.2}}
\put(2,0){\circle*{0.2}}\put(3,0){\circle*{0.2}}
\put(4,0){\circle*{0.2}}\put(5,0){\circle*{0.2}}
\put(6,0){\circle*{0.2}}\put(7,0){\circle*{0.2}}
\put(8,0){\circle*{0.2}}\put(9,0){\circle*{0.2}}
\put(10,0){\circle*{0.2}}\put(11,0){\circle*{0.2}}
\put(12,0){\circle*{0.2}}\put(12.6,-0.1){$\rightarrow$}
\put(-0.095,3.75){$\uparrow$}
\put(1,1){\circle*{0.2}}\put(2,1){\circle*{0.2}}\put(3,2){\circle*{0.2}}\put(4,3){\circle*{0.2}}\put(5,1){\circle*{0.2}}
\put(6,2){\circle*{0.2}}\put(7,3){\circle*{0.2}}\put(8,4){\circle*{0.2}}\put(9,4){\circle*{0.2}}
\put(10,1){\circle*{0.2}}\put(11,2){\circle*{0.2}}

\end{picture}
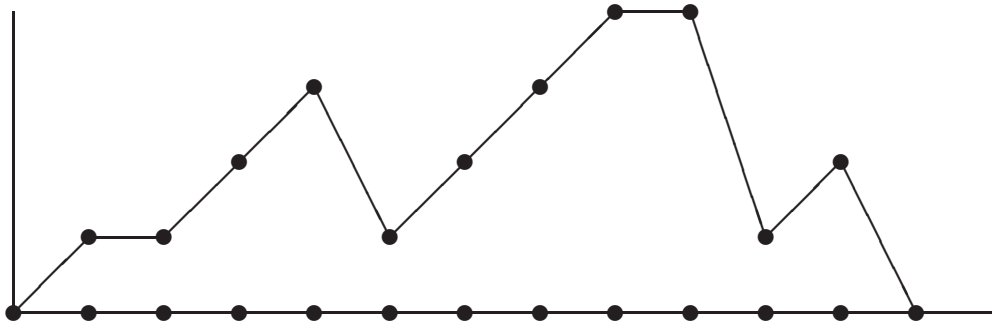
Figure~2. A $3$-\L{}ukasiewicz path of length $12$.
\end{center}

%

In what follows we introduce the $m$-Jacobi--Rogers triangle in
\cite{PSZ18} and define a slightly generalized
$\bm{m}$-Jacobi--Rogers triangle.
\begin{definition}   \label{definition.mJR}
\rm Fix $1 \le m \le \infty$, and let $\bm\beta =
(\beta_i^{(\ell)})_{-1 \le \ell \le m, \, i \ge \ell}$ be
indeterminates. Let $J^{(m)}_{n,k}(\bm\beta)$ be the generating
polynomial for partial $m$-\L{}ukasiewicz paths from $(0,0)$ to
$(n,k)$ in which each rise at height~$i$ gets
weight~$\beta_i^{(-1)}$, each level step at height~$i$ gets weight
$\beta_i^{(0)}$, and each $\ell$-fall from height~$i$ gets weight
$\beta_i^{(\ell)}$. Let the infinite lower-triangular array
$\textbf{J}^{(m)} = \big( J^{(m)}_{n,k}(\bm\beta) \big)_{n,k \ge
0}$. We call the infinite lower-triangular array $\textbf{J}^{(m)} =
\big( J^{(m)}_{n,k}(\bm\beta) \big)_{n,k \ge 0}$ {\it the
\textbfit{generalized $\bm{m}$-Jacobi--Rogers triangle}} and call
its zeroth column $J^{(m)}_{n,0}(\bm\beta)$ the \textbfit{
generalized $\bm{m}$-Jacobi--Rogers polynomial} of order~$n$,
denoted by $J^{(m)}_{n}(\bm\beta)$.
\end{definition}

For all $\beta_i^{(-1)}=1$, $\textbf{J}^{(m)}$ reduces to {\it the
\textbfit{$\bm{m}$-Jacobi--Rogers triangle}} and the zeroth column
$J^{(m)}_{n,0}$ reduces to the \textbfit{$\bm{m}$-Jacobi--Rogers
polynomial} of order~$n$ in \cite{PSZ18}. Though here is a slight
generalization, it will provide more plentiful combinatorial
triangles whose row-generating polynomials are coefficientwise
Hankel-totally positive (see Section \ref{sec+row-Cat-Stie}).

Let $f_0(t) = \sum_{n=0}^\infty J^{(m)}_{n}(\bm\beta) \, t^n$ be the
ordinary generating function for the $\bm{m}$-Jacobi--Rogers
polynomials. For $m=1$, we have the classical Jacobi-type continued
fraction:
\begin{eqnarray} \label{eq+f0+JF}
   f_0(t)
   & = &
   \cfrac{1}
         {1 \,-\, \beta_{0}^{(0)}t-
                             \cfrac{\beta_0^{(-1)}\beta_{1}^{(1)}t^2}
            {1 \,-\, \beta_{1}^{(0)} t-
                             \cfrac{\beta_1^{(-1)}\beta_{2}^{(1)}t^2}
            {1 \,-\, \beta_{2}^{(0)}t-
               \cfrac{\beta_2^{(-1)}\beta_{3}^{(1)}t^2}{1 - \cdots}
            }
           }
         },
\end{eqnarray}
which was studied by Flajolet \cite{Fla80} and where the
$1$-Jacobi--Rogers polynomials are the classical Jacobi--Rogers
polynomials \cite{Fla80}. For $m>1$, $ f_0(t)$ can be written as an
$m$-branched Jacobi-type continued fraction (see
\cite[pp.~{V-39}--\hbox{V-40}]{Viennot_83}, \cite[pp.~22--24,
143]{Roblet_94}, \cite[pp.~5]{Varvak_04} for instance).


For $1 \le m \le \infty$, in terms of the weights of partial
$m$-\L{}ukasiewicz paths, we define a matrix $P^{(m)}(\bm\beta)$ as
follows:
 $P^{(m)}(\bm\beta)$ is the $m$-banded lower-Hessenberg
       matrix with $\beta_i^{(-1)}$ on the superdiagonal and $i\geq0$,
       $\beta_0^{(0)},\beta_1^{(0)},\ldots$ on the diagonal,
       $\beta_1^{(1)},\beta_2^{(1)},\ldots$ on the first subdiagonal,
       $\beta_2^{(2)},\beta_3^{(2)},\ldots$ on the second subdiagonal, etc.,
       and in general
\be
   P^{(m)}(\bm\beta)_{ij}
   \;=\;
   \begin{cases}
       \beta_i^{(i-j)}   & \textrm{if $i-m \le j \le i+1$}  \\[1mm]
       0                 & \textrm{otherwise}
   \end{cases}
 \label{def.Pm}
\ee or in other words

\be
   P^{(m)}(\bm\beta)
   \;=\;
   \begin{bmatrix}
      \beta_0^{(0)}   & \beta_0^{(-1)}               &                 &     &     &    \\
      \beta_1^{(1)}   & \beta_1^{(0)}   & \beta_1^{(-1)}               &     &     &    \\
      \beta_2^{(2)}   & \beta_2^{(1)}   & \beta_2^{(0)}   & \beta_2^{(-1)}   &     &    \\
      \beta_3^{(3)}   & \beta_3^{(2)}   & \beta_3^{(1)}   & \beta_3^{(0)} & \beta_3^{(-1)}   & \\
      \vdots   &\vdots    & \vdots   & \vdots   & \vdots & \ddots
   \end{bmatrix}
   \;
 \label{def.Pm.0}.
 \ee
Obviously, $P^{(\infty)}(\bm\beta)$ is simply the generic
lower-Hessenberg matrix.

The $m$-Jacobi--Rogers triangle $\textbf{J}^{(m)} $ can be viewed as
the output matrix of $P^{(m)}(\bm\beta)$ for all $\beta_i^{(-1)}=1$
\cite{PSZ18}. Then the theory of production matrices provides a
powerful approach to coefficientwise Hankel-total positivity of
$m$-Jacobi--Rogers polynomials $J^{(m)}_{n}(\bm\beta)$.

\section{Total positivity from production
matrices}\label{sec+TP+production}

In recent years, the method of production matrices \cite{DFR05,DFR}
has become an important tool in enumerative combinatorics. It dates
back to the work (the Karlin iteration) of Karlin \cite[section 3.6,
pp. 132-140]{Kar68}. It also rose in denumerable Markov chains which
can only move one step to the right \cite{Kem66} and played an
important role in continued fractions (Stieltjes \cite{S89}).

Following \cite{PSZ18}, a \textbfit{partially ordered commutative
ring} is a pair $(R,\scrp)$ where $R$ is a commutative ring and
$\scrp$ is a subset of $R$ satisfying the following three
conditions: (1) $0,1 \in \scrp$. (2) If $a,b \in \scrp$, then $a+b
\in \scrp$ and $ab \in \scrp$. (3) $\scrp \cap (-\scrp) = \{0\}$. We
call $\scrp$ the {\em nonnegative elements}\/ of $R$, and we define
a partial order on $R$ (compatible with the ring structure) by
writing $a \le b$ as a synonym for $b-a \in \scrp$.

Now let $(R,\scrp)$ be a partially ordered commutative ring and let
$\bfx = \{x_i\}_{i \in I}$ be a collection of indeterminates. In the
polynomial ring $R[\bfx]$ and the formal-power-series ring
$R[[\bfx]]$, let $\scrp[\bfx]$ and $\scrp[[\bfx]]$ be the subsets
consisting of polynomials (resp.\ series) with nonnegative
coefficients. Then $(R[\bfx],\scrp[\bfx])$ and
$(R[[\bfx]],\scrp[[\bfx]])$ are partially ordered commutative rings;
we refer to this as the {\em coefficientwise order}\/ on $R[\bfx]$
and $R[[\bfx]]$.

A (finite or infinite) matrix with entries in a partially ordered
commutative ring is called \textbfit{totally positive} if all its
minors are nonnegative; it is called \textbfit{totally positive of
order~$\bm{r}$} if all its minors of size $\le r$ are nonnegative.
It follows immediately from the Cauchy--Binet formula that the
product of two totally positive (resp. totally positive of
order~$r$) matrices is totally positive (resp. totally positive of
order~$r$).

We say that a sequence $\ba = (a_n)_{n \ge 0}$ with entries in a
partially ordered commutative ring is \textbfit{Hankel-totally
positive} (resp.\ \textbfit{Hankel-totally positive of
order~$\bm{r}$}) if its associated infinite Hankel matrix $H(\ba) =
[a_{i+j}]_{i,j \ge 0}$ is totally positive (resp.\ totally positive
of order~$r$). It is \textbfit{Toeplitz-totally positive} (resp.\
\textbfit{Toeplitz-totally positive of order~$\bm{r}$}) if its
associated infinite Toeplitz matrix $\Gamma(\ba) = [a_{i-j}]_{i,j
\ge 0}$ is totally positive (resp.\ totally positive of order~$r$).
We also say that $f(t)=\sum_{n}f_nt^n$ is a \textbf{P\'olya
frequency ogf } (resp. \textbf{P\'olya frequency ogf }of order $r$)
in $R[[t]]$ if $(f_n)_{n\geq0}$ is Toeplitz-totally positive (resp.
Toeplitz-totally positive of order~$\bm{r}$).

Let $\textbf{P} = (p_{ij})_{i,j \ge 0}$ be an infinite matrix with
entries in a partially ordered commutative ring $R$. Assume that
$\textbf{P}$ is either row-finite (i.e.\ has only finitely many
nonzero entries in each row) or column-finite. Define an infinite
matrix $\textbf{A} = (a_{nk})_{n,k \ge 0}$ by \be
   a_{nk}  \;=\;  (\textbf{P}^n)_{0k}
 \label{def.iteration}
\ee (in particular, $a_{0k} = \delta_{0k}$). We can write $a_{n,k}$
as \be
   a_{nk}
   \;=\;
   \sum_{i_1,\ldots,i_{n-1}}
      p_{0 i_1} \, p_{i_1 i_2} \, p_{i_2 i_3} \,\cdots\,
        p_{i_{n-2} i_{n-1}} \, p_{i_{n-1} k}
   \;.
 \label{def.iteration.walk}
\ee We also have another equivalent formulation for $a_{nk}$ defined
by the recurrence \be
   a_{nk}  \;=\;  \sum_{i=0}^\infty a_{n-1,i} \, p_{ik}
   \qquad\hbox{for $n \ge 1$}
 \label{def.iteration.bis}
\ee with the initial condition $a_{0k} = \delta_{0k}$. We call
$\textbf{P}$ the \textbfit{production matrix} and $\textbf{A}$ the
\textbfit{output matrix} of $\textbf{P}$, and we write $\textbf{A} =
\scro(\textbf{P})$. Given a matrix $\textbf{M}$, define a matrix,
denoted by $\overline{\textbf{M}}$, obtained from the matrix
$\textbf{M}$ with the first row removed. So we have
\begin{eqnarray}
\overline{\scro(\textbf{P})}=\scro(\textbf{P})\textbf{P}.
\end{eqnarray}

When $\textbf{P}$ is coefficientwise totally positive, the output
matrix $\scro(\textbf{P})$ has two total-positivity properties as
follows:

\begin{thm}[Total positivity of the output matrix]\emph{\cite{PSZ18,Sok}}
   \label{thm+iteration+TP}
Let $\textbf{P}$ be a row-finite or column-finite matrix with
entries in a partially ordered commutative ring $R$. If $\textbf{P}$
is totally positive of order~$r$ \emph{($1\leq r\leq \infty$)}, then
\begin{itemize}
\item [\rm (i)]
the output matrix $\scro(\textbf{P})$ is totally positive of
order~$r$ \emph{($1\leq r\leq \infty$)};
\item [\rm (ii)]
the zeroth column of the output matrix $\scro(\textbf{P})$ is
Hankel-totally positive of order $r$.
\end{itemize}
\end{thm}

Let $\textbf{B}_q=[\binom{n}{k}q^{n-k}]_{n,k}$. For a matrix
$\textbf{M}=[M_{n,k}]_{n,k}$, define an associated matrix
$$\textbf{M}(q):=\textbf{M}\textbf{B}_q.$$
Following \cite{Sok21}, we call the matrix $\textbf{M}(q)$ the {\it
\textbf{binomial row-generating matrix}}. Let
$\textbf{M}(q)=[M_{n,k}(q)]_{n,k}$. Obviously, we have
\begin{eqnarray}\label{def+q+triangle}
\textbf{M}_{n,k}(q)=\sum_{i}M_{n,i}\binom{i}{k}q^{i-k}.
\end{eqnarray}

Assume that $\textbf{P}$ is a row-finite or column-finite matrix
with entries in a partially ordered commutative ring $R$ and its
output matrix $\scro(\textbf{P})$ is $\textbf{M}$. Since the
binomial matrix $\textbf{B}_q$ is coefficientwise totally positive
in the indeterminate $q$, it follows immediately from Theorem
\ref{thm+iteration+TP} (i) and the Cauchy-Binet theorem that:
\begin{itemize}
\item [\rm (i$'$)]
$\textbf{M}(q)$ is totally positive of order~$r$ in the ring $R[q]$
equipped with the coefficientwise order.
\end{itemize}
Note that the \textbf{row-generating polynomials}
$M_n(q)=\sum_{k=0}^nM_{n,k}q^k$ of $\textbf{M}$ locate exactly in
the zeroth column of $\textbf{M}(q)$. It is natural to ask whether
the row-generating polynomials $M_n(q)$ are Hankel-totally positive
of order $r$ or not. However, the foregoing hypotheses are not
sufficient to imply that the row-generating polynomials are
Hankel-totally positive in the ring $R[q]$ equipped with the
coefficientwise order, or even Hankel-totally positive in the ring
$R$ when $q$ is specialized to strictly positive values. The
following simple example \cite[Example 2.15]{Sok21} shows this:

\begin{ex}(Failure of Hankel-totally positive of the row-generating polynomials).
Let $\textbf{P} = \textbf{e}_{00}+\Delta$ be the upper-bidiagonal
matrix with $1$ on the superdiagonal and $1,0,0,0,\ldots$ on the
diagonal; it is totally positive because every bidiagonal matrix
with nonnegative entries is totally positive \cite[Lemma
2.1]{Sok21}. Then $\textbf{A} =\scro(\textbf{P})$ is the
lower-triangular matrix with all entries $1$, so that
$A_n(q)=\sum_{k=0}^nq^k$. Since $A_0(q)A_2(q)- A_1(q)^2 = -q$, the
sequence $(A_n(q))_{n\ge0}$ is not even log-convex (i.e.
Hankel-totally positive of order $2$) for any real number $q >0$.
\end{ex}

Thus, in order to consider the Hankel-total positivity of
row-generating polynomials of $\scro(\textbf{P})$, we need a
stronger condition on $\textbf{P}$, namely:

\begin{definition}
Let $\textbf{P}$ be a row-finite matrix with entries in a
commutative ring $R$. We say that $\textbf{P}$ is {\it
\textbf{binomial-totally positive of order $r$}} ($1\leq r \leq
\infty$) if the matrix $\textbf{B}^{-1}_q\textbf{P}\textbf{B}_q$ is
totally positive of order $r$ in the ring $R[q]$ equipped with the
coefficientwise order.
\end{definition}
Clearly, specializing to $q=0$, we see that binomial-total
positivity implies total positivity; but in general it is much
stronger. We will develop some ideas in \cite{PS19,Sok21,Zhu213} to
be further. The reason for the appearance here of the combination
$\textbf{B}^{-1}_q\textbf{P}\textbf{B}_q$ is the following simple
fact \cite[Lemma Lemma 2.6]{Sok21}: Let $\textbf{P}$ be a row-finite
matrix with entries in a commutative ring $R$, with its output
matrix $\textbf{A} =\scro(\textbf{P})$ and let $\textbf{B}
=(b_{ij})_{i,j\geq0}$ be a lower-triangular matrix with invertible
(in $R$) diagonal entries. Then $\textbf{AB}=
b_{00}\scro(\textbf{B}^{-1}\textbf{PB}).$ Applying this fact to the
case $\textbf{B} =\textbf{B}_q$, working in the ring $R[q]$ equipped
with the coefficientwise order, by Theorem \ref{thm+iteration+TP},
we immediately deduce:

\begin{thm}[Hankel-total positivity of row-generating polynomials]
   \label{thm+main+Hankel-Total positivity+1}
Let $\textbf{P}$ be a row-finite or column-finite matrix with
entries in a partially ordered commutative ring $R$. If $\textbf{P}$
is binomial-totally positive of order~$r$ \emph{($1\leq r\leq
\infty$)}, then
\begin{itemize}
\item [\rm (ii$'$)]
the sequence of row-generating polynomials of $\scro(\textbf{P})$ is
Hankel-totally positive of order $r$ in the polynomial ring $R[q]$
equipped with the coefficientwise order.
\end{itemize}
\end{thm}

Note that for any two matrices $\textbf{P}_1$ and $\textbf{P}_2$, we
have
\begin{eqnarray}
\textbf{B}^{-1}_q\textbf{P}_1\textbf{P}_2\textbf{B}_q&=&\underbrace{\textbf{B}^{-1}_q\textbf{P}_1\textbf{B}_q}\underbrace{\textbf{B}^{-1}_q\textbf{P}_2\textbf{B}_q}.
\end{eqnarray}

So, by the Cauchy-Binet theorem, the product of two binomial-totally
positive matrices is still binomial-totally positive. We immediately
have:

\begin{thm} \label{thm+main+Hankel-Total positivity+factor}
Let $\textbf{P}$ be a row-finite or column-finite matrix. If
$\textbf{P}$ is the product of some matrices being binomial-totally
positive of order~$r$ \emph{($1\leq r\leq \infty$)} with entries in
a partially ordered commutative ring $R$, then
\begin{itemize}
\item [\rm (i)]
$\textbf{M}(q)$ is totally positive of order~$r$ in the polynomial
ring $R[q]$ equipped with the coefficientwise order;
\item [\rm (ii)]
the sequence of row-generating polynomials of $\scro(\textbf{P})$ is
Hankel-totally positive of order $r$ in the polynomial ring $R[q]$
equipped with the coefficientwise order.
\end{itemize}
\end{thm}

Theorem \ref{thm+main+Hankel-Total positivity+1} and Theorem
\ref{thm+main+Hankel-Total positivity+factor} are our main tools to
prove the Hankel-total positivity of row-generating polynomials.

\section{Binomial-totally positive matrices}\label{sec+Bino-TP}

By Theorem \ref{thm+main+Hankel-Total positivity+factor}, it is
natural to study which kind of matrices are binomial-totally
positive. In this section, we will provide some new kinds of
matrices that are binomial-totally positive.

Define $\textbf{L}(a,b,c)$ to be a lower-bidiagonal matrix with
$ka+b$ for $k\geq0$ on the diagonal and $kc$ for $k\geq1$ on the
subdiagonal:
\begin{eqnarray}\label{def+lower-bi+L}
\textbf{L}(a,b,c)=\left[\begin{array}{ccccc}
b&&&&\\
c&a+b&&&\\
&2c&2a+b&&\\
&&3c&3a+b&\\
&&&\ddots&\ddots\\
\end{array}\right].
\end{eqnarray}
For the matrix $\textbf{L}(a,b,c)$, it is totally positive in the
ring $\mathbb{Z}[a,b,c]$ equipped with the coefficientwise order;
this is obvious since it is a bidiagonal matrix with entries that
are (coefficientwise) nonnegative. The following result shows that
$\textbf{L}(a,b,c)$ is also binomial-totally positive in the ring
$\mathbb{Z}[a,b,c]$ equipped with the coefficientwise order.

\begin{prop}\label{prop+low-bi+TP}
Let $\textbf{L}(a,b,c)$ be defined by (\ref{def+lower-bi+L}). Then
we have
\begin{eqnarray}\label{eq+B+L}
\textbf{B}^{-1}_q\textbf{L}(a,b,c)\textbf{B}_q=\textbf{L}(a,b,c+qa)
\end{eqnarray}
and the matrix $\textbf{L}(a,b,c)$ is binomial-totally positive in
the ring $\mathbb{Z}[a,b,c]$ equipped with the coefficientwise
order.
\end{prop}

\begin{proof}
 By (\ref{eq+B+L}), the total positivity of $\textbf{L}(a,b,c)$
in the ring $\mathbb{Z}[a,b,c]$ equipped with the coefficientwise
order immediately implies that $\textbf{L}(a,b,c)$ is
binomial-totally positive in the ring $\mathbb{Z}[a,b,c]$ equipped
with the coefficientwise order. So it suffices to prove
(\ref{eq+B+L}). We will show that
$$\textbf{L}(a,b,c)\textbf{B}_q=\textbf{B}_q\textbf{L}(a,b,c+qa).$$
Since \begin{eqnarray} [\textbf{L}(a,b,c)]_{n,n}=b+na,\quad\quad
[\textbf{L}(a,b,c)]_{n,n-1}=nc,
\end{eqnarray}
we have
\begin{align}
[\textbf{L}(a,b,c)B_q]_{n,k}&=(b+na)\binom{n}{k}q^{n-k}+nc\binom{n-1}{k}q^{n-1-k}\\
[B_q\textbf{L}(a,b,c+qa)]_{n,k}&=\binom{n}{k}q^{n-k}(b+ka)+\binom{n}{k+1}q^{n-1-k}(k+1)(c+qa)\\
&=\binom{n}{k}q^{n-k}(b+ka)+n(c+qa)\binom{n-1}{k}q^{n-1-k}.
\end{align}
In consequence, we obtain
\begin{eqnarray}
[\textbf{L}(a,b,c)B_q]_{n,k}-[B_q\textbf{L}(a,b,c+qa)]_{n,k}&=&0.
\end{eqnarray}
This completes the proof.
\end{proof}
Assume that $\textbf{U}(a,b,u,v)$ is a upper-bidiagonal matrix with
$ka+b$ for $k\geq0$ on the diagonal and $ku+v$ for $k\geq1$ on the
superdiagonal:
\begin{eqnarray}\label{def+upper-bi+U}
\textbf{U}(a,b,u,v)=\left[\begin{array}{ccccc}
b&v&&&\\
&a+b&u+v&&\\
&&2a+b&2u+v&\\
&&&\ddots&\ddots\\
\end{array}\right].
\end{eqnarray}
Suppose that $\textbf{J}(a,b,u,v,\lambda)$ is the following
tri-diagonal matrix
\begin{eqnarray}\label{def+tri+J}
\textbf{J}(a,b,u,v,\lambda)=\left[
\begin{array}{ccccc}
s_0 & r_0 &  &  &\\
t_1 & s_1 & r_1 &\\
 & t_2 & s_2 &r_2&\\
& & \ddots&\ddots & \ddots \\
\end{array}\right],
\end{eqnarray}
where $r_n=un+v$, $s_n=(2\lambda u+a)n+b+\lambda v$ and
$t_n=\lambda(a+\lambda u)n$.

Obviously, $\textbf{U}(a,b,u,v)=\textbf{J}(a,b,u,v,0)$. Since
$\textbf{U}$ is bidiagonal, it follows immediately that $\textbf{U}$
is totally positive in the ring $\mathbb{Z}[a,b,u,v]$ equipped with
the coefficientwise order. But it is less obvious that
$\textbf{J}(a,b,u,v,\lambda)$ is totally positive in the ring
$\mathbb{Z}[a,b,u,v,\lambda]$ equipped with the coefficientwise
order. In fact, the following result gives the stronger property
that $\textbf{J}(a,b,u,v,\lambda)$ is binomial-totally positive in
the ring $\mathbb{Z}[a,b,u,v,\lambda]$ equipped with the
coefficientwise order.

\begin{prop}\label{prop+tri+J+TP}
Let $\textbf{J}(a,b,u,v,\lambda)$ be defined by (\ref{def+tri+J}).
Then
\begin{itemize}
\item [\rm (i)]
the matrix $\textbf{J}(a,b,u,v,\lambda)$ is totally positive in the
ring $\mathbb{Z}[a,b,u,v,\lambda]$ equipped with the coefficientwise
order.
\item [\rm (ii)]
we have \begin{eqnarray}\label{eq+B+U}
\textbf{B}^{-1}_q\textbf{J}(a,b,u,v,\lambda)\textbf{B}_q=\textbf{J}(a,b,u,v,\lambda+q).
\end{eqnarray}
Therefore, the matrix $\textbf{J}(a,b,u,v,\lambda)$ is
binomial-totally positive in the ring $\mathbb{Z}[a,b,u,v,\lambda]$
equipped with the coefficientwise order.
\end{itemize}
\end{prop}

\begin{proof}
(i) Consider the matrix
\begin{equation} \textbf{J}^{*}=\left[
\begin{array}{ccccc}
s_0 & r^{*}_0 &  &  &\\
t^{*}_1 & s_1 & r^{*}_1 &\\
 & t^{*}_2 & s_2 &r^{*}_2&\\
& & \ddots&\ddots & \ddots \\
\end{array}\right],
\end{equation}
where $r^{*}_n=\lambda(un+v-u)$, $s_n=[2\lambda u+a]n+b+\lambda v$
and $t^{*}_n=[a+\lambda u]n$. It satisfies
$r_it_{i+1}=r^{*}_it_{i+1}^{*}$ for all $i$, so by \cite[Claim 2 in
the proof of Theorem 2.1]{Zhu21}, $\textbf{J}(a,b,u,v,\lambda)$ is
totally positive in the ring $\mathbb{Z}[a,b,u,v,\lambda]$ equipped
with the coefficientwise order if and only if $\textbf{J}^{*}$ is
so. On the other hand, the matrix $\textbf{J}^{*}$ is
coefficientwise row-diagonally-dominant (that is, the coefficients
of $s^{*}_n-r^{*}_n-t^{*}_n$ are nonnegative), so by
\cite[Proposition 3.3 (i)]{Zhu21}, it follows that $\textbf{J}^{*}$
is totally positive in the ring $\mathbb{Z}[a,b,u,v,\lambda]$
equipped with the coefficientwise order.

(ii) By (i) and $
\textbf{B}^{-1}_q\textbf{J}(a,b,u,v,\lambda)\textbf{B}_q=\textbf{J}(a,b,u,v,\lambda+q),
$ the matrix $\textbf{J}(a,b,u,v,\lambda)$ is immediately
binomial-totally positive in the ring $\mathbb{Z}[a,b,u,v,\lambda]$
equipped with the coefficientwise order. So we will prove
$\textbf{J}(a,b,u,v,0)\textbf{B}_q=\textbf{B}_q\textbf{J}(a,b,u,v,q).$
Since
\begin{align}
[\textbf{J}(a,b,u,v,\lambda)]_{n,n+1}&=un+v,\\
[\textbf{J}(a,b,u,v,\lambda)]_{n,n}&=(2\lambda u+a)n+b+\lambda v,\\
[\textbf{J}(a,b,u,v,\lambda)]_{n,n-1}&=\lambda(a+\lambda u)n,
\end{align}
we have
\begin{align}
[\textbf{J}(a,b,u,v,0)\textbf{B}_q]_{n,k}=&(un+v)\binom{n+1}{k}q^{n+1-k}+(an+b)\binom{n}{k}q^{n-k},\label{JB}\\
[\textbf{B}_q\textbf{J}(a,b,u,v,q)]_{n,k}=&\binom{n}{k-1}q^{n-k+1}[u(k-1)+v]+\binom{n}{k}q^{n-k}[(2qu+a)k+b+qv]+\nonumber\\
&\binom{n}{k+1}q^{n-k-1}q(a+qu)(k+1)\nonumber\\
=&\binom{n}{k-1}q^{n-k+1}[u(k-1)+v]+\binom{n}{k}q^{n-k}q(2uk+v)+\nonumber\\
&\binom{n}{k}q^{n-k}(ak+b)+(n-k)\binom{n}{k}q^{n-k}(a+qu)(k+1).\label{BJ}
\end{align}
Then it is easy to check that the two right-hand sides of (\ref{JB})
and (\ref{BJ}) are equal, which gives
$[\textbf{J}(a,b,u,v,0)\textbf{B}_q]_{n,k}=[\textbf{B}_q\textbf{J}(a,b,u,v,q)]_{n,k}.$
This completes the proof.\end{proof}

It is easy to observe that $\textbf{J}(a,b,0,0,\lambda)
=\textbf{L}(a,b,\lambda a)$. So if we write $\lambda= c/a$ and work
in a ring of Laurent polynomials, Proposition \ref{prop+low-bi+TP}
is a special case of Proposition \ref{prop+tri+J+TP}. In what
follows we will introduce a generalization of $\textbf{L}$,
$\textbf{U}$ and $\textbf{J}$. Define a tridiagonal matrix
\begin{equation} \textbf{M}(a,b,c,u,v):=\left[
\begin{array}{cccccc}
b & v &  &  &&\\
c & b+a & v+u &&\\
 & 2c & b+2a &v+2u&&\\
& & 3c&b+3a & v+3u &\\
& & &\ddots&\ddots&\ddots
\end{array}\right].
\end{equation}
It is of the form (\ref{def+tri+J}) with $r_n=v+un,s_n=b+an$ and
$t_n=cn$. Clearly, all of the matrices $\textbf{L}$, $\textbf{U}$
and $\textbf{J}$ are special cases of this general form: in
particular,
\begin{eqnarray}
\textbf{L}(a,b,c)&=&\textbf{M}(a,b,c,0,0),\\
\textbf{U}(a,b,u,v)&=&\textbf{M}(a,b,0,u,v),\\
\textbf{J}(a,b,u,v,\lambda)&=&\textbf{M}(2\lambda u+a,b+\lambda
u,\lambda (a+\lambda u),u,v).
\end{eqnarray}
For this matrix $\textbf{M}(a,b,c,u,v)$, we have the following key
property.

\begin{prop}\label{prop+binom+M}
We have
\begin{eqnarray}
\textbf{B}^{-1}_q\textbf{M}(a,b,c,u,v)\textbf{B}_q&=&\textbf{M}(2uq+a,qu+b,c+aq+uq^2,u,v).
\end{eqnarray}
\end{prop}
\begin{proof}
From definitions, we have
$\textbf{M}(a,b,c,u,v)=\textbf{U}(0,0,u,v)+\textbf{L}(a,b,c).$ So
\begin{eqnarray}
\textbf{B}^{-1}_q\textbf{M}(a,b,c,u,v)\textbf{B}_q&=&\textbf{B}^{-1}_q\textbf{U}(0,0,u,v)\textbf{B}_q+\textbf{B}^{-1}_q\textbf{L}(a,b,c)\textbf{B}_q.
\end{eqnarray}
By Proposition \ref{prop+low-bi+TP} and Proposition
\ref{prop+tri+J+TP}, we get
\begin{eqnarray}
\textbf{B}^{-1}_q\textbf{M}(a,b,c,u,v)\textbf{B}_q&=&\textbf{J}(0,0,u,v,q)+\textbf{L}(a,b,c+qa)\\
&=&\textbf{M}(2qu,qu,q^2u,u,v)+\textbf{M}(a,b,c+qu,0,0)\\
&=&\textbf{M}(2uq+a,qu+b,c+aq+uq^2,u,v).
\end{eqnarray}
This completes the proof.
\end{proof}

Proposition \ref{prop+binom+M} says that the family of matrices
$\textbf{M}(a,b,c,u,v)$ is mapped into itself under a binomial
similarity transformation. This will allow us to leverage total
positivity into binomial-total positivity.  The matrices
$\textbf{M}(a,b,c,u,v)$ are not in general totally positive; but by
suitably specializing the variables $a,b,c,u,v$, one can devise a
matrix $\widetilde{\textbf{M}}(a,b,c,d,e,\mu,\nu)$ in seven
indeterminates $a,b,c,d,e,\mu,\nu$ by
\begin{eqnarray}
\widetilde{\textbf{M}}(a,b,c,d,e,\mu,\nu):=\textbf{M}(c+e+(\mu+\nu)b,d+(\mu+\nu)a,(\mu+\nu)c+\nu
e,b,a)
\end{eqnarray}
 or in other words (\ref{def+tri+J})
with
\begin{eqnarray}
r_n&=&a+bn\\
s_n&=&d+(\mu+\nu)a+[c+e+(\mu+\nu)b]n\\
 t_n&=&[(\mu+\nu)c+\nu e]n.
\end{eqnarray}
In addition, $\widetilde{\textbf{M}}(a,b,c,d,e,\mu,\nu)$ satisfies
the following relations
\begin{eqnarray*}
\widetilde{\textbf{M}}(0,0,c,d,e,\mu,\nu)&=&\textbf{L}(c+e,d,(\mu+\nu)c+\nu e)\\
\widetilde{\textbf{M}}(a,b,c,d,e,0,0)&=&\textbf{U}(c+e,d,b,a)\\
\widetilde{\textbf{M}}(a,b,c+\lambda
b,d,e,0,\lambda)&=&\textbf{J}(c+e,d,b,a,\lambda).
\end{eqnarray*}

For the matrix $\widetilde{\textbf{M}}(a,b,c,d,e,\mu,\nu)$, the
following result says that it is binomial-totally positive in seven
variables $a,b,c,d,e,\mu,\nu$.

\begin{prop}\label{prop+binom+TP+M*}
The matrix $\widetilde{\textbf{M}}(a,b,c,d,e,\mu,\nu)$ is
binomial-totally positive in the ring
$\mathbb{Z}[a,b,c,d,e,\mu,\nu]$ equipped with the coefficientwise
order.
\end{prop}

\begin{proof}
By the definition of $\widetilde{\textbf{M}}(a,b,c,d,e,\mu,\nu)$ and
Proposition \ref{prop+binom+M}, we immediately have
\begin{eqnarray}\label{eq+Binom+M}
\textbf{B}^{-1}_q\widetilde{\textbf{M}}(a,b,c,d,e,\mu,\nu)\textbf{B}_q=\widetilde{\textbf{M}}(a,b,c+bq,d,e,\mu,\nu+q).
\end{eqnarray}
So it suffices to prove that
$\widetilde{\textbf{M}}(a,b,c,d,e,\mu,\nu)$ is totally positive in
the ring $\mathbb{Z}[a,b,c,d,e,\mu,\nu]$ equipped with the
coefficientwise order.

 Consider the tridiagonal matrix
\begin{equation} \textbf{M}^{*}(a,b,c,d,e,\mu,\nu)=\left[
\begin{array}{cccccc}
s^{*}_0 & r^{*}_0 &  &  &&\\
t^{*}_1 & s^{*}_1 & r^{*}_1 &&\\
 & t^{*}_2 & s^{*}_2 &r^{*}_2&&\\
 & &\ddots&\ddots&\ddots
\end{array}\right]
\end{equation}
with
\begin{eqnarray}
r^{*}_n&=&(a+bn)(\mu+\nu)\\
s^{*}_n&=&d+(\mu+\nu)a+[c+e+(\mu+\nu)b]n\\
t^{*}_n&=&(c+e)n.
\end{eqnarray}
Obviously, the matrix $\textbf{M}^{*}(a,b,c,d,e,\mu,\nu)$ is
coefficientwise row-diagonally-dominant. So by \cite[Proposition 3.3
(i)]{Zhu21}, it is
 totally positive in the ring
$\mathbb{Z}[a,b,c,d,e,\mu,\nu]$ equipped with the coefficientwise
order. Let
\begin{eqnarray}
r^{**}_n&=&a+bn\\
s^{**}_n&=&d+(\mu+\nu)a+[c+e+(\mu+\nu)b]n\\
 t^{**}_n&=&(c+e)(\mu+\nu)n.
\end{eqnarray} Obviously, $r^{*}_it^{*}_{i+1}=r^{**}_it_{i+1}^{**}$ for all $i$. Then by \cite[Claim 2 in the proof of Theorem 2.1]{Zhu21},
the tridiagonal matrix
\begin{equation} \textbf{M}^{**}(a,b,c,d,e,\mu,\nu)=\left[
\begin{array}{cccccc}
s^{**}_0 & r^{**}_0 &  &  &&\\
t^{**}_1 & s^{**}_1 & r^{**}_1 &&\\
 & t^{**}_2 & s^{**}_2 &r^{**}_2&&\\
 & &\ddots&\ddots&\ddots
\end{array}\right]
\end{equation}
is totally positive in the ring $\mathbb{Z}[a,b,c,d,e,\mu,\nu]$
equipped with the coefficientwise order. For
$\widetilde{\textbf{M}}(a,b,c,d,e,\mu,\nu)$ and
$\textbf{M}^{**}(a,b,c,d,e,\mu,\nu)$, we have
\begin{eqnarray}
r_n&=&r^{**}_n\\
s_n&=&s^{**}_n\\
 t_n&=& t^{**}_n-\mu en.
\end{eqnarray}
Therefore, $0\leq t_n\leq t^{**}_n$ in the coefficientwise sense. In
consequence, by perturbation result \cite[Proposition 3.1]{Zhu21},
$\widetilde{\textbf{M}}(a,b,c,d,e,\mu,\nu)$ is also totally positive
in the ring $\mathbb{Z}[a,b,c,d,e,\mu,\nu]$ equipped with the
coefficientwise order. This completes the proof.
\end{proof}

By Proposition \ref{prop+binom+TP+M*}, if we take (i) $a\rightarrow
af$, $b\rightarrow bf$, $c\rightarrow0$, $\nu\rightarrow0$ and
$\mu\rightarrow1$, (ii) $c\rightarrow0$, $e\rightarrow c$,
$\mu+\nu\rightarrow f$ and $\nu\rightarrow1$, then we have the
following much simpler and cleaner way, respectively.
\begin{cor}
Let $a,b,c,d,e,f$ be elements of a partially ordered commutative
ring $R$.
\begin{itemize}
\item [\rm (i)]
If $a,b,d,e,f\geq0$ and $0\leq c\leq e$, then the matrix $M(e+bf,d +
af,c,bf,af)$ is totally positive.
\item [\rm (ii)]
If $a,b,d,e,f\geq0$ and $0\leq c\leq ef$, then the matrix $M(e+bf,d
+ af,c,b,a)$ is totally positive.
\end{itemize}
\end{cor}

Let $\textbf{f}=(f_n)_n$ and $f(t)=\sum_{n\geq0}f_nt^n$. Define a
matrix $\Lambda:=\left[
  \begin{array}{ccccc}
    0!&&& \\
    &1!&&\\
   & &2!&&\\
 &&&\ddots\\
  \end{array}
\right]$ and a matrix
$\bm{\Gamma}^{(\Lambda,f(t))}:=\Lambda\Gamma(\textbf{f}){\Lambda}^{-1}$.
We obtain the following result.

\begin{prop}\label{prop+ERA+simil}
Let $R$ be a partially ordered commutative ring and $f(t)\in
R[[t]]$. We have
\begin{eqnarray}
\textbf{B}_q^{-1}\bm{\Gamma}^{(\Lambda,f(t))}\textbf{B}_q=\bm{\Gamma}^{(\Lambda,f(t))}.
\end{eqnarray}
Therefore, if $f(t)$ is a P\'olya frequency ogf of order $r$ in
$R[[t]]$, then the matrix $\bm{\Gamma}^{(\Lambda,\textbf{f})}$ is
binomial-totally positive of order $r$ in the ring $R$.
\end{prop}
\begin{proof}
Obviously, if $f(t)$ is a P\'olya frequency ogf of order $r$ in
$R[[t]]$, then $\bm{\Gamma}^{(\Lambda,f(t))}$ is totally positive of
order $r$ in the ring $R$. It follows from
$\textbf{B}_q^{-1}\bm{\Gamma}^{(\Lambda,f(t))}\textbf{B}_q=\bm{\Gamma}^{(\Lambda,f(t))}$
that the matrix $\bm{\Gamma}^{(\Lambda,f(t))}$ is binomial-totally
positive of order $r$ in the ring $R$. In consequence, we only need
to prove
$$\textbf{B}_q^{-1}\bm{\Gamma}^{(\Lambda,f(t))}\textbf{B}_q=\bm{\Gamma}^{(\Lambda,f(t))},$$
which is immediate from the following observing that
\begin{itemize}
\item [\rm (a)]
Any two Toeplitz matrices $\Gamma(\textbf{f})$ and
$\Gamma(\textbf{g})$ commute; hence the matrices
$\Lambda\Gamma(\textbf{f}){\Lambda}^{-1}$ and
$\Lambda\Gamma(\textbf{g}){\Lambda}^{-1}$ commute.
\item [\rm (b)]
$B_q=\Lambda\Gamma(\textbf{h}){\Lambda}^{-1}$, where
$h_n=\frac{q^k}{k!}$.
\end{itemize}
We complete the proof.
\end{proof}

\section{Row-generating
polynomials of generalized $m$-Jacobi--Rogers
triangles}\label{sec+generalized $m$-Jacobi--Rogers triangle}

In this section, we will present coefficientwise Hankel-total
positivity of row-generating polynomials of generalized
$m$-Jacobi--Rogers triangles and combinatorial interpretations for
such row-generating polynomials in terms of trees or forests.

\subsection{Coefficientwise Hankel-total positivity of row-generating
polynomials}

It follows from definitions that the production matrix of the
generalized $m$-Jacobi-Rogers triangle $\textbf{J}^{(m)}$ equals
$P^{(m)}(\bm\beta)$. Thus, by Theorem \ref{thm+main+Hankel-Total
positivity+factor} and results in Section \ref{sec+Bino-TP}, we can
immediately state the following result concerning coefficientwise
Hankel-total positivity of row-generating polynomials of the
generalized $m$-Jacobi--Rogers triangle.

\begin{thm}\label{thm+HTP+Row+GmJR+triangle}
Let $R$ be a partially ordered commutative ring and
$\textbf{J}^{(m)} $ be the infinite generalized $m$-Jacobi--Rogers
triangle. Assume that $P^{(m)}(\bm\beta)$ can be written as the
product of some matrices, such as $\textbf{L}(a_i,b_i,c_i)$,
$\textbf{U}(\widehat{a}_i,\widehat{b}_i,\widehat{u}_i,\widehat{v}_i)$,
$\textbf{J}(\widetilde{a}_i,\widetilde{b}_i,\widetilde{u}_i,\widetilde{v}_i,\lambda_i)$,
$\widetilde{\textbf{M}}(\overline{a}_i,\overline{b}_i,\overline{c}_i,d_i,e_i,\mu_i,\nu_i)$
and $\bm{\Gamma}^{(\Lambda,f_i(t))}$ for $i\geq1$, and denote by
$\bm x$ all indeterminates in $\textbf{L}$, $\textbf{U}$,
$\textbf{J}$ and $\widetilde{\textbf{M}}$. If $f_i(t)$ is a P\'olya
frequency ogf of order $r$ in $R[[t]]$ for each $i\geq1$, then
\begin{itemize}
\item [\rm (i)]
the binomial row-generating matrix $\textbf{J}^{(m)}(q)$ is
coefficientwise totally positive of order $r$ in the ring $R[\bm
x,q]$ equipped with the coefficientwise order;
\item [\rm (ii)]
the row-generating polynomial sequence $(\textbf{J}_n(q))_{n\geq0}$
of $\textbf{J}^{(m)}$ is Hankel-totally positive of order $r$ in the
ring $R[\bm x,q]$ equipped with the coefficientwise order.
\end{itemize}
\end{thm}

In particular, for the $m$-Stieltjes--Rogers triangle
$\textbf{S}^{(m)}$, we have the following result.
\begin{thm}\label{thm+HTP+Row+GF+Stie-Roger+triangle}
Let $S_n(\bm\alpha;q)$ be the $n$th row-generating polynomial of the
$m$-Stieltjes--Rogers triangle $\textbf{S}^{(m)}$. If
$\balpha=(y,\underbrace{x_1,\ldots,x_{m}}_{m},y+x_0,\underbrace{2x_1,\ldots,2x_{m}}_{m},y+2x_0,\underbrace{3x_1,\ldots,3x_{m}}_{m},\ldots),$
then $(S_n(\bm\alpha;q))_{n\geq0}$ is Hankel-totally positive in the
polynomial ring $\mathbb{Z}[\textbf{x},y,q]$ equipped with the
coefficientwise order.
\end{thm}

\begin{proof}
Note by \cite[Proposition 8.2 (b)]{PSZ18}, for
$$(\alpha_{i})_{i\geq
m}=(y,\underbrace{x_1,\ldots,x_{m}}_{m},y+x_0,\underbrace{2x_1,\ldots,2x_{m}}_{m},y+2x_0,\underbrace{3x_1,\ldots,3x_{m}}_{m},\ldots)$$
that the production matrix of the $m$-Stieltjes-Rogers triangle
$\textbf{S}^{(m)}$ is
\begin{eqnarray}\label{production+matrix+Stieltjes-Rogers}
\left(\prod_{i=1}^{m}\left[\begin{array}{ccccc}
1&&&&\\
x_i&1&&&\\
&2x_i&1&&\\
&&3x_i&1&\\
&&&\ddots&\ddots\\
\end{array}\right]\right)\left[\begin{array}{cccccc}
y&1&&&\\
&x_0+y&1&&\\
&&2x_0+y&1&\\
&&&3x_0+y&1&\\
&&&&\ddots&\ddots
\end{array}\right],
\end{eqnarray}
where the product over $i$ of these matrices is to be understood as
left-to-right (i.e. $i = 1$ on the far left, $i = m$ on the far
right). It is easy to observe that the matrices over $i$ in the
product (\ref{production+matrix+Stieltjes-Rogers}) are $L(0; 1;
x_i)$, and the final matrix is $U(x_0; y; 0; 1)$. Therefore, it
follows from Theorem \ref{thm+HTP+Row+GmJR+triangle} that
$(S_n(\bm\alpha;q))_{n\geq0}$ is Hankel-totally positive in the
polynomial ring $\mathbb{Z}[\textbf{x},y,q]$ equipped with the
coefficientwise order.
\end{proof}

 The following result concerning
coefficientwise Hankel-total positivity of polynomials is very
useful.
\begin{prop}\label{prop+recp}
Let $A_n(q)$ be a polynomial of degree $n$ and define a new
polynomial
\begin{eqnarray}\widehat{A}_n(a,b,c,d,q)&=&(a+bq)^nA_{n}\left(\frac{c+dq}{a+bq}\right)
\end{eqnarray}
for $n\geq0$. If $(A_n(q))_{n\geq0}$ is coefficientwise
Hankel-totally positive of order $r$ in $q$, then
$(\widehat{A}_n(q))_{n\geq0}$ is coefficientwise Hankel-totally
positive of order $r$ in $(a,b,c,d,q)$.
\end{prop}
\begin{proof}
For $k\leq r$, let $\mathcal
{A}^{i_1,i_2,\ldots,i_k}_{j_1,j_2,\ldots,j_k}(q)$ (resp. $\mathcal
{\widehat{A}}^{i_1,i_2,\ldots,i_k}_{j_1,j_2,\ldots,j_k}(a,b,c,d,q)$)
be the minor of $[A_{i+j}(q)]_{i,j}$ (resp.
$[\widehat{A}_{i+j}(q)]_{i,j}$) by taking its rows
${i_1,i_2,\ldots,i_k}$ and columns ${j_1,j_2,\ldots,j_k}$. In terms
of the assumption that $[A_{i+j}(q)]_{i,j}$ is coefficientwise
totally positive of order $r$ in $q$, $\mathcal
{A}^{i_1,i_2,\ldots,i_k}_{j_1,j_2,\ldots,j_k}(q)$ is a polynomial
with nonnegative coefficients of degree $\leq
i_1+i_2+\cdots+i_k+j_1+j_2+\cdots+j_k$. Since
$\widehat{A}_{i+j}(q)=(a+bq)^{i+j}A_{i+j}(\frac{c+dq}{a+bq})$,
taking out $(a+bq)^{i_1},(a+bq)^{i_2},\ldots,(a+bq)^{i_k}$ from rows
and $(a+bq)^{j_1},(a+bq)^{j_2},\ldots,(a+bq)^{j_k}$ from columns of
$\mathcal
{\widehat{A}}^{i_1,i_2,\ldots,i_k}_{j_1,j_2,\ldots,j_k}(a,b,c,d,q)$,
respectively, we get
\begin{eqnarray}
\mathcal
{\widehat{A}}^{i_1,i_2,\ldots,i_k}_{j_1,j_2,\ldots,j_k}(a,b,c,d,q)&=&(a+bq)^{i_1+i_2+\cdots+i_k+j_1+j_2+\cdots+j_k}\mathcal
{A}^{i_1,i_2,\ldots,i_k}_{j_1,j_2,\ldots,j_k}\left(\frac{c+dq}{a+bq}\right).
\end{eqnarray}
In consequence, $\mathcal
{\widehat{A}}^{i_1,i_2,\ldots,i_k}_{j_1,j_2,\ldots,j_k}(a,b,c,d,q)$
is a multivariate polynomial with nonnegative coefficients in
$a,b,c,d,q$. This completes the proof.
\end{proof}

For a polynomial $A_n(q)$ with $\deg(A_n(q))=n$, its
\textbf{reversed polynomial} $A_n^{*}(q)$ is defined by
$A_n^{*}(q)=q^nA_n(\frac{1}{q})$. The following result is immediate
from Proposition \ref{prop+recp}.

\begin{cor}\label{cor+HTP+reversed}
If $(A_n(q))_{n\geq0}$ is coefficientwise Hankel-totally positive of
order $r$ in $(\textbf{x},q)$, then so is the sequence
$(A^{*}_n(q))_{n}$ of reversed polynomials.
\end{cor}

This implies in particular the coefficientwise Hankel-total
positivity of $(S^{*}_n(\bm\alpha;q))_{n\geq0}$ and
$(\textbf{J}^{*}_n(q))_{n\geq0}$ in Theorem
\ref{thm+HTP+Row+GF+Stie-Roger+triangle} and Theorem
\ref{thm+HTP+Row+GmJR+triangle}, respectively.

Coefficientwise Hankel-total positivity of polynomials also plays an
important role in triangular convolutions preserving the Stieltjes
moment properties. Let $\textbf{A}=[a_{n,k}]_{n,k\ge 0}$ be an
infinite matrix. Define the $\textbf{A}$-convolution
\begin{eqnarray}\label{a-c}
z_n&=&\sum_{k=0}^{n}a_{nk}x_ky_{n-k}
\end{eqnarray}
for $n\geq0$ We say that \eqref{a-c} preserves the Stieltjes moment
property: if both $(x_n)_{n\ge 0}$ and $(y_n)_{n\ge 0}$ are
Stieltjes moment sequences, then so is $(z_n)_{n\ge 0}$. P\'olya and
Szeg\"o~\cite[Part VII, Theorem 42]{PS64} proved for $n\geq0$ that
the binomial convolution
\begin{eqnarray}
z_n&=&\sum_{k=0}^{n}\binom{n}{k}x_ky_{n-k}
\end{eqnarray}
preserves the Stieltjes moment property in $\mathbb{R}$. Recently,
it was proved in \cite{WZ16} that the next sufficient condition for
the triangular convolution preserving the Stieltjes moment property
\cite{WZ16}.

\begin{thm}\label{thm+conv}\emph{\cite{WZ16}}
Let $A_n(q)=\sum_{k=0}^{n}a_{n,k}q^k$ be the $n$th row-generating
function of the matrix $A=[a_{n,k}]_{n,k}$. Assume that
$(A_n(q))_{n\ge 0}$ is a Stieltjes moment sequence for any fixed
$q\ge 0$. Then the $\textbf{A}$-convolution \eqref{a-c} preserves
the Stieltjes moment property in $\mathbb{R}$.
\end{thm}
Clearly, if the sequence $(A_n(q))_{n\geq0}$ is coefficientwise
Hankel-totally positive in $q$, then $(A_n(q))_{n\ge 0}$ is a
Stieltjes moment sequence for any fixed $q\ge 0$. In addition, it
also implies $3$-$q$-log-convexity of $(A_n(q))_{n\geq0}$ in terms
of the next result.

\begin{prop}\emph{\cite{Zhu21}}\label{prop+3-q-log-convex}
For a polynomial sequence $(A_n(\textbf{x}))_{n\geq0}$, if the
Hankel matrix $[A_{i+j}(\textbf{x})]_{i,j\geq0}$ is coefficientwise
totally positive of order $4$ in $\textbf{x}$, then
$(A_n(\textbf{x}))_{n\geq0}$ is $3$-$\textbf{x}$-log-convex.
\end{prop}
In consequence, by Theorems
\ref{thm+HTP+Row+GF+Stie-Roger+triangle}, \ref{thm+conv} and
Proposition \ref{prop+3-q-log-convex}, we obtain:

\begin{thm}\label{thm+triangle+transformation+3-q+Stie-Roger}
Let $S_n(\bm\alpha;q)$ be the $n$th row-generating polynomial of the
$m$-Stieltjes--Rogers triangle $\textbf{S}^{(m)}$. If
$\balpha=(y,\underbrace{x_1,\ldots,x_{m}}_{m},y+x_0,\underbrace{2x_1,\ldots,2x_{m}}_{m},y+2x_0,\underbrace{3x_1,\ldots,3x_{m}}_{m},\ldots),$
then we have
\begin{itemize}
\item [\rm (i)]
$(S_n(\bm\alpha;q))_{n\geq0}$ is $3$-$(\textbf{x},y,q)$-log-convex;
\item [\rm (ii)]
the triangular convolution
$z_n=\sum_{k=0}^{n}S^{(m)}_{n,k}(\bm\alpha)s_kt_{n-k}$
 preserves the Stieltjes moment property in $\mathbb{R}$ for all $\textbf{x}\geq0$ and $y\geq0$.
\end{itemize}
\end{thm}
Similarly, by Theorems \ref{thm+HTP+Row+GmJR+triangle},
\ref{thm+conv} and Proposition \ref{prop+3-q-log-convex}, we obtain:

\begin{thm}\label{thm+TP+triangle+transformation+3-q}
Let $\textbf{J}^{(m)}=[J^{(m)}_{n,k}(\bm\beta)]_{n,k}$ be the
infinite generalized $m$-Jacobi--Rogers triangle. Assume that
$P^{(m)}(\bm\beta)$ can be written as the product of some matrices,
such as $\textbf{L}(a_i,b_i,c_i)$,
$\textbf{U}(\widehat{a}_i,\widehat{b}_i,\widehat{u}_i,\widehat{v}_i)$,
$\textbf{J}(\widetilde{a}_i,\widetilde{b}_i,\widetilde{u}_i,\widetilde{v}_i,\lambda_i)$,
$\widetilde{\textbf{M}}(\overline{a}_i,\overline{b}_i,\overline{c}_i,d_i,e_i,\mu_i,\nu_i)$
and $\bm{\Gamma}^{(\Lambda,f_i(t))}$ for $i\geq1$, and denote by
$\bm x$ all indeterminates in $\textbf{L}$, $\textbf{U}$,
$\textbf{J}$ and $\widetilde{\textbf{M}}$. If $f_i(t)$ is a P\'olya
frequency ogf in $\mathbb{R}[[t]]$ for each $i\geq1$, then
\begin{itemize}
\item [\rm (i)]
$(\textbf{J}_n(q))_{n\geq0}$ is $3$-$(\bm x,q)$-log-convex;
\item [\rm (ii)]
the triangular convolution
$z_n=\sum_{k=0}^{n}J^{(m)}_{n,k}(\bm\beta)x_ky_{n-k}$
 preserves the Stieltjes moment property in $\mathbb{R}$ for all indeterminates $\bm x\geq0$.
\end{itemize}
\end{thm}


\subsection{Combinatorial interpretations for row-generating
polynomials}

Recall \cite[pp.~294--295]{Stanley_86} that an {\em ordered tree}\/
(also called {\em plane tree}\/) is an (unlabeled) rooted tree in
which the children of each vertex are linearly ordered. An {\em
ordered forest of ordered trees}\/ (also called {\em plane
forest}\/) is a linearly ordered collection of ordered trees. Using
a well-known bijection from the set of ordered forests of ordered
trees with $n+1$ total vertices and $k+1$ components onto the set of
partial \L{}ukasiewicz paths from $(0,0)$ to $(n,k)$ (see
 \cite[pp.~30--36]{Stanley_99},
   \cite[Chapter~11]{Lothaire_97},
   \cite[section~6.2]{Pitman_06}
   and \cite[proof of Proposition~3.1]{Hackl_18} for instance), in \cite{PSZ18}, the
$m$-Jacobi--Rogers polynomials $J_n^{(m)}(\bm\beta)$ are interpreted
as the generating polynomials for certain classes of ordered trees,
and the $m$-Jacobi--Rogers triangle
$[J_{n,k}^{(m)}(\bm\beta)]_{n,k}$ are interpreted as the generating
polynomials for certain classes of ordered forests of ordered trees.

\begin{prop}\emph{\cite[Proposition 6.1]{PSZ18}} \label{prop.Jnk.forests}
Let $\bm\beta = (\beta_i^{(\ell)})_{i \ge \ell \ge -1}$ be
indeterminates with all $\beta_i^{(-1)}=1$. Then the
$\infty$-Jacobi--Rogers triangle $J^{(\infty)}_{n,k}(\bm\beta)$ is
the generating polynomial for ordered forests of ordered trees with
$n+1$ total vertices and $k+1$ components in which each vertex at
level $L$ with $c$ children gets weight 1 if it is a leaf ($c=0$)
and weight $\beta_{L+c-1}^{(c-1)}$ otherwise.
\end{prop}

\begin{rem}
 The $m$-Jacobi--Rogers polynomials with $m < \infty$
correspond to the forests in which each vertex has at most $m+1$
children.
\end{rem}

For the $m$-Stieltjes--Rogers triangle $\textbf{S}^m$ in Theorem
\ref{thm+HTP+Row+GF+Stie-Roger+triangle} and the $m$-Jacobi--Rogers
triangle $\textbf{J}^m$ in Theorem \ref{thm+HTP+Row+GmJR+triangle},
obviously, their corresponding production matrix
$\textbf{P}=\prod_{i\geq1}\textbf{P}_i$ is an $m$-banded
lower-Hessenberg matrix. In consequence, by Propositions
\ref{prop+low-bi+TP}, \ref{prop+tri+J+TP}, \ref{prop+binom+TP+M*}
and \ref{prop+ERA+simil}, the production matrix of the binomial
row-generating matrix:
\begin{eqnarray}
\textbf{B}^{-1}_q\textbf{P}\textbf{B}_q&=&\textbf{B}^{-1}_q\left(\prod_{i\geq1}\textbf{P}_i\right)\textbf{B}_q=\prod_{i\geq1}\textbf{B}^{-1}_q\textbf{P}_i\textbf{B}_q
\end{eqnarray}
is a $r$-banded lower-Hessenberg matrix for $r=m$ or $r=m+1$. Thus
for the $m$-Stieltjes--Rogers triangle $\textbf{S}^m$ and the
$m$-Jacobi--Rogers triangle $\textbf{J}^m$, their row-generating
polynomials are both $r$-Jacobi--Rogers polynomials with $r=m$ or
$r=m+1$. Therefore we can interpreter them in terms of ordered
forests of ordered trees.

Furthermore, for $x_0=0$ in Theorem
\ref{thm+HTP+Row+GF+Stie-Roger+triangle}, i.e.,
$$\bm\alpha=(y,\underbrace{x_1,\ldots,x_{m}}_{m},y,\underbrace{2x_1,\ldots,2x_{m}}_{m},y,\underbrace{3x_1,\ldots,3x_{m}}_{m},\ldots),$$
then, by (\ref{production+matrix+Stieltjes-Rogers}), we have
\begin{eqnarray}
\textbf{P}=\left(\prod_{i=1}^{m}\left[\begin{array}{ccccc}
1&&&&\\
x_i&1&&&\\
&2x_i&1&&\\
&&3x_i&1&\\
&&&\ddots&\ddots\\
\end{array}\right]\right)\left[\begin{array}{cccccc}
q+y&1&&&\\
&q+y&1&&\\
&&q+y&1&\\
&&&q+y&1&\\
&&&&\ddots&\ddots
\end{array}\right].
\end{eqnarray}
In consequence, by \cite[Proposition 8.2 (b)]{PSZ18}, the
row-generating polynomial sequence $(S_n(\bm\alpha;q))_{n\geq0}$ is
still an $m$-Stieltjes--Rogers polynomial $S^{(m)}_n(\balpha')$ with
$$\bm\alpha'=(q+y,\underbrace{x_1,\ldots,x_{m}}_{m},q+y,\underbrace{2x_1,\ldots,2x_{m}}_{m},q+y,\underbrace{3x_1,\ldots,3x_{m}}_{m},\ldots).$$

For $y=x_0$ in Theorem \ref{thm+HTP+Row+GF+Stie-Roger+triangle},
i.e.,
$$\bm\alpha=(\underbrace{x_0,x_1,\ldots,x_{m}}_{m+1},\underbrace{2x_0,2x_1,\ldots,2x_{m}}_{m+1},\underbrace{3x_0,3x_1,\ldots,3x_{m}}_{m+1},\ldots),$$
then the $m$-Stieltjes--Rogers polynomial $S^{(m)}_{n,0}(\bm\alpha)$
can be interpreted in terms of increasing $(m+1)$-ary trees and is
called the \textbf{multivariate Eulerian polynomial} \cite[Section
12.2]{PSZ18}.

\section{Row-generating polynomials of
Catalan-Stieltjes matrices}\label{sec+row-Cat-Stie}

In this section, we will apply our results for the generalized
$m$-Jacobi--Rogers triangle $\textbf{J}^{(m)} = \big(
J^{(m)}_{n,k}(\bm\beta) \big)_{n,k \ge 0}$ to tridiagonal matrices.

For $P^{(m)}(\bm\beta)$ with $m=1$, we use the notation
 $\beta_i = \beta_i^{(1)}$, $\gamma_i = \beta_i^{(0)}$ and $\delta_i=\beta_i^{(-1)}$ for $i\geq0$ and also write $J^{(1)}_{n,k}(\bm\beta)$ as $J_{n,k}$ for short.
So $P^{(m)}(\bm\beta)$ reduces to a tridiagonal matrix denoted by
\begin{equation} \textbf{P}=\left[
\begin{array}{cccccc}
\gamma_0 & \delta_0 &  &  &&\\
\beta_1 & \gamma_1 & \delta_1 &&\\
 & \beta_2 & \gamma_2 &\delta_2&&\\
& & \beta_3&\gamma_3 & \delta_3 &\\
& & &\ddots&\ddots&\ddots
\end{array}\right].
\end{equation}
Meanwhile we observe that the output matrix
$\scro(\textbf{P})=\textbf{J}^{(1)}$ satisfies the recurrence
relation:
 \begin{equation}\label{recurr+PJS}
J_{n,k}=\delta_{k-1}J_{n-1,k-1}+\gamma_kJ_{n-1,k}+\beta_{k+1}J_{n-1,k+1},
 \end{equation}
with initial conditions $J_{n,k}=0$ unless $0\le k\le n$ and
$J_{0,0}=1$. The generalized $1$-Jacobi--Rogers triangle
$\textbf{J}^{(1)}$ is often called the {\it
\textbf{Catalan-Stieltjes triangle}} \cite{Aig01,PZ16}. In
particular, for all $\delta_i=1$, $\textbf{J}^{(1)}$ is called the
{\it \textbf{Stieltjes matrix}} \cite{Fla80}. The
$\textbf{x}$-log-convexity, $\textbf{x}$-log-convexity of higher
order and coefficientwise Hankel-total positivity of
$(J_{n,0})_{n\geq0}$ have been studied, see
\cite{CLW15,MMW17,PZ16,Sok,WZ16,Zhu13,Zhu17,Zhu19,Zhu21} for
instance. But the following question is open.
\begin{ques}\label{ques+tri}
For the Catalan-Stieltjes matrix defined by (\ref{recurr+PJS}), let
$J_n(q)=\sum_kJ_{n,k}q^k$ be its n$th$ row-generating polynomial.
Find conditions ensuring coefficientwise Hankel-total positivity of
the sequence $(J_n(q))_{n\geq0}$.
\end{ques}

In general, for Question \ref{ques+tri}, there exists the cases such
that $J_n(q)$ is not coefficientwise Hankel-totally positive: see
e.g. Example 2.2. For the Catalan-Stieltjes triangle defined by
(\ref{recurr+PJS}), let all $\delta_k$, $\gamma_k$ and $\beta_{k}$
be linear in $k$. In such case, if the production matrix
$\textbf{P}$ equals one of $\textbf{L}(a,b,c)$,
$\textbf{U}(\widehat{a},\widehat{b},\widehat{u},\widehat{v})$,
$\textbf{J}(\widetilde{a},\widetilde{b},\widetilde{u},\widetilde{v},\lambda)$,
and
$\widetilde{\textbf{M}}(\overline{a},\overline{b},\overline{c},d,e,\mu,\nu)$,
then, by Theorem \ref{thm+HTP+Row+GmJR+triangle}, the row-generating
polynomial sequence $(J_{n}(q))_{n\geq0}$ is Hankel-totally positive
in the polynomial ring $\mathbb{Z}[a,b,c,d,e,\mu,\nu,q]$ equipped
with the coefficientwise order.

\subsection{Main results for Catalan-Stieltjes matrices}

In this subsection, we consider nonlinear cases for $\delta_k$,
$\gamma_k$ and $\beta_{k}$, i.e.,  $\delta_k$, $\gamma_k$ and
$\beta_{k}$ are nonlinear functions in $k$. We will give the first
result for Question \ref{ques+tri} as follows.

\begin{thm}\label{thm+TP+tridi+even}
Let $\textbf{J}=[J_{n,k}]_{n,k}$ be a Catalan-Stieltjes triangle
defined by (\ref{recurr+PJS}) with $\delta_k=(ak+b)(uk+v)$,
$\gamma_k=[(ax+cu)k^2+(ay+bx+cv-cu)k+by]$ and
$\beta_{k+1}=c(y+kx)(k+1)$. Then we have
 \begin{itemize}
 \item [\rm (i)]
 the binomial row-generating matrix $\textbf{J}(q)$ is totally positive in the polynomial ring $\mathbb{Z}[a,b,c,u,v,x,y,q]$ equipped with the
coefficientwise order;
 \item [\rm (ii)]
both $(J_{n}(q))_{n\geq0}$ and its reversed polynomial sequence
$(J^{*}_{n}(q))_{n\geq0}$ are Hankel-totally positive in the
polynomial ring $\mathbb{Z}[a,b,c,u,v,x,y,q]$ equipped with the
coefficientwise order and $3$-$(a,b,c,u,v,x,y,q)$-log-convex;
\item [\rm (iii)]
the sequence $(J_{n,0})_{n\geq0}$ is Hankel-totally positive in the
ring $\mathbb{Z}[a,b,c,u,v,x,y,q]$ equipped with the coefficientwise
order and $3$-$(a,b,c,u,v,x,y,q)$-log-convex;
 \item [\rm (iv)]
the convolution $r_n=\sum_{k\geq0}J_{n,k}s_kt_{n-k}$ preserves the
Stieltjes moment property in $\mathbb{R}$ for all
$a,b,c,u,v,x,y\geq0$.
 \end{itemize}
\end{thm}

\begin{proof}
For this Catalan-Stieltjes triangle with $\delta_k=(ak+b)(uk+v)$,
$\gamma_k=[(ax+cu)k^2+(ay+bx+cv-cu)k+by]$ and
$\beta_{k+1}=c(y+kx)(k+1)$, we have its production matrix
 \begin{eqnarray}\label{factor+tri}
\textbf{P}&=&\left[
\begin{array}{ccccccc}
\gamma_0 & \delta_0 &  &  &\\
\beta_1 & \gamma_1& \delta_1 &\\
 & \beta_2  & \gamma_3&\delta_2&\\
 & &\beta_3  & \gamma_4&\delta_3&\\
& && \ddots&\ddots & \ddots \\
\end{array}\right]=\textbf{L}(a,b,c)\textbf{U}(x,y,u,v).
\end{eqnarray}
Hence, (i), (ii) and (iv) are immediate from Theorem
\ref{thm+HTP+Row+GmJR+triangle}, Corollary \ref{cor+HTP+reversed}
and Theorem \ref{thm+TP+triangle+transformation+3-q}, respectively.
For $q=0$, (ii) implies (iii) because $J_{n}(0)=J_{n,0}$.
\end{proof}

If we change the order of $\textbf{L}(a,b,c)$ and
$\textbf{U}(x,y,u,v)$ in (\ref{factor+tri}), we immediately obtain
the following result for Question \ref{ques+tri}.

\begin{thm}\label{thm+TP+tridi+odd}
Let $\textbf{J}=[J_{n,k}]_{n,k}$ be a Catalan-Stieltjes triangle
defined by (\ref{recurr+PJS}) with $\delta_k=(ak+a+b)(ku+v)$,
$\gamma_k=(ax+cu)k^2+(ay+bx+cv+cu)k+by+cv$ and
$\beta_{k+1}=c[y+x(k+1)](k+1)$. Then we have conclusions (i)-(iv) in
Theorem \ref{thm+TP+tridi+even}.
\end{thm}

\begin{rem}
We want to point out that the $1$-Jacobi--Rogers triangle has the
restricted condition: all $\beta_i^{(-1)}=1$. Then such condition
enforces $a=u=0$ for the Catalan-Stieltjes triangle in Theorems
\ref{thm+TP+tridi+even} and \ref{thm+TP+tridi+odd}. Meanwhile,
$\gamma_k$ must be a linear function in $k$. Such examples for the
Catalan-Stieltjes triangles are restricted. So we see the good for
the generalized $m$-Jacobi--Rogers triangle $\textbf{J}^{(m)}$
without the restriction $\beta_i^{(-1)}=1$.
\end{rem}

For the Catalan-Stieltjes triangle $\textbf{J}=[J_{n,k}]_{n,k}$
defined by (\ref{recurr+PJS}), let $\widetilde{J}_{n,k}:=J_{n,k}k!$.
Then by (\ref{recurr+PJS}), the associated triangle
$[\widetilde{J}_{n,k}]_{n,k\geq0}$ satisfies the recurrence
relation:
 \begin{equation}
\widetilde{J}_{n,k}=k\delta_{k-1}\widetilde{J}_{n-1,k-1}+\gamma_k\widetilde{J}_{n-1,k}+\frac{\beta_{k+1}}{k+1}\widetilde{J}_{n-1,k+1},
 \end{equation}
with initial conditions $\widetilde{J}_{n,k}=0$ unless $0\le k\le n$
and $\widetilde{J}_{0,0}=1$. From the view point of production
matrices, that is to say
$\widetilde{\textbf{J}}=\textbf{J}\Lambda=\scro(\Lambda^{-1}\textbf{P}\Lambda)$,
where $\Lambda=\diag(k!).$ Thus, some Catalan-Stieltjes triangles
$\textbf{J}$ may not satisfy the condition of Theorems
\ref{thm+TP+tridi+even} or \ref{thm+TP+tridi+odd}, but their
associated matrices $\widetilde{\textbf{J}}$ do. So we can still
obtain some positivity results from $\widetilde{\textbf{J}}$.

For the Catalan-Stieltjes triangle in the following result, we have
the same conclusions in (i) and (iii) of Theorem
\ref{thm+TP+tridi+even}, but we want to list and emphasize the
similar results (ii) and (iv).

\begin{prop}\label{prop+row+k!}
Let $\textbf{J}=[J_{n,k}]_{n,k}$ be a Catalan-Stieltjes triangle
defined by (\ref{recurr+PJS}) with $\delta_k=(ak+b)u$,
$\gamma_k=[(ax+cu)k^2+(ay+bx)k+by]$ and
$\beta_{k+1}=c(y+kx)(k+1)^2$. Let
$\widetilde{J}_n(q)=\sum_{k=0}^nJ_{n,k}k!q^k$. Then we have
 \begin{itemize}
 \item [\rm (ii)]
both $(\widetilde{J}_{n}(q))_{n\geq0}$ and
$(\widetilde{J}^{*}_{n}(q))_{n\geq0}$ are Hankel-totally positive in
the polynomial ring $\mathbb{Z}[a,b,c,u,x,y,q]$ equipped with the
coefficientwise order and $3$-$(a,b,c,u,x,y,q)$-log-convex;
 \item [\rm (iv)]
the convolution $r_n=\sum_{k\geq0}J_{n,k}k!s_kt_{n-k}$ preserves the
Stieltjes moment property in $\mathbb{R}$ for all
$a,b,c,u,x,y\geq0$.
 \end{itemize}
\end{prop}

Similar to Proposition \ref{prop+row+k!}, we also obtain:
\begin{prop}\label{prop+row+k!+odd}
Let $\textbf{J}=[J_{n,k}]_{n,k}$ be a Catalan-Stieltjes triangle
defined by (\ref{recurr+PJS}) with $\delta_k=(ak+a+b)u$,
$\gamma_k=(ax+cu)k^2+(ay+bx+2cu)k+by+cu$ and
$\beta_{k+1}=c[y+x(k+1)](k+1)^2$. Then we have the same conclusions
(i)-(iv) in Proposition \ref{prop+row+k!}.
\end{prop}

\subsection{A generalized Jacobi-Stirling triangle}

The Jacobi-Stirling numbers $\JS_n^k( z)$ of the second kind satisfy
the following recurrence relation:
\begin{equation}
\left\{ \begin{array}{l} \JS_0^0(z)=1, \qquad \JS_n^k(z)=0, \quad  \text{if} \ k \not\in\{1,\ldots,n\}, \\
\JS_n^k(z)= \JS_{n-1}^{k-1}(z)+k(k+z)\,\JS_{n-1}^{k}(z),  \quad n,k
\geq 1,\end{array} \right.
 \end{equation}
 where $z=\alpha+\beta+1$ and $\alpha,\beta\geq-1$. They were introduced in \cite{EKLWY07} and are the coefficients of the
integral composite powers of the Jacobi differential operator
\begin{eqnarray}
\ell_{\alpha,\beta}[y](t)&=&\frac{1}{(1-t)^\alpha(1+t)^\beta}\left(
-(1-t)^{\alpha+1}(1+t)^{\beta+1} y'(t)\right)' .
\end{eqnarray}
For $\alpha =\beta=0$, the Jacobi-Stirling numbers reduce to the
Legendre-Stirling numbers of the second kind. It was found that the
Jacobi-Stirling numbers and Legendre-Stirling numbers have many
similar properties with the classical Stirling numbers. In fact,
there has been an extensive literature in recent years on
Jacobi-Stirling numbers and the Legendre-Stirling numbers,
see~Andrews {\it et al.}~\cite{AEGL11,AGL11,AL09}, Egge~\cite{Eg10},
Everitt {\it et al.}~\cite{ELW02,EKLWY07}, Gelineau and
Zeng~\cite{GZ10}, Mongelli~\cite{Monge12} and Zhu
\cite{Zhu14,Zhu182} for details.

Let $U(n,k)$ be the \emph{central factorial numbers}
\cite[A036969]{Slo}, which first occurred in MacMahon \cite[pp.
106]{Mac1921} and were also defined in Riordan's book \cite[pp.
213-217]{Rio68} by
\begin{align}
x^{2n}=\sum_{k=0}^{2n}U(n,k)\,x\prod_{i=1}^{2k-1}\left(x+k-i\right).
\end{align}
They satisfy the following recurrence
\begin{eqnarray}
U(n,k)&=&U(n-1,k-1)+k^2U(n-1,k)
 \end{eqnarray}
with $U(0,0)=1$. Obviously, the central factorial numbers are simply
the Jacobi-Stirling numbers specialized to $z=0$. In addition, the
central factorial numbers have a close connection with the famous
Genocchi numbers $G_n$ by
\begin{eqnarray}
G_{2n+2}&=&\sum_{k=0}^n(-1)^{k+1}U(n,k)(k!)^2
\end{eqnarray}
\cite{Dum74} and also play an important role in the proof of one
inequality concerning rank and crank moments related to Andrews'
spt-function \cite{Gar11}.

Motivated by Jacobi-Stirling numbers, Legendre-Stirling numbers, and
central factorial numbers, we introduce a \textbf{generalized
Jacobi-Stirling triangle} $[\textbf{JS}_{n,k}]_{0\le k\le n}$
defined by
\begin{eqnarray}\label{rr-generalized Jacobi-Stirling}
\textbf{JS}_{n,k}&=&\textbf{JS}_{n-1,k-1}+(xk^2+yk)\textbf{JS}_{n-1,k}
\end{eqnarray}
where $\textbf{JS}_{n,k}=0$ unless $0\le k\le n$ and
$\textbf{JS}_{0,0}=1$. Obviously, the generalized Jacobi-Stirling
number $\textbf{JS}_{n,k}$ is a common generalization of the
Jacobi-Stirling number, the Legendre-Stirling number, and the
central factorial number. From \cite{Zhu14,Zhu182}, we know that the
row-generating polynomials of $[\textbf{JS}_{n,k}]_{0\le k\le n}$
are $q$-log-convex. In 2018, the author conjectured that the
row-generating polynomials of $[\textbf{JS}_{n,k}]_{0\le k\le n}$
are Hankel-totally positive in the polynomial ring
$\mathbb{Z}[x,y,q]$ equipped with the coefficientwise order. Now, we
will obtain a weaker result in the following. Let
$\widetilde{\textbf{JS}}_{n,k}=\textbf{JS}_{n,k}k!$ for $n\geq
k\geq0$. It follows from (\ref{rr-generalized Jacobi-Stirling}) that
the associated triangle $\widetilde{\textbf{JS}}$ satisfies the
recurrence relation
\begin{eqnarray}
\widetilde{\textbf{JS}}_{n,k}&=&k
\widetilde{\textbf{JS}}_{n-1,k-1}+(xk^2+yk)\widetilde{\textbf{JS}}_{n-1,k}
\end{eqnarray}
with the initial condition $\widetilde{\textbf{JS}}_{0,0}=1$.
Obviously, the associated triangle $\widetilde{\textbf{JS}}$ is the
Catalan-Stieltjes triangle in Theorem \ref{thm+TP+tridi+odd}
specialized to $a=v=1$ and $b=c=u=0$. Therefore, applying Theorem
\ref{thm+TP+tridi+odd} to the associated triangle
$\widetilde{\textbf{JS}}$, we immediately obtain:
\begin{prop}
Let $[\textbf{JS}_{n,k}]_{0\le k\le n}$ be the generalized
Jacobi-Stirling triangle defined by (\ref{rr-generalized
Jacobi-Stirling}) and
$\mathcal{\widetilde{J}}_{n}(q)=\sum_{k=0}^n\textbf{JS}_{n,k}k!q^k$.
Then we have
\begin{itemize}
 \item [\rm (i)]
 the binomial row-generating matrices $\textbf{JS}(q)$ and $\widetilde{\textbf{JS}}(q)$ are totally positive in the polynomial ring
 $\mathbb{Z}[x,y,q]$ equipped with the
coefficientwise order;
 \item [\rm (ii)]
both $(\mathcal{\widetilde{J}}_{n}(q))_{n\geq0}$ and
$(\mathcal{\widetilde{J}^{*}}_{n}(q))_{n\geq0}$ are Hankel-totally
positive in the polynomial ring
 $\mathbb{Z}[x,y,q]$ equipped with the
coefficientwise order and $3$-$(x,y,q)$-log-convex;
\item [\rm (iii)]
the convolution $r_n=\sum_{k\geq0}\textbf{JS}_{n,k}k!s_kt_{n-k}$
preserves the Stieltjes moment property in $\mathbb{R}$ for both
$x,y\geq0$.
 \end{itemize}
\end{prop}

\subsection{A generalized elliptic polynomial}
The Jacobi elliptic functions $\sn(u,\alpha)$, $\cn(u,\alpha)$ and
$\dn(u,\alpha)$ appear in different fields (i.e., real and complex
analysis, number theory, combinatorics, physics and so on) and have
been extensively studied (see \cite{CF05,Dum81,FF89,R07} for
instance).

For a fixed modulus $\alpha$, $\sn(u,\alpha)$ is defined as the
inverse of an elliptic integral:
\begin{eqnarray}
\sn(u,\alpha)&=&y\quad \text{iff}\quad u=\int_0^{y}\frac{d {\it
t}}{\sqrt{(1-t^2)(1-{\alpha}^2\sin^2 t})}.
\end{eqnarray}
The elliptic functions $\cn(u,\alpha)$ and $\dn(u,\alpha)$ are
defined by
\begin{eqnarray}
\cn(u,\alpha)=\cos am(u,\alpha);
\quad \dn(u,\alpha)=\sqrt{1-{\alpha}^2\sin^2 am(u,\alpha)},
\end{eqnarray}
 where
$am(u,\alpha)$ is the inverse of an elliptic integral: by definition
\begin{eqnarray}
am(u,\alpha)=\phi\quad \text{iff}\quad u=\int_0^{\phi}\frac{d {\it t}}{\sqrt{1-{\alpha}^2\sin^2
t}}.
\end{eqnarray}
The functions $\cn(u,\alpha)$ and $\dn(u,\alpha)$ have power series
expansions:
\begin{eqnarray}
\cn(u,\alpha)&=& \sum_{n\geq0}(-1)^{n-1}
c_n({\alpha}^2)\frac{u^{2u}}{(2n)!};\\
\dn(u,\alpha)&=& \sum_{n\geq0}(-1)^{n-1}
d_n({\alpha}^2)\frac{u^{2u}}{(2n)!},
\end{eqnarray}
where $c_n(\lambda)$ and $d_n(\lambda)$ are polynomials of degree
$n-1$ and $d_n(\lambda)={\lambda}^{n-1}c_n(\frac{1}{\lambda})$. We
call $c_n(\lambda)$ and $d_n(\lambda)$ the {\it \textbf{elliptic
polynomials}}. Flajolet \cite[Theorem 4]{Fla80} proved that the
coefficient $c_{n,k}$ of the polynomial $c_n(\lambda)$ counts the
alternating permutations over $[2n]$ having $k$ minima of even value
and
\begin{eqnarray}
\sum\limits_{n=0}^{\infty}c_n(\lambda)t^n=\frac{1}{1-s_0t-\cfrac{r_1t^2}{1-s_1t-\cfrac{r_2t^2}{1-s_2t-\cfrac{r_3t^2}{1-s_3t-\cdots}}}},
\end{eqnarray}
where $s_n=4(1+\lambda)n^2+4n+1$ and
$r_{n+1}=(2n+1)^2(2n+2)^2\lambda$ for $n\geq0$. Obviously, the
polynomial $c_n(\lambda)$ is the zeroth column of the
Catalan-Stieltjes triangle $\mathcal {C}=[\mathcal
{C}_{n,k}]_{n,k}$, which satisfies the recurrence
\begin{eqnarray}\label{rec+elliptic+C}
\mathcal {C}_{n,k}=(2k-1)k\mathcal
{C}_{n-1,k-1}+[4(\lambda+1)k^2+4k+1]\mathcal
{C}_{n-1,k}+4\lambda(1+2k)(k+1)\mathcal {C}_{n-1,k+1}
\end{eqnarray}
with $\mathcal {C}_{0,0}=1$. Let its row-generating polynomial
$\mathcal {C}_n(\lambda,q)=\sum_{k=0}^n\mathcal {C}_{n,k}q^k$.
Obviously, $\mathcal {C}_n(\lambda,0)=c_n(\lambda)$. Therefore we
call $\mathcal {C}_n(\lambda,q)$ the {\it \textbf{generalized
elliptic polynomial}}. By Theorem \ref{thm+TP+tridi+even} with
$a=x=2$, $b=u=v=y=1$ and $c=4\lambda$, we immediately obtain:
\begin{prop}
Let the Catalan-Stieltjes triangle $\mathcal {C}$ be defined by
(\ref{rec+elliptic+C}). Then we have
 \begin{itemize}
 \item [\rm (i)]
 the binomial row-generating matrix $\bm{C}(\lambda,q)$ is totally positive in the polynomial ring
 $\mathbb{Z}[\lambda,q]$ equipped with the
coefficientwise order;
 \item [\rm (ii)]
$(\mathcal {C}_{n}(q))_{n\geq0}$ is Hankel-totally positive in the
polynomial ring $\mathbb{Z}[\lambda,q]$ equipped with the
coefficientwise order and $3$-$(\lambda,q)$-log-convex;
\item [\rm (iii)]
$(c_n(\lambda))_{n\geq0}$ and $(d_n(\lambda))_{n\geq0}$ are
Hankel-totally positive in the polynomial ring $\mathbb{Z}[\lambda]$
equipped with the coefficientwise order and
$3$-$\lambda$-log-convex;
 \item [\rm (iv)]
the convolution $r_n=\sum_{k\geq0}\mathcal {C}_{n,k}s_kt_{n-k}$
preserves the Stieltjes moment property in $\mathbb{R}$ for
$\lambda\geq0$.
 \end{itemize}
\end{prop}

\subsection{A refined Stirling cycle polynomial}
Let $\left[
  \begin{array}{ccccc}
    n \\
   k\\
  \end{array}
\right]$ be the signless Stirling number. It counts the number of
permutations of $n$ elements which is the product of $k$ disjoint
cycles. It is well-known that it satisfies the following recurrence
$\left[
  \begin{array}{ccccc}
    n \\
   k\\
  \end{array}
\right]=(n-1)\left[
  \begin{array}{ccccc}
    n-1 \\
   k\\
  \end{array}
\right]+\left[
  \begin{array}{ccccc}
    n-1 \\
   k-1\\
  \end{array}
\right].$ The row-generating polynomial
$s_n(\lambda)=\sum_{k=0}^n\left[
  \begin{array}{ccccc}
    n \\
   k\\
  \end{array}
\right]{\lambda}^k$ is called the {\it Stirling cycle polynomial},
which has many nice properties \cite{Com74}. For example, its
explicit formula is given as
$s_n(\lambda)=\lambda(\lambda+1)\cdots(\lambda+n-1)$. Define a
generalized cycle-triangle $[S_{n,k}]_{n,k}$ satisfying the
following recurrence
\begin{eqnarray}
S_{n,k}=[b_0(n-1)+b_1]S_{n-1,k-1}+[a_0(n-1)+a_1]S_{n-1,k}
\end{eqnarray}
with $S_{0,0}=1$. From this recurrence relation, we derive the row
generating polynomial
\begin{eqnarray}
S_n(a_0,a_1,b_0,b_1,\lambda):=\sum_{k}S_{n,k}\lambda^k=\prod_{k=0}^{n-1}[(a_0+b_0\lambda)k+a_1+b_1\lambda]
\end{eqnarray}
for $n\geq1$. Obviously, $S_n(a_0,a_1,b_0,b_1,\lambda)$ is a
polynomial in five indeterminates $a_0,a_1,b_0,b_1$ and $\lambda$.
%
In addition, $S_n(1,0,0,1,\lambda)=s_n(\lambda)$,
$S_n(1,0,0,1,1)=n!$ and $S_n(1,1,1,0,1)=(2n-1)!!$. By
\cite[(4.3)]{Zhu21}, we have
\begin{eqnarray}
\sum\limits_{n=0}^{\infty}S_n(a_0,a_1,b_0,b_1,\lambda)t^n=\frac{1}{1-\gamma_0t-\cfrac{\beta_1t^2}{1-\gamma_1t-\cfrac{\beta_2t^2}{1-\gamma_2t-\cfrac{\beta_3t^2}{1-\gamma_3t-\cdots}}}},
\end{eqnarray} where
$\beta_{k+1}=[(a_0+b_0\lambda)k+b_1\lambda+a_1](a_0+b_0\lambda)(k+1)$
and $\gamma_{k}=[2(a_0+b_0\lambda)k+a_1+b_1\lambda]$. Thus
$S_n(a_0,a_1,b_0,b_1,\lambda)$ is exactly the zeroth column of the
Catalan-Stieltjes triangle $\mathcal {S}=[\mathcal {S}_{n,k}]_{n,k}$
satisfying the recurrence
\begin{eqnarray}\label{rec+refine+Euler+S}
\mathcal {S}_{n,k}&=&\mathcal {S}_{n-1,k-1}+[2(a_0+b_0\lambda)k+a_1+b_1\lambda]\mathcal {S}_{n-1,k}+\nonumber\\
&&\quad\quad\quad\quad[(a_0+b_0\lambda)k+b_1\lambda+a_1](a_0+b_0\lambda)(k+1)\mathcal
{S}_{n-1,k+1}
\end{eqnarray}
with $\mathcal {S}_{0,0}=1$. Let its row-generating polynomial
$\mathcal {S}_n(q)=\sum_{k=0}^n\mathcal {S}_{n,k}q^k$. Obviously,
$\mathcal {S}_n(0)=S_n(a_0,a_1,b_0,b_1,\lambda)$. In \cite[Theorem
1.3 (i)]{Zhu21}, we proved that $(S_n(a_0,a_1,b_0,b_1,\lambda))_{n}$
is coefficientwise Hankel-totally positive in
$(a_0,a_1,b_0,b_1,\lambda)$ and
$3$-$(a_0,a_1,b_0,b_1,\lambda)$-log-convex, which can be extended to
those of $\mathcal {S}_n(q)$. Further results for total positivity
are stated as follows, which is obvious by Theorem
\ref{thm+TP+tridi+even} with $a=u=0$, $b=v=1$, $c=x=a_0+b_0\lambda$
and $y=a_1+b_1\lambda$.
\begin{prop}\label{prop+refined Stirling cycle polynomial}
Let $\mathcal {S}=[\mathcal {S}_{n,k}]_{n,k}$ be the
Catalan-Stieltjes triangle defined by (\ref{rec+refine+Euler+S}).
Then we have
 \begin{itemize}
 \item [\rm (i)]
 the binomial row-generating matrix $\mathcal {S}(q)$ is totally positive in the polynomial ring
 $\mathbb{Z}[a_0,a_1,b_0,b_1,\lambda,q]$ equipped with the
coefficientwise order;
 \item [\rm (ii)]
 $(\mathcal {S}_{n}(q))_{n\geq0}$ and $(\mathcal {S}^{*}_{n}(q))_{n\geq0}$ are Hankel-totally
positive in the ring
 $\mathbb{Z}[a_0,a_1,b_0,b_1,\lambda,q]$ equipped with the
coefficientwise order and
$3$-$(a_0,a_1,b_0,b_1,\lambda,q)$-log-convex;
\item [\rm (iii)]
the convolution $r_n=\sum_{k\geq0}\mathcal {S}_{n,k}s_kt_{n-k}$
preserves the Stieltjes moment property in $\mathbb{R}$ for all
$a_0,a_1,b_0,b_1,\lambda\geq0$.
 \end{itemize}
\end{prop}

\subsection{A refined Eulerian polynomial}
Let $[A_{n,k}]_{n,k}$ be a triangular array satisfying the following
recurrence
\begin{eqnarray}
A_{n,k}=(b_0n-b_0k+b_2)A_{n-1,k-1}+(a_1k+a_2)A_{n-1,k}
\end{eqnarray}
with $A_{0,0}=1.$ The triangle $[A_{n,k}]_{n,k}$ has been studied in
different papers. It was conjectured that $[A_{n,k}]_{n,k}$ is
coefficientwise totally positive in the indeterminates
$a_1,a_2,b_0,b_2$ in \cite{CDDGS21}. Let
$A_n(a_1,a_2,b_0,b_2,\lambda)=\sum_{k}A_{n,k}\lambda^k$. Obviously,
$A_n(a_1,a_2,b_0,b_2,\lambda)$ can be regarded as a polynomial in
five indeterminates $a_1,a_2,b_0,b_2$ and $\lambda$. In fact, this
polynomial $A_n(a_1,a_2,b_0,b_2,\lambda)$ can be viewed as a common
generalization of many famous combinatorial polynomials. For
example:
\begin{itemize}
 \item
Let $\left\langle
  \begin{array}{ccccc}
    n \\
   k\\
  \end{array}
\right\rangle$ be the classical Eulerian number counting the number
of permutations of $n$ elements having $k-1$ descents \cite{Com74}.
Then $A_n(1,0,1,1,\lambda)$ is the classical Eulerian polynomial
$E_n(\lambda)=\sum_{k}\left\langle
  \begin{array}{ccccc}
    n \\
   k\\
  \end{array}
\right\rangle \lambda^k$ and
$A_n(1,1,1,1,\lambda)=E_{n+1}(\lambda)/{\lambda}$;
 \item
Let $\left\langle
  \begin{array}{ccccc}
    n \\
   k\\
  \end{array}
\right\rangle_B$ be the Eulerian number of type $B$ counting the
elements of signed group $B_n$ with $k$ $B$-descents
\cite{Bre94EuJC}. Then the Eulerian polynomial of type $B$ equals
$A_n(2,1,2,1,\lambda)$.
 \end{itemize}
By \cite[(4.6)]{Zhu21}, we know
\begin{eqnarray}
\sum\limits_{n=0}^{\infty}A_n(a_1,a_2,b_0,b_2,\lambda)t^n=\frac{1}{1-\gamma_0t-\cfrac{\beta_1t^2}{1-\gamma_1t-\cfrac{\beta_2t^2}{1-\gamma_2t-\cfrac{\beta_3t^2}{1-\gamma_3t-\cdots}}}},
\end{eqnarray} where
$\beta_{k+1}=\lambda(a_1b_0k+a_2b_0+a_1b_2)(k+1)$ and
$\gamma_{k}=(a_1+b_0\lambda)k+a_2+b_2\lambda$. It gives that
$A_n(a_1,a_2,b_0,b_2,\lambda)$ is exactly the zeroth column of the
next Catalan-Stieltjes triangle
$\mathscr{A}=[\mathscr{A}_{n,k}]_{n,k}$ satisfying the recurrence
\begin{eqnarray}\label{rec+refine+Euler+A}
\mathscr{A}_{n,k}=\mathscr{A}_{n-1,k-1}+[(a_1+b_0\lambda)k+a_2+b_2\lambda]\mathscr{A}_{n-1,k}+\lambda(a_1b_0k+a_2b_0+a_1b_2)(k+1)\mathscr{A}_{n-1,k+1}\nonumber\\
\end{eqnarray}
with $\mathscr{A}_{0,0}=1$. Denote its row-generating polynomial
$\mathscr{A}_n(q)=\sum_{k=0}^n\mathscr{A}_{n,k}q^k$. Obviously,
$\mathscr{A}_n(0)=A_n(a_1,a_2,b_0,b_2,\lambda)$. In \cite[Theorem
1.3 (v)]{Zhu21}, we proved for $0\in\{a_2,b_2,a_1-a_2,b_0-b_2\}$
that $(A_n(a_1,a_2,b_0,b_2,\lambda))_{n\geq0}$ is coefficientwise
Hankel-totally positive in $(\textbf{x},\lambda)$ and
$3$-$(\textbf{x},\lambda)$-log-convex,\footnote{
   Here $\textbf{x}$ means the remaining variables of $\{a_1,a_2,b_0,b_2\}$ excluding one being $0$ in the hypothesis.} which can be extended to
those of $\mathscr{A}_n(q)$. In addition, we get more results for
total positivity as follows.
\begin{prop}
Let $[\mathscr{A}_{n,k}]_{n,k}$ be defined by
(\ref{rec+refine+Euler+A}). If $0\in\{a_2,b_2,a_1-a_2,b_0-b_2\}$,
then
 \begin{itemize}
 \item [\rm (i)]
 the binomial row-generating matrix $\mathscr{A}(q)$ is totally positive in the polynomial ring
 $\mathbb{Z}[\textbf{x},\lambda,q]$ equipped with the
coefficientwise order;
 \item [\rm (ii)]
 $(\mathscr{A}_{n}(q))_{n\geq0}$ and $(\mathscr{A}^{*}_{n}(q))_{n\geq0}$ are Hankel-totally
positive in the ring
 $\mathbb{Z}[\textbf{x},\lambda,q]$ equipped with the
coefficientwise order and $3$-$(\textbf{x},\lambda,q)$-log-convex;
\item [\rm (iii)]
the convolution $r_n=\sum_{k\geq0}\mathscr{A}_{n,k}s_kt_{n-k}$
preserves the Stieltjes moment property in $\mathbb{R}$ for
$\textbf{x}\geq0$ and $\lambda\geq0$.
 \end{itemize}
\end{prop}
\begin{proof}
(1) If $a_2=0$, then the desired results in (i)-(iii) follow from
Theorem \ref{thm+TP+tridi+even} by taking $a=u=0$, $b=v=1$, $c=a_1$,
$x=b_0\lambda$ and $y=b_2\lambda$.

(2) If $b_2=0$, then the desired results in (i)-(iii) follow from
Theorem \ref{thm+TP+tridi+even} by taking $a=u=0$, $b=v=1$,
$c=b_0\lambda$, $x=a_1$ and $y=a_2$.

(3) If $a_1=a_2$, then the desired results in (i)-(iii) follow from
Theorem \ref{thm+TP+tridi+odd} by taking $a=u=0$, $b=v=1$, $c=a_1$,
$x=b_0\lambda$ and $y=b_2\lambda$.

(4) If $b_0=b_2$, then the desired results in (i)-(iii) follow from
Theorem \ref{thm+TP+tridi+odd} by taking $a=u=0$, $b=v=1$,
$c=b_0\lambda$, $x=a_1$ and $y=a_2$.
\end{proof}

\section{Row-generating polynomials of exponential Riordan
arrays}\label{section+EEA}

The Riordan array is vast and still growing and the applications
cover a wide range of subjects, such as enumerative combinatorics,
combinatorial sums, recurrence relations and computer science, among
other topics \cite{Bar11-2,CLW15,DFR,DS,MRSV97,SGWW91,Spr11}. The
{\it exponential Riordan array}~\cite{Bar11-2,DFR,DS}, denoted by
$\textbf{R}=[\textbf{R}_{n,k}]_{n,k}=\left(g(t),f(t)\right)$, is an
infinite lower triangular matrix whose exponential generating
function of the $k$th column is $ g(t)f^k(t)/k! $ for $k\geq0$,
where $g(0)f'(0)\neq0$ and $f(0)=0$. In other words, for $n,k\geq0$,
\begin{eqnarray}
\textbf{R}_{n,k}&=&\frac{n!}{k!}[t^n]g(t)f^k(t).
\end{eqnarray}
Let $\textbf{R}_n(q)=\sum_{k=0}^n\textbf{R}_{n,k}q^k$ be the
row-generating polynomial of $\textbf{R}$. Then we have
\begin{eqnarray}\label{generating funtion+GR}
\sum_{n\geq0}\textbf{R}_n(q)\frac{t^n}{n!}&=&g(t)\exp\left(qf(t)\right).
\end{eqnarray}
The group law is given by
\begin{eqnarray}\label{Riordan+product}
(g,f)*(h,\ell)=(g\times h(f),\ell(f)).
\end{eqnarray}
The identity for this law is $I =(1,t)$ and the inverse of $(g,f)$
is $(g,f)^{-1}=(1/(g(\overline{f})),\overline{f})$, where
$\overline{f}$ is the compositional inverse of $f$, i.e.,
$\overline{f}(f(t))=f(\overline{f}(t))=t$.

In \cite{CLW152,CW19}, Chen et al. gave two criteria for total
positivity of ordinary Riordan arrays. In \cite{Zhu213}, we proved:
Given a P\'olya frequency ogf $g(t)$ of order $r$, if one of $f(t)$
and $1/\bar{f}'(t)$ is a P\'olya frequency ogf of order $r$, then
exponential Riordan arrays $\left(g(t),f(t)\right)$ and
$\left(g(f(t)),f(t)\right)$ are totally positive of order $r$. The
following question concerning coefficientwise Hankl-total positivity
of the row-generating polynomials is open.
\begin{ques}\label{ques+ERA}
What conditions can ensure coefficientwise Hankel-total positivity
of the sequence of row-generating polynomials for the exponential
Riordan array $\left(g(t),f(t)\right)$ ?
\end{ques}

Generally speaking, if one of $f(t)$ or $g(t)$ is known, then it is
a difficult problem to choose another proper function to form an
exponential Riordan array $(g(t),f(t))$ with such desired
properties. In what follows our results will give the partial
answers.

\subsection{Main results for exponential Riordan
arrays} In \cite{Zhu213}, we obtained the following criterion for
coefficientwise total positivity.
\begin{thm} \cite{Zhu213}\label{thm+zhu21}
If $1/\bar{f}'(t)$ is a P\'olya frequency ogf of order ${r}$, then
\begin{itemize}
  \item [\rm (i)]
  $\left(e^{\lambda f(t)},f(t)\right)$ is coefficientwise
totally positive of order $r$ in $\lambda$;
 \item [\rm (ii)]
 $(f'(x),f(x))$ is totally positive of order $r$;
  \item [\rm (iii)]
 the sequence of row-generating polynomials
of $\left(e^{\lambda f(t)},f(t)\right)$ is coefficientwise
Hankel-totally positive of order $r$ in $(q,\lambda)$;
  \item [\rm (iv)]
 the sequence of row-generating polynomials of $(f'(x),f(x))$ is coefficientwise Hankel-totally positive of order $r$ in
 $q$;
  \item [\rm (v)]
 the sequence $(f_{n+1})_{n\geq0}$ is Hankel-totally positive of order $r$, where $f(t)=\sum_{n\geq1}f_n\frac{t^n}{n!}$.
 \end{itemize}
\end{thm}

Let $f(t)=\sum_{n\geq1}f_{n}\frac{t^n}{n!}$. For the exponential
Riordan array $(1,f(t))$, its entry $\textbf{R}_{n,k}$ is the
generating polynomial for unordered forests of increasing ordered
trees on $n+1$ total vertices with $k$ components in which each
vertex with $i$ children gets a weight $x_{i}$. In particular,
$f_n=\textbf{R}_{n,1}$ is the generating polynomial for increasing
ordered trees $n+1$ vertices in which each vertex with $i$ children
gets a weight $x_{i}$. We refer the reader to \cite{BFS92,PS19}.

In what follows we will extend Theorem \ref{thm+zhu21} to the
following generalized result.

\begin{thm}\label{thm+TP+ERA+oder+r}
Let $R$ be a partially ordered commutative ring and
$1/\overline{f}'(t)=\varphi(t)\phi(t)$. Assume that
$[\textbf{R}_{n,k}]_{n,k}=(\varphi(f(t))\exp(\lambda f(t)),f(t))$
and $\textbf{R}_n(q)=\sum_{k}\textbf{R}_{n,k}q^k$. If both $\varphi$
and $\phi$ are P\'olya frequency ogf of order $r$ in $R[[t]]$, then
\begin{itemize}
  \item [\rm (i)]
 $(\varphi(f(t))\exp(\lambda f(t)),f(t))$ is
totally positive of order $r$ in the polynomial ring $R[\lambda]$
equipped with the coefficientwise order;
 \item [\rm (ii)]
both $(\textbf{R}_n(q))_{n\geq0}$ and
$(\textbf{R}^{*}_n(q))_{n\geq0}$ are Hankel-totally positive of
order $r$ in the polynomial ring $R[\lambda,q]$ equipped with the
coefficientwise order;
  \item [\rm (iii)]
both $(T_n(\lambda))_{n\geq0}$ and $(T^{*}_n(\lambda))_{n\geq0}$ are
Hankel-totally positive of order $r$ in the polynomial ring
$R[\lambda]$ equipped with the coefficientwise order, where
$\varphi(f(t))\exp(\lambda f(t))=\sum_{n}T_{n}(\lambda)
\frac{t^n}{n!}$;
  \item [\rm (iv)]
$(f_{n+1})_{n\geq0}$ is Hankel-totally positive of order $r$ in $R$,
where $f(t)=\sum_{n\geq1}f_n\frac{t^n}{n!}$.
 \end{itemize}
\end{thm}

\begin{rem}
\begin{itemize}
  \item [\rm (1)]
Note that $1/\overline{f}'(t)=\varphi(t)\phi(t)$ implies
$f'(t)=\varphi(f(t))\phi(f(t))$. This type equation often is called
the {\it autonomous differential equation} \cite{BFS92}. It plays an
important role in the flow function of Copeland \cite[A145271]{Slo},
which is also closely related to formal group laws for elliptic
curves, the Abel equation, the Schr\"{o}der's functional equation,
Koenigs functions for compositional iterates, the renormalization
group equation and Hopf algebra.
  \item [\rm (2)]
By taking $\lambda=0$ and $\phi=1$, (i) and (ii) of Theorem
\ref{thm+TP+ERA+oder+r} implies Theorem \ref{thm+zhu21} (ii) and
(iv), respectively;
 \item [\rm (3)]
By taking $\varphi=1$, (i) and (ii) of Theorem
\ref{thm+TP+ERA+oder+r} implies Theorem \ref{thm+zhu21} (i) and
(iii), respectively.
   \end{itemize}
\end{rem}

 In order to prove Theorem \ref{thm+TP+ERA+oder+r}, we need the following equivalent characterization for the exponential Riordan arrays.
\begin{prop}\emph{\cite{DFR}}\label{prop+TP+production}
Let $[\textbf{R}_{n,k}]_{n,k\geq0}=\left(g(t),f(t)\right)$ be an
exponential Riordan array. Then there exist two sequences
$(z_n)_{n\geqslant0}$ and $(a_n)_{n\geqslant0}$ such that
\begin{eqnarray}
\textbf{R}_{0,0}=1,\ \
\textbf{R}_{n,0}=\sum_{i\geqslant0}i!z_iR_{n-1,i},\ \
\textbf{R}_{n,k}=\frac{1}{k!}\sum_{i\geqslant
k-1}i!(z_{i-k}+ka_{i-k+1})\textbf{R}_{n-1,i}
\end{eqnarray}
for $n,k\geqslant1$. In particular,
\begin{eqnarray}
Z(t)=\frac{g'(\bar{f}(t))}{g(\bar{f}(t))},\quad
A(t)=f'(\bar{f}(t)),
\end{eqnarray}
where $Z(t)=\sum_{n\geqslant0} z_nt^n$ and $A(t)=\sum_{n\geqslant0}
a_nt^n$.
\end{prop}

Associated to each exponential Riordan array
$\textbf{R}=\left(g(t),f(t)\right)$, there is a production matrix
$\textbf{P}$ such that
\begin{eqnarray}
\overline{\textbf{R}}=\textbf{R}\textbf{P},
\end{eqnarray}
where $\overline{\textbf{R}}$ is obtained from $\textbf{R}$ with the
first row removed. Assume that $z_{-1}=0.$ Deutsch {\it et
al.}~\cite{DFR} showed the production matrix
\begin{eqnarray}\label{PP}
\textbf{P}=[p_{i,j}]_{i,j\geqslant0}=\left[
  \begin{array}{ccccccc}
    z_0&a_0& \\
   1!z_1& \frac{1!}{1!}(z_0+a_1)&a_0&&\\
   2!z_2&\frac{2!}{1!}(z_1+a_2) &\frac{2!}{2!}(z_0+2a_1)&a_0&\\
   3!z_3 &\frac{3!}{1!}(z_2+a_3)&\frac{3!}{2!}(z_1+2a_2)&\frac{3!}{3!}(z_0+3a_1)&\ddots\\
\vdots &\vdots&\vdots&\vdots&\ddots \\
  \end{array}
\right],
\end{eqnarray}
where the elements
\begin{eqnarray}
p_{i,j}=\frac{i!}{j!}(z_{i-j}+ja_{i-j+1})\quad \text{for}\quad i,j\geqslant 0.
\end{eqnarray}

\textbf{Proof of Theorem \ref{thm+TP+ERA+oder+r}:} Let
$g(t)=\varphi(f(t))\exp(\lambda f(t))$. We derive
\begin{eqnarray}\frac{g'(t)}{g(t)}=\frac{\varphi'(f(t))f'(t)+\varphi(f(t))\lambda f'(t)}{\varphi(f(t))}.
\end{eqnarray}
In addition, $f(\overline{f}(t))=t$ also implies
$f'(\overline{f}(t))\overline{f}'(t)=1$. Then for the exponential
Riordan array $(\varphi(f(t))\exp(\lambda f(t)),f(t))$, by
Proposition \ref{prop+TP+production}, we get its two formal power
series
\begin{eqnarray}
Z(t)=\frac{\varphi'(t)}{\varphi(t)\overline{f}'(t)}+\frac{\lambda}{ \overline{f}'(t)}=\phi(t)(\varphi'(t)+\lambda\varphi(t)),\quad A(t)=\phi(t)\varphi(t).
\end{eqnarray}
It follows from (\ref{PP}) that its production matrix
\begin{eqnarray}\label{PPPP}
\textbf{P}&=&\Lambda \Gamma(Z)\Lambda^{-1}+\Lambda \Gamma(A)\Theta\Lambda^{-1}\nonumber\\
&=&\Lambda
\Gamma(\phi(t))\left(\lambda\Gamma(\varphi(t))+\Gamma(\varphi'(t))+
\Gamma(\varphi(t))\left(\begin{array}{ccccc}
0&1&&&\\
&&2&&\\
&&&\ddots&\\
\end{array}\right)\right)\Lambda^{-1}\nonumber\\
&=&\Lambda
\Gamma(\phi(t))\left(\lambda\Gamma(\varphi(t))+\left(\begin{array}{ccccc}
0&1&&&\\
&&2&&\\
&&&3&\\
&&&&\ddots\\
\end{array}\right)
\Gamma(\varphi(t))\right)\Lambda^{-1}\nonumber\\
&=&\Lambda \Gamma(\phi(t))\left(\begin{array}{cccccc}
\lambda&1&&&\\
&\lambda&2&&\\
&&\lambda&3&&\\
&&&\ddots&\ddots
\end{array}\right)\Gamma(\varphi(t))\Lambda^{-1}\nonumber\\
&=&\underbrace{\Lambda
\Gamma(\phi(t))\Lambda^{-1}}_{\bm{\Gamma}^{(\Lambda,\phi)}}\left(\begin{array}{cccccc}
\lambda&1&&\\
&\lambda&1&\\
&&\lambda&1&\\
&&&\ddots&\ddots\\
\end{array}\right)
\underbrace{\Lambda\Gamma(\varphi(t))\Lambda^{-1}}_{\bm{\Gamma}^{(\Lambda,\varphi)}}.
\end{eqnarray}
Thus (i) follows from Theorem \ref{thm+HTP+Row+GmJR+triangle} (i)
with $q=0$ and (ii) is immediate from Theorem
\ref{thm+HTP+Row+GmJR+triangle} (ii) and Corollary
\ref{cor+HTP+reversed}. By
$\textbf{R}_n(0)=\textbf{R}_{n,0}=T_n(\lambda)$, (ii) implies (iii).
Note $A(t)=\phi(t)\varphi(t)$ implies
$f'(t)=\varphi(f(t))\phi(f(t))$. Taking $\phi(t)=1$ and $\lambda=0$,
we have
$$(\varphi(f(t))\exp(\lambda f(t)),f(t))=(f'(t),f(t)).$$ Thus
combining (iii) and $T_n(0)=f_{n+1}$ gives (iv). \qed

Taking $r\rightarrow\infty$ in Theorem \ref{thm+TP+ERA+oder+r}, we
get the following result, whose (v) is by Theorem \ref{thm+conv}.

\begin{thm}\label{thm+TP+ERA}
Let $R$ be a partially ordered commutative ring. If there exist two
P\'olya frequency ogf $\varphi$ and $\phi$ in $R[[t]]$ such that
$1/\overline{f}'(t)=\varphi(t)\phi(t)$, then
\begin{itemize}
\item [\rm (i)]
the exponential Riordan array $(\varphi(f(t))\exp(\lambda
f(t)),f(t))$ is totally positive in the polynomial ring $R[\lambda]$
equipped with the coefficientwise order;
\item [\rm (ii)]
both $(\textbf{R}_n(q))_{n\geq0}$ and
$(\textbf{R}^{*}_n(q))_{n\geq0}$ are Hankel-totally positive in the
polynomial ring $R[\lambda,q]$ equipped with the coefficientwise
order, where $\textbf{R}_n(q)$ is the row-generating polynomial of
$(\varphi(f(t))\exp(\lambda f(t)),f(t))$;
\item [\rm (iii)]
both $(T_n(\lambda))_{n\geq0}$ and $(T^{*}_n(\lambda))_{n\geq0}$ are
Hankel-totally positive in the polynomial ring $R[\lambda]$ equipped
with the coefficientwise order, where $\varphi(f(t))\exp(\lambda
f(t))=\sum_{n}T_{n}(\lambda) \frac{t^n}{n!}$;
\item [\rm (iv)]
 the sequence $(f_{n+1})_{n\geq0}$ is Hankel-totally positive in $R$, where
 $f(t)=\sum_{n\geq1}f_n\frac{t^n}{n!}$;
 \item [\rm (v)]
the convolution $z_n=\sum_{k=0}^{n}\textbf{R}_{n,k}x_ky_{n-k}$
preserves the Stieltjes moment property in $\mathbb{R}$ for
$\lambda\geq0$ and $R=\mathbb{R}$, where
$[\textbf{R}_{n,k}]_{n,k}=(\varphi(f(t))\exp(\lambda f(t)),f(t))$.
\end{itemize}
\end{thm}

For Theorem \ref{thm+TP+ERA}, if $\phi(t)=1$ and
$\varphi(t)=\prod_{i=0}^{m}\frac{1}{1-x_it}$, then by (\ref{PPPP}),
the production matrix of the exponential Riordan array
$(f'(t),f(t))$ is
\begin{eqnarray}\label{Production+decomposition}
P&=&\left(\begin{array}{cccccc}
\lambda&1&&\\
&\lambda&1&\\
&&\lambda&1&\\
&&&\ddots&\ddots\\
\end{array}\right)
\Lambda\Gamma(\varphi(t))\Lambda^{-1}.
\end{eqnarray}
By (\ref{Production+decomposition}) with $q=0$ and \cite[Propostion
12.20: b]{PSZ18}, $f_{n+1}$ equals the \textbf{multivariate Eulerian
polynomial of negative type} $\mathcal {Q}_n^{(m)-}(\textbf{x})$ in
\cite[Section 12.3.1]{PSZ18}, which was defined in terms of
increasing multi-$m$-ary trees on $n+1$ vertices in each $i$-edge
gets a weights $x_{i}$ and can be expressed in terms of differential
operators \cite[Section 12.3.1]{PSZ18}. Meanwhile, for the
exponential Riordan array $(f'(t),f(t))$, its entry
$\textbf{R}_{n,k}$ is the generating polynomial for unordered
forests of increasing multi-$m$-ary trees on $n+1$ total vertices
with $k+1$ components in which each $i$-edge gets a weight $x_{i}$.
When $m\rightarrow \infty$, the multivariate Eulerian polynomial of
negative type gives the \textbf{Eulerian symmetric function of
negative type} $\mathcal {Q}_{n}^{(\infty)-}(\textbf{x})$
\cite[Section 12.3.1]{PSZ18}. Note $T_n(0)=f_{n+1}$ and
$\textbf{R}_{n}(q)=T_n(\lambda)$. Thus, both $\textbf{R}_{n}(q)$ and
$T_n(\lambda)$ are the generalizations of $\mathcal
{Q}_{n}^{(m)-}(\textbf{x})$. In addition, when $m\rightarrow
\infty$, both $\textbf{R}_{n}(q)$ and $T_n(\lambda)$ are also
extended to the generalizations of $\mathcal
{Q}_{n}^{(\infty)-}(\textbf{x})$.

In \cite{PSZ18}, the $m$-branched Stieltjes-type continued fraction
(\ref{eq+f0+mSfrac}) plays a heart role in coefficientwise
Hankel-total positivity. In the following, we also present the
$m$-branched Stieltjes-type continued fraction for the ordinary
generating function of the row-generating polynomials for
exponential Riordan arrays.

\begin{thm}\label{thm+TP+ERA+branched fraction}
Assume that
$1/\overline{f}'(t)=\phi(t)\varphi(t)=\prod_{i=0}^m(1+x_it)$. Let
$\textbf{R}_n(q)$ and $T_{n}(\lambda)$ be defined in Theorem
\ref{thm+TP+ERA}.
\begin{itemize}
\item [\rm (i)]
If $\phi(t)=1$, then we have the $(m+1)$-branched Stieltjes-type
continued fraction
\begin{eqnarray*}
   1+\sum_{n\geq0}(q+\lambda)\textbf{R}_{n}(q)t^n
   & = &
   \cfrac{1}
         {1 \,-\, \alpha_{m+1} t
            \prod\limits_{i_1=1}^{m+1}
                 \cfrac{1}
            {1 \,-\, \alpha_{m+1+i_1} t
               \prod\limits_{i_2=1}^{m+1}
               \cfrac{1}{1 - \cdots}
           }
         }
%
\end{eqnarray*}
with coefficients
$$(\alpha_{i})_{i\geq m+1}=(\lambda+q,\underbrace{x_0,\ldots,x_m}_{m+1},\lambda+q,\underbrace{2x_0,\ldots,2x_m}_{m+1},\lambda+q,\underbrace{3x_0,\ldots,3x_m}_{m+1},\ldots);$$
In particular, we have
\begin{eqnarray*}
   1+\sum_{n\geq0}\lambda T_{n}(\lambda)t^n
  =
   \cfrac{1}
         {1 \,-\, \alpha_{m+1} t
            \prod\limits_{i_1=1}^{m+1}
                 \cfrac{1}
            {1 \,-\, \alpha_{m+1+i_1} t
               \prod\limits_{i_2=1}^{m+1}
               \cfrac{1}
            {1 \,-\,  \cdots}
            }
         }
%
\end{eqnarray*}
with coefficients $(\alpha_{i})_{i\geq
m+1}=(\lambda,\underbrace{x_0,\ldots,x_m}_{m+1},\lambda,\underbrace{2x_0,\ldots,2x_m}_{m+1},\lambda,\underbrace{3x_0,\ldots,3x_m}_{m+1},\ldots).$
\item [\rm (ii)]
If $\varphi(t)=1$, then the $(m+1)$-branched Stieltjes-type
continued fraction
\begin{eqnarray*}
  \sum_{n\geq0}\textbf{R}_{n}(q)t^n
   & = &
   \cfrac{1}
         {1 \,-\, \alpha_{m+1} t
            \prod\limits_{i_1=1}^{m+1}
                 \cfrac{1}
            {1 \,-\, \alpha_{m+1+i_1} t
               \prod\limits_{i_2=1}^{m+1}
               \cfrac{1}{1 - \cdots}
           }
         }
%
\end{eqnarray*}
with coefficients
$$(\alpha_{i})_{i\geq m+1}=(\lambda+q,\underbrace{x_0,\ldots,x_m}_{m+1},\lambda+q,\underbrace{2x_0,\ldots,2x_m}_{m+1},\lambda+q,\underbrace{3x_0,\ldots,3x_m}_{m+1},\ldots);$$
In particular, we have
\begin{eqnarray*}
  \sum_{n\geq0} T_{n}(\lambda)t^n
  =
   \cfrac{1}
         {1 \,-\, \alpha_{m+1} t
            \prod\limits_{i_1=1}^{m+1}
                 \cfrac{1}
            {1 \,-\, \alpha_{m+1+i_1} t
               \prod\limits_{i_2=1}^{m+1}
               \cfrac{1}
            {1 \,-\, \alpha_{m+1+i_1+i_2} t
               \prod\limits_{i_3=1}^{m+1}
               \cfrac{1}{1 - \cdots}
            }
           }
         }
%
\end{eqnarray*}
with coefficients $(\alpha_{i})_{i\geq
m+1}=(\lambda,\underbrace{x_0,\ldots,x_m}_{m+1},\lambda,\underbrace{2x_0,\ldots,2x_m}_{m+1},\lambda,\underbrace{3x_0,\ldots,3x_m}_{m+1},\ldots).$
\item [\rm (iii)]
We have the $m$-branched Stieltjes-type continued fraction
\begin{eqnarray*}
   1+\sum_{n\geq0}x_0f_{n+1}t^{n+1}
  =
   \cfrac{1}
         {1 \,-\, \alpha_{m} t
            \prod\limits_{i_1=1}^{m}
                 \cfrac{1}
            {1 \,-\, \alpha_{m+i_1} t
               \prod\limits_{i_2=1}^{m}
               \cfrac{1}
            {1 \,-\, \alpha_{m+i_1+i_2} t
               \prod\limits_{i_3=1}^{m}
               \cfrac{1}{1 - \cdots}
            }
           }
         }
%
\end{eqnarray*}
with coefficients $(\alpha_{i})_{i\geq
m}=(\underbrace{x_0,\ldots,x_m}_{m+1},\underbrace{2x_0,\ldots,2x_m}_{m+1},\underbrace{3x_0,\ldots,3x_m}_{m+1},\ldots).$
\end{itemize}
\end{thm}

\begin{proof}
(i) If $\phi(t)=1$ and $\varphi(t)=\prod_{i=0}^m(1+x_it)$, then by
(\ref{PPPP}), we have
\begin{eqnarray}
\textbf{P} &=&\left(\begin{array}{cccccc}
\lambda&1&&\\
&\lambda&1&\\
&&\lambda&1&\\
&&&\lambda&1&\\
&&&&\ddots&\ddots\\
\end{array}\right)
\underbrace{\Lambda\Gamma(\varphi(t))\Lambda^{-1}}\nonumber\\
&=&\left(\begin{array}{cccccc}
\lambda&1&&\\
&\lambda&1&\\
&&\lambda&1&\\
&&&\lambda&1&\\
&&&&\ddots&\ddots\\
\end{array}\right)\Lambda\left(\prod_{i=0}^m\Gamma(1+x_it)\right)\Lambda^{-1}\nonumber\\
&=&\left(\begin{array}{cccccc}
\lambda&1&&\\
&\lambda&1&\\
&&\lambda&1&\\
&&&\lambda&1&\\
&&&&\ddots&\ddots\\
\end{array}\right)\left(\prod_{i=0}^m\Lambda\Gamma(1+x_it)\Lambda^{-1}\right)\nonumber\\
&=&\left(\begin{array}{cccccc}
\lambda&1&&\\
&\lambda&1&\\
&&\lambda&1&\\
&&&\lambda&1&\\
&&&&\ddots&\ddots\\
\end{array}\right)\prod_{i=0}^m\left[
  \begin{array}{ccccccc}
    1&& \\
    x_i&1&\\
    &2x_i&1&&\\
    &&3x_i&1&\\
 &&&\ddots&\ddots \\
  \end{array}
\right].
\end{eqnarray}
In consequence, the production matrix of $(\varphi(f(t))\exp(\lambda
f(t)),f(t))\textbf{B}_q$ equals
\begin{eqnarray*}
\textbf{B}_q^{-1}\textbf{P}\textbf{B}_q
&=&\left(\begin{array}{cccccc}
q+\lambda&1&&\\
&q+\lambda&1&\\
&&q+\lambda&1&\\
&&&q+\lambda&1&\\
&&&&\ddots&\ddots\\
\end{array}\right)\prod_{i=0}^m\left[
  \begin{array}{ccccccc}
    1&& \\
    x_i&1&\\
    &2x_i&1&&\\
    &&3x_i&1&\\
 &&&\ddots&\ddots \\
  \end{array}
\right],
\end{eqnarray*}
which in terms of Proposition 7.6 (Odd contraction formula for
m-Stieltjes-Rogers polynomials) and Proposition 8.2 (b) in
\cite{PSZ18} is exactly the production matrix for the
$(m+1)$-branched Stieltjes-type continued fraction
\begin{eqnarray}
  1+ \sum_{n\geq0}(q+\lambda)\textbf{R}_{n}(q)t^n
   =
   \cfrac{1}
         {1 \,-\, \alpha_{m} t
            \prod\limits_{i_1=1}^{m}
                 \cfrac{1}
            {1 \,-\, \alpha_{m+i_1} t
               \prod\limits_{i_2=1}^{m}
               \cfrac{1}
            {1 \,-\, \alpha_{m+i_1+i_2} t
               \prod\limits_{i_3=1}^{m}
               \cfrac{1}{1 - \cdots}
            }
           }
         }
%
\end{eqnarray}
with coefficients
$$(\alpha_{i})_{i\geq m+1}=(\lambda+q,\underbrace{x_0,\ldots,x_m}_{m+1},\lambda+q,\underbrace{2x_0,\ldots,2x_m}_{m+1},\lambda+q,\underbrace{3x_0,\ldots,3x_m}_{m+1},\ldots).$$

(ii) Similarly, if $\varphi(t)=1$ and
$\phi(t)=\prod_{i=0}^m(1+x_it)$, then
\begin{eqnarray*}
\textbf{B}_q^{-1}\textbf{P}\textbf{B}_q &=&\prod_{0=1}^m\left[
  \begin{array}{ccccccc}
    1&& \\
    x_i&1&\\
    &2x_i&1&&\\
    &&3x_i&1&\\
 &&&\ddots&\ddots \\
  \end{array}
\right]\left(\begin{array}{cccccc}
q+\lambda&1&&\\
&q+\lambda&1&\\
&&q+\lambda&1&\\
&&&q+\lambda&1&\\
&&&&\ddots&\ddots\\
\end{array}\right),
\end{eqnarray*}
which in terms of \cite[Proposition 7.2 and Proposition 8.2
(b)]{PSZ18} implies the desired $(m+1)$-branched Stieltjes-type
continued fraction
\begin{eqnarray}
\sum_{n\geq0}\textbf{R}_{n}(q)t^n
   =
   \cfrac{1}
         {1 \,-\, \alpha_{m} t
            \prod\limits_{i_1=1}^{m}
                 \cfrac{1}
            {1 \,-\, \alpha_{m+i_1} t
               \prod\limits_{i_2=1}^{m}
               \cfrac{1}
            {1 \,-\, \alpha_{m+i_1+i_2} t
               \prod\limits_{i_3=1}^{m}
               \cfrac{1}{1 - \cdots}
            }
           }
         }
%
\end{eqnarray}
with coefficients
$$(\alpha_{i})_{i\geq m+1}=(\lambda+q,\underbrace{x_0,\ldots,x_m}_{m+1},\lambda+q,\underbrace{2x_0,\ldots,2x_m}_{m+1},\lambda+q,\underbrace{3x_0,\ldots,3x_m}_{m+1},\ldots).$$

(iii) For $\phi(t)=1$ and $\lambda=0$, we have $T_n(0)=f_{n+1}$ and
the production matrix
\begin{eqnarray}
\textbf{P} &=&\left(\begin{array}{cccccc}
0&1&&\\
&0&1&\\
&&0&1&\\
&&&0&1&\\
&&&&\ddots&\ddots\\
\end{array}\right)\prod_{i=0}^m\left[
  \begin{array}{ccccccc}
    1&& \\
    x_i&1&\\
    &2x_i&1&&\\
    &&3x_i&1&\\
 &&&\ddots&\ddots \\
  \end{array}
\right]\nonumber\\
&=&\left(\begin{array}{cccccc}
x_0&1&&\\
&2x_0&1&\\
&&3x_0&1&\\
&&&4x_0&1&\\
&&&&\ddots&\ddots\\
\end{array}\right)\prod_{i=1}^m\left[
  \begin{array}{ccccccc}
    1&& \\
    x_i&1&\\
    &2x_i&1&&\\
    &&3x_i&1&\\
 &&&\ddots&\ddots \\
  \end{array}
\right],
\end{eqnarray}
which in terms of Proposition 7.6 (Odd contraction formula for
m-Stieltjes-Rogers polynomials) and Proposition 8.2 (b) in
\cite{PSZ18} is exactly the production matrix for the $m$-branched
Stieltjes-type continued fraction
\begin{eqnarray}
   1+\sum_{n\geq0}x_0f_{n+1}t^{n+1}
  =
   \cfrac{1}
         {1 \,-\, \alpha_{m} t
            \prod\limits_{i_1=1}^{m}
                 \cfrac{1}
            {1 \,-\, \alpha_{m+i_1} t
               \prod\limits_{i_2=1}^{m}
               \cfrac{1}
            {1 \,-\, \alpha_{m+i_1+i_2} t
               \prod\limits_{i_3=1}^{m}
               \cfrac{1}{1 - \cdots}
            }
           }
         }
%
\end{eqnarray}
with coefficients $(\alpha_{i})_{i\geq
m}=(\underbrace{x_0,\ldots,x_m}_{m+1},\underbrace{2x_0,\ldots,2x_m}_{m+1},\underbrace{3x_0,\ldots,3x_m}_{m+1},\ldots).$
\end{proof}

For Theorem \ref{thm+TP+ERA+branched fraction}, it follows from
\cite[Section 12]{PSZ18} that the $m$-Stieltjes--Rogers polynomial
$f_{n+1}$ equals the \textbf{multivariate Eulerian polynomial}
$\mathcal {Q}_{n}^{(m)}(\textbf{x})$ in \cite[Section 12.2]{PSZ18},
which was defined in terms of increasing $(m+1)$-ary trees on $n+1$
total vertices in which each $i$-edge gets a weight $x_{i}$ and has
differential expressions \cite[Section 12.2]{PSZ18}. Meanwhile, for
the exponential Riordan array $(f'(t),f(t))$, its entry
$\textbf{R}_{n,k}$ is the generating polynomial for unordered
forests of increasing $(m+1)$-ary trees on $n+1$ total vertices with
$k+1$ components in which each $i$-edge gets a weight $x_{i}$. In
particular, $\textbf{R}_{n,1}=f_{n+1}$. When $m\rightarrow \infty$,
the multivariate Eulerian polynomial $\mathcal
{Q}_{n}^{(m)}(\textbf{x})$ gives the \textbf{Eulerian symmetric
function} $\mathcal {Q}_{n}^{(\infty)}(\textbf{x})$ \cite[Section
12.2]{PSZ18}. Note $T_n(0)=f_{n+1}$ and
$\textbf{R}_{n}(q)=T_n(\lambda)$. Thus, both $R_{n}(q)$ and
$T_n(\lambda)$ are the generalizations of the multivariate Eulerian
polynomial $\mathcal {Q}_{n}^{(m)}(\textbf{x})$. In addition, when
$m\rightarrow \infty$, both $\textbf{R}_{n}(q)$ and $T_n(\lambda)$
also give the generalizations of the Eulerian symmetric function
$\mathcal {Q}_{n}^{(\infty)}(\textbf{x})$. Copeland also call the
multivariate Eulerian polynomial as \textbf{the refined Eulerian
polynomial} \cite[A145271]{Slo}.

Finally, based on the exponential Riordan arrays, we also consider
an associated array as follows.
\begin{thm}\label{thm+TP+ERA+fraction}
Let  $f(t)\in \mathbb{R}[[t]]$.  Define $[F_{n,k}]_{n,k}:=(1,f(t))$
and
\begin{eqnarray}
\left(\frac{1}{1-qf(t)}\right)^y:&=&\sum_{n\geq0}F^{\diamond}_{n}(q,y)\frac{t^n}{n!}=\sum_{n\geq0}\sum_{k=0}^nF^{\diamond}_{n,k}(q)y^k\frac{t^n}{n!}.
\end{eqnarray}
If $1/\overline{f}'(t)$ is a P\'olya frequency ogf in
$\mathbb{R}[[t]]$, then
\begin{itemize}
  \item [\rm (i)]
 $[F^{\diamond}_{n,k}(q)]_{n,k}$ is
coefficientwise totally positive in $q$;
 \item [\rm (ii)]
$(F^{\diamond}_n(q,y))_{n\geq0}$ is a Stieltjes moment sequence (of
real numbers) for $q\geq0$ and $y\geq0$;
  \item [\rm (iii)]
$t_n=\sum_{k=0}^{n}F^{\diamond}_{n,k}(q)r_ks_{n-k}$ preserves the
Stieltjes moment property in $\mathbb{R}$ for $q\geq0$;
  \item [\rm (iv)]
if $(F_{n,k})_{k=0}^n$ is a P\'olya frequency sequence, then so is
$(F^{\diamond}_{n,k}(q))_{k=0}^n$ for $q\geq0$.
 \end{itemize}
\end{thm}

\begin{proof}
(i) Let $\langle y\rangle_k=y(y+1)\cdots(y+k-1)$. By \cite[Theorem
B, p.141]{Com74}, we have
\begin{eqnarray}
F^{\diamond}_{n}(q,y)&=&\sum_{k}F_{n,k}q^k\langle y\rangle_k\\
&=&\sum_{k}F_{n,k}q^k\sum_{i}\left[
  \begin{array}{ccccc}
    k \\
   i\\
  \end{array}
\right]y^i\\
&=&\sum_{i}\left(\sum_{k}F_{n,k}q^k\left[
  \begin{array}{ccccc}
    k \\
   i\\
  \end{array}
\right]\right)y^i,
\end{eqnarray}
where $\left[
  \begin{array}{ccccc}
    k \\
   i\\
  \end{array}
\right]$ is the signless Stirling number of the first kind. This
implies that
\begin{eqnarray}
F^{\diamond}_{n,i}(q)=\sum_{k}F_{n,k}q^k\left[
  \begin{array}{ccccc}
    k \\
   i\\
  \end{array}
\right].
\end{eqnarray}
So we have the decomposition
\begin{eqnarray}\label{decom+F+F}
[F^{\diamond}_{n,k}(q)]_{n,k}=[F_{n,k}q^k]_{n,k}\left[\left[
  \begin{array}{ccccc}
    n \\
   k\\
  \end{array}
\right]\right]_{n,k}=[F_{n,k}]_{n,k}\left[
  \begin{array}{ccccccc}
    1&&& \\
    &q&&\\
   & &q^2&&&\\
 &&&\ddots \\
  \end{array}
\right]\left[\left[
  \begin{array}{ccccc}
    n \\
   k\\
  \end{array}
\right]\right]_{n,k}.
\end{eqnarray} Note by Theorem
\ref{thm+TP+ERA} (i) that $[F_{n,k}]_{n,k}$ is totally positive
 in $\mathbb{R}$. In addition, it is known that
$\left[\left[
  \begin{array}{ccccc}
    n \\
   k\\
  \end{array}
\right]\right]_{n,k}$ is totally positive (see \cite{Bre95} for
instance), which also follows from Theorem \ref{thm+TP+ERA} (i)
because $\left[\left[
  \begin{array}{ccccc}
    n \\
   k\\
  \end{array}
\right]\right]_{n,k}=(1,-\log(1-t))$ in terms of
\begin{eqnarray}
\sum_{n,k}\left[
  \begin{array}{ccccc}
    n \\
   k\\
  \end{array}
\right]x^k\frac{t^n}{n!}=(1-t)^{-x}=e^{x\log\frac{1}{1-t}}.
\end{eqnarray}
Thus, applying the classical Cauchy-Binet formula to the
decomposition above (\ref{decom+F+F}), we have
$[F^{\diamond}_{n,k}(q)]_{n,k}$ is coefficientwise totally positive
in $q$.

(ii) and (iii) In terms of the proof in (i), we have
$F^{\diamond}_{n}(q,y)=\sum_{k}F_{n,k}q^k\langle y\rangle_k$. Note
that $(q^n\langle y\rangle_n)_{n\geq0}$ is a Stieltjes moment
sequence for $q>0$ and $y>0$ because
\begin{eqnarray}
\sum_{n\geq0}q^n\langle
y\rangle_nt^n=\frac{1}{1-\cfrac{qt}{1-\cfrac{qyt}{1-\cfrac{2qt}{1-\cfrac{(y+1)qt}{1-\cdots}}}}}
\end{eqnarray}
by \cite[section 26]{Eul1760}. Thus by Theorem \ref{thm+TP+ERA} (v),
we have $(F^{\diamond}_n(q,y))_{n\geq0}$ is a Stieltjes moment
sequence (of real numbers) for $q\geq0$ and $y\geq0$. This combines
Theorem \ref{thm+conv} to give (iii).

Finally, for (iv), it follows from the Brenti's result \cite[Theorem
2.4.3]{Bre89}: if a polynomial $\sum_{k=0}^n a_{n,k}y^k$ has only
real zeros, then so does $\sum_{k=0}^n a_{n,k}\langle y\rangle_k$.
This completes the proof.\end{proof}

\subsection{Rook polynomials and signless Laguerre polynomials}
Let $\mathfrak{S}_n(q)$ denote the rook polynomial of a square of
side $n$, which is given by
$$\mathfrak{S}_n(q)=\sum_{k=0}^n\binom{n}{k}^2k!q^k$$
(see \cite[Chapter 3. Problems 18]{Rio68} for instance). It
coincides with the matching polynomial of the complete bipartite
graph $K_{n,n}$. It has only real zeros in terms of the rook theory
or matching polynomials and is strongly $q$-log-convex \cite{ZS15}.
Sokal conjectured the following stronger property.

\begin{conj}\emph{\cite{Sok19}}\label{conj+sokal}
The sequence $(\mathfrak{S}_n(q))_{n\geq0}$ is coefficientwise
Hankel-totally positive in $q$.
\end{conj}

For the Laguerre polynomial $L^{(\alpha)}_n(q)$ with $\alpha\geq-1$
(see \cite{AAR99} for instance), its exponential generating function
is
\begin{equation}\label{Exp+GF+Lag}
\sum_{n\geq0}L^{(\alpha)}_n(q)\frac{t^n}{n!}=\frac{1}{(1-t)^{\alpha+1}}\exp\left(\frac{qt}{t-1}\right)
\end{equation}
and its explicit formula is
\begin{eqnarray}
L^{(\alpha)}_n(q)&=&\sum_{k=0}^n\binom{n+\alpha}{n-k}\frac{n!}{k!}(-q)^k
\end{eqnarray}
for $n\geq0$. Let $\widetilde{L}^{(\alpha)}_n(q)=L^{(\alpha)}_n(-q)$
for $n\geq0$. It is obvious for $\alpha=0$ that
\begin{eqnarray}
\mathfrak{S}_n(q)&=&q^n\widetilde{L}^{(0)}_n(1/q).
\end{eqnarray}
For signless Laguerre polynomials $\widetilde{L}^{(\alpha)}_n(q)$,
the following result was recently proved in \cite{Zhu213}, whose
(ii) also implies Conjecture \ref{conj+sokal}. Now we present a
different proof.

\begin{prop}
Let $\alpha\geq-1$ and
$\widetilde{L}^{(\alpha)^{*}}_n(q)=q^n\widetilde{L}^{(\alpha)}_n(1/q)$.
Then we have
 \begin{itemize}
\item [\rm (i)]
the triangular matrix $[\binom{n+\alpha}{n-k}\frac{n!}{k!}]_{n,k}$
is totally positive;
 \item [\rm (ii)]
both $(\widetilde{L}^{(\alpha)}_n(q))_{n\geq0}$ and
$(\widetilde{L}^{(\alpha)^{*}}_n(q))_{n\geq0}$ are coefficientwise
Hankel-totally positive in $q$ and $3$-$q$-log-convex;
\item [\rm (iii)]
the convolution
\begin{eqnarray}
z_n&=&\sum_{k\geq0}\binom{n+\alpha}{n-k}\frac{n!}{k!}x_ky_{n-k}
\end{eqnarray}
preserves the Stieltjes moment property in $\mathbb{R}$.
 \end{itemize}
\end{prop}
\begin{proof}
Let
\begin{eqnarray}
g(t)=\frac{1}{(1-t)^{\alpha+1}},\quad f(t)=\frac{t}{1-t}.
\end{eqnarray}
Rewrite (\ref{Exp+GF+Lag}) as
\begin{equation}
\sum_{n\geq0}\widetilde{L}^{(\alpha)}_n(q)\frac{t^n}{n!}=g(t)\exp(qf(t)).
\end{equation}
Obviously, $\widetilde{L}^{(\alpha)}_n(q)$ is the row-generating
polynomial of $(g(t),f(t)):=[\widetilde{L}^{(\alpha)}_{n,k}]_{n,k}$.
By (\ref{PP}), we have two formal power series
\begin{eqnarray}
Z(t)=(\alpha+1)(1+t),\quad A(t)=(1+t)^2.
\end{eqnarray}
Hence we get the production matrix of $(g(t),f(t))$ is
\begin{eqnarray*}
\left[
\begin{array}{cccccc}
\alpha+1 & 1 &  &  &\\
\alpha+1 & \alpha+3 & 1 &\\
 & 2(\alpha+2) & \alpha+5 &1&\\
 & &3(\alpha+3) & \alpha+7 &1&\\
& && \ddots&\ddots & \ddots \\
\end{array}\right].
\end{eqnarray*}
Thus, the exponential Riordan array
$[\widetilde{L}^{(\alpha)}_{n,k}]_{n,k}$ satisfies the recurrence
\begin{eqnarray}
\widetilde{L}^{(\alpha)}_{n,k}&=&\widetilde{L}^{(\alpha)}_{n-1,k-1}+(\alpha+1+2k)\widetilde{L}^{(\alpha)}_{n-1,k}+(\alpha+k+1)(k+1)\widetilde{L}^{(\alpha)}_{n-1,k+1}
\end{eqnarray}
with $\widetilde{L}^{(\alpha)}_{0,0}=1$. It follows from Theorem
\ref{thm+TP+tridi+even} by taking $a=u=0$, $b=c=x=v=1$ and
$y=\alpha+1$ that we get the desired results in (i)-(iii).
\end{proof}

\subsection{Enumerative labeled trees and forests}

Let $f_{n,k}$ denote the number of forests of rooted trees on $n$
labeled vertices having $k$ components (i.e. $k$ trees) and it was
proved that
\begin{equation}
 f_{n,k}=\binom{n-1}{k-1}n^{n-k},
\end{equation} which also counts the number of rooted labeled trees
on $n+1$ vertices with a root degree $k$, see \cite[A137452]{Slo}
for instance. In particular, $f_{n,1}=n^{n-1}$ is exactly the number
of rooted trees on $n$ labeled vertices and
$\sum_{k}f_{n,k}=(n+1)^{n-1}$ is the number of forests of rooted
trees on $n$ labeled vertices. Denote by $\mathcal
{T}(t)=\sum_{n\geq1}n^{n-1}\frac{t^n}{n!}$ the {\it tree function},
which is closely related to the famous Lambert function
$W(t)=-\mathcal {T}(-t)$. Then one has the exponential generating
function
\begin{equation}
 \sum_{n\geq0}\sum_{k=1}^nf_{n,k}q^k\frac{t^n}{n!}=\exp{(q\mathcal
{T}(t))}.
\end{equation}

  In \cite{Zhu213}, we proved that the matrix $[f_{n,k}]_{n,k}$ is
totally positive and its row-generating polynomial is
coefficientwise Hankel-totally positive. In addition,
$(n^{n-1})_{n\geq1}$ and $((n+1)^{n-1})_{n\geq1}$ are Stieltjes
moment sequences (i.e., their Hankel matrices are totally positive,
respectively). More generally, the convolution
$z_n=\sum_{k\geq0}f_{n+1,k+1}x_ky_{n-k}$ preserves Stieltjes moment
property of sequences. Some of results above were also independently
proved in \cite{Sok21}. Note that
\begin{eqnarray}
\overline{\mathcal {T}}(t)&=&te^{-t},\quad \text{and}
\quad\frac{1}{\overline{\mathcal {T}}'(t)}=\frac{e^t}{1-t}
\end{eqnarray}
is a P\'olya frequency ogf. Thus by Theorem \ref{thm+TP+ERA}, we
have a more generalized result as follows.

\begin{prop} \label{prop+labeled trees+f}
Let $\mathcal {\textbf{F}}^{(i,j)}=[\mathcal
{F}^{(i,j)}_{n,k}]_{n,k}$ be the exponential Riordan array
$(\frac{e^{i\mathcal {T}(t)}}{(1-\mathcal {T}(t))^j},\mathcal
{T}(t))$ for $i\in \{0,1\}$ and $j\in \{0,1\}$ and its
row-generating polynomial $\mathcal
{F}^{(i,j)}_n(q)=\sum_{k}\mathcal {F}^{(i,j)}_{n,k}q^k$. Then we
have
\begin{itemize}
\item [\rm (i)]
 the lower-triangle $\mathcal
{\textbf{F}}^{(i,j)}$ is totally positive;
\item [\rm (ii)]
both the row-generating polynomial sequence $(\mathcal
{F}^{(i,j)}_n(q))_{n}$ and its reversed polynomial sequence
$(\mathcal {F}^{(i,j)*}_n(q))_{n}$ are coefficientwise
Hankel-totally positive in $q$ and $3$-$q$-log-convex;
\item [\rm (iii)]
the convolution $z_n=\sum_{k=0}^{n}\mathcal
{\textbf{F}}^{(i,j)}_{n,k}x_ky_{n-k}$ preserves the Stieltjes moment
property in $\mathbb{R}$.
\end{itemize}
\end{prop}
Obviously, $\mathcal {\textbf{F}}^{(0,0)}=[f_{n,k}]_{n,k}$ and
$\mathcal {\textbf{F}}^{(1,1)}=[f_{n+1,k+1}]_{n,k}$, which have
different combinatorial interpretations. It is natural to ask the
following question.
\begin{ques}
Find the combinatorial interpretation for $\mathcal
{\textbf{F}}^{(0,1)}$ and $\mathcal {\textbf{F}}^{(1,0)}$,
respectively.
\end{ques}

An {\it ordered forest of rooted trees} is simply a forest of rooted
trees in which we have specified a linear ordering of the trees. Let
$f_{n,k}^{ord}$ be the number of ordered forests of rooted trees on
the vertex set $[n]$ with $k$ components and it is obvious that
$f_{n,k}^{ord}=f_{n,k}k!$. It is known that
\begin{equation}
 \sum_{n\geq0}\sum_{k=0}^nf_{n,k}^{ord}q^k\frac{t^n}{n!}=\frac{1}{1-q\mathcal
{T}(t)}.
\end{equation}
Let $F_n^{ord}(q)=\sum_{k=0}^nf_{n,k}^{ord}q^k$. Sokal
\cite[Conjctures 6.7 and 6.8]{Sok21} made the following two
conjectures.
\begin{conj}\label{conj+F}
The polynomial sequence $(F_n^{ord}(q))_{n\geq0}$ is coefficientwise
Hankel-totally positive in $q$.
\end{conj}

\begin{conj}\label{conj+F/n}
The polynomial sequence $(F_n^{ord}(q)/n!)_{n\geq0}$ is
coefficientwise Hankel-totally positive in $q$.
\end{conj}
The following result confirms Conjecture \ref{conj+F/n} and gives a
stronger support for Conjecture \ref{conj+F}.

\begin{prop} Let $F_n^{ord}(q)$ and $f_{n,k}^{ord}$ be defined above. Then we have
\begin{itemize}
\item [\rm (i)]
$(F_n^{ord}(q)/n!)_{n\geq0}$ is coefficientwise Hankel-totally
positive in $q$;
\item [\rm (ii)]
the convolution $z_n=\sum_{k\geq0}\frac{1}{n!}f_{n,k}^{ord}\,
x_ky_{n-k}$ preserves the Stieltjes moment property in $\mathbb{R}$.
 \item [\rm (iii)]
$(F_n^{ord}(q))_{n\geq0}$ is a Stieltjes moment sequence (of real
numbers) for any fixed $q\geq0$;
 \end{itemize}
\end{prop}
\begin{proof}
Note that the sequence $(n^n/n!)_{n\geq0}$ is a Stieltjes moment
sequence because
\begin{eqnarray}
\frac{n^n}{n!}&=&\frac{1}{\pi}\int_{0}^{\pi}\left(\frac{\sin
x}{x}e^{x\cot v}\right)^n dx
\end{eqnarray}
(see \cite[Corollary 2.4]{KJC12} for instance) and $(1/n)_{n\geq1}$
is also a Stieltjes moment sequence. Thus their product sequence
$(n^{n-1}/n!)_{n\geq1}$ is Stieltjes moment. In terms of Stieltjes's
continued fraction criterion for Stieltjes moment sequences, there
exist nonnegative real numbers $\alpha_0,\alpha_1,\alpha_2,\ldots$
such that
\begin{eqnarray}
\mathcal{T}(t)&=&\sum_{n\geq1}\frac{n^{n-1}}{n!}t^n=\frac{t}{1-\cfrac{\alpha_0t}{1-\cfrac{\alpha_1t}{1-\cdots}}}.
\end{eqnarray}
Thus we have
\begin{eqnarray}
\sum_{n\geq0}F_n^{ord}(q)\frac{t^n}{n!}&=&\frac{1}{1-q\mathcal
{T}(t)}=\frac{1}{1-\cfrac{qt}{1-\cfrac{\alpha_0t}{1-\cfrac{\alpha_1t}{1-\cdots}}}}.
\end{eqnarray}
By (\ref{eq+f0+mSfrac}) and Theorem \ref{thm+BSCF+Hankel}, we
immediately get that $(F_n^{ord}(q)/n!)_{n\geq0}$ is a 1-Stieltjes-
Rogers polynomial sequence with coefficientwise Hankel-total
positivity in $q$. This also implies that (ii) holds in terms of
Theorem \ref{thm+conv}. For (iii), note that $(n!)_{n}$ is a
Stieltjes moment sequence and the product of two Stieltjes moment
sequences is still a Stieltjes moment sequence. Thus (iii) is
immediate from (i).
\end{proof}
In fact, the ordered forest number $f_{n,k}^{ord}$ also counts the
number of functional digraphs on the vertex set $[n]$ with $k$
cyclic vertices. Recall that a functional digraph is a directed
graph $G=(V, \overrightarrow{E})$ in which every vertex has
out-degree $1$ and a vertex of a functional digraph is {\it cyclic}
if it lies on one of the cycles (or equivalently, is the root of one
of the underlying trees).  Let $\psi_{n,k}$ be the number of
functional digraphs on the vertex set $[n]$ with $k$ (weakly
connected) components, whose bivariate exponential generating
function is
\begin{equation}
 \sum_{n\geq0}\sum_{k=0}^n\psi_{n,k}y^k\frac{t^n}{n!}=\frac{1}{(1-\mathcal
{T}(t))^y}.
\end{equation}
Let $\bm{\psi}_n(y)=\sum_{k=0}^n\psi_{n,k}y^k$. Sokal
\cite[Conjcture 6.9]{Sok21} made the following conjecture.
\begin{conj}\label{conj+functional digraph+y}Let $\psi_{n,k}$ and $\bm{\psi}_n(y)$ be defined above.
\begin{itemize}
 \item [\rm (i)]
The lower-triangular matrix $[\psi_{n,k}]_{n,k}$ is totally
positive.
\item [\rm (ii)]
The polynomial sequence $(\bm{\psi}_n(y))_{n\geq0}$ is
coefficientwise Hankel-totally positive in $y$.
\item [\rm (iii)]
The sequence $(\psi_{n+1,1})_{n\geq0}$ is Hankel-totally positive
(i.e. is a Stieltjes moment sequence).
 \end{itemize}
\end{conj}
Let $\bm{\psi}_n(q,y)$ be the generating polynomial for functional
digraphs on the vertex set $[n]$ with a weight $q$ for each cyclic
vertex and a weight $y$ for each component. Obviously,
$\bm{\psi}_n(q,1)=F^{ord}_n(q)$ and
$\bm{\psi}_n(1,y)=\bm{\psi}_n(y)$. In addition, one has the
bivariate exponential generating function
\begin{equation}\label{eq+EGF+functional digraphs}
 \sum_{n\geq0}\bm{\psi}_n(q,y)\frac{t^n}{n!}=\frac{1}{(1-q\mathcal
{T}(t))^y}.
\end{equation}
Let $\psi^Y_{n,k}(q)=[y^k]\bm{\psi}_n(q,y)$. Sokal \cite[Conjcture
6.10 ]{Sok21} also conjectured the following, which in particular
implies Conjecture \ref{conj+functional digraph+y}.
\begin{conj}\label{conj+functional digraph+y+q}Let $\bm{\psi}_n(q,y)$ and $\psi^Y_{n,k}(q)$ be defined above.
\begin{itemize}
\item [\rm (i)]
The lower-triangular matrix $[\psi^Y_{n,k}(q)]_{n,k}$ is
coefficientwise totally positive in $q$.
 \item [\rm (ii)]
The polynomial sequence $(\bm{\psi}_n(q,y))_{n\geq0}$ is
coefficientwise Hankel-totally positive in $(q,y)$.
\item [\rm (iii)]
The polynomial sequence $(\bm{\psi}_{n+1,1}^Y(q))_{n\geq0}$ is
coefficientwise Hankel-totally positive in $q$.
 \end{itemize}
\end{conj}

From \cite{Sok21}, we know that
$\bm{\psi}_{n+1,1}^Y=\sum_{k=1}^n(k-1)!f_{n,k}q^k$, which enumerates
connected functional digraphs on the vertex set $[n]$ with a weight
$q$ for each cyclic vertex; equivalently, they enumerate cyclically
ordered forests of rooted trees on the vertex set $[n]$ with a
weight $q$ for each tree. Its exponential generating function is
$-\log(1-q\mathcal {T}(t))$.

For Conjectures \ref{conj+functional digraph+y} and
\ref{conj+functional digraph+y+q}, we obtain the next result, whose
(i) confirms Conjecture \ref{conj+functional digraph+y+q} (i) and
implies Conjecture \ref{conj+functional digraph+y} (i), (ii) gives a
stronger support for Conjecture \ref{conj+functional digraph+y+q}
(ii), Conjecture \ref{conj+functional digraph+y} (ii) and Conjecture
\ref{conj+F}, and (iii) provide a stronger support for Conjecture
\ref{conj+functional digraph+y+q} (iii) and confirms Conjecture
\ref{conj+functional digraph+y} (iii).

\begin{prop}
\begin{itemize}
\item [\rm (i)]
The lower-triangular matrix $[\psi^Y_{n,k}(q)]_{n,k}$ is
coefficientwise totally positive in $q$.
\item [\rm (ii)]
For any fixed $y>0$ and $q>0$, $(\bm{\psi}_n(q,y))_{n\geq0}$ is a
Stieltjes moment sequence (of real numbers).
\item [\rm (iii)]
For any fixed $q\geq0$, $(\bm{\psi}_{n+1,1}^Y(q))_{n\geq0}$ is a
Stieltjes moment sequence (of real numbers).
\item [\rm (iv)]
The convolution $z_n=\sum_{k\geq0}\psi^Y_{n,k}(q)\, x_ky_{n-k}$
preserves the Stieltjes moment property in $\mathbb{R}$ for
$q\geq0$.
\item [\rm (v)]
The polynomial $\bm{\psi}_n(q,y)$ in $y$ has only real zeros for
$q>0$.
 \end{itemize}
\end{prop}
\begin{proof}
For (i) and (ii), they follow from Theorem \ref{thm+TP+ERA+fraction}
(i) and (ii), respectively. (iv) is immediate from (ii). Finally,
for (iii), by Proposition \ref{prop+labeled trees+f} (iii) that the
convolution $z_n=\sum_{k\geq0}f_{n,k} x_ky_{n-k}$ preserves the
Stieltjes moment property in $\mathbb{R}$, for any fixed $q>0$,
$\bm{\psi}_{n+1,1}^Y=\sum_{k=1}^n(k-1)!f_{n,k}q^k$ forms a Stieltjes
moment sequence for $n\geq0$.

(v) It is obviously that
$\sum_{k}f_{n,k}q^k=\sum_{k}\binom{n-1}{k-1}n^{n-k}x^k=(n+x)^{n-1}$
has only real zeros. In consequence, $(f_{n,k})_{k}$ is a P\'{o}lya
frequency sequence, which by Theorem \ref{thm+TP+ERA+fraction} (iv)
implies that $(\psi^Y_{n,k}(q))_{k}$ is a P\'{o}lya frequency
sequence for $q>0$. So the polynomial $\bm{\psi}_n(q,y)$ in $y$ has
only real zeros for $q>0$.
\end{proof}

\subsection{Stirling permutations and $\textbf{r}$th-order Eulerian polynomials}
A few years ago, Sokal ever conjectured that the sequence of
reversed $2$th-order Eulerian polynomials is coefficientwise
Hankel-totally positive. In \cite{PSZ18}, authors gave a
combinatorial proof of coefficientwise Hankel-total positivity for
$r$th-order Eulerian polynomials. Now we will present an algebraic
proof.

Let $\bfr = (r_1,\ldots,r_n)$ and $|\textbf{r}|=\sum_{i}r_i$ for
$r_i\in \mathbb{N}$, and define the multiset $M_\bfr = \{1^{r_1},
2^{r_2}, \ldots, n^{r_n}\}$ consisting of $r_i$ copies of the letter
$i$. A \textbfit{permutation} of $M_\bfr$ is a word $w_1 \cdots
w_{|\bfr|}$ containing $r_i$ copies of the letter $i$, for each $i
\in [n]$; it is called a \textbfit{Stirling permutation} of $M_\bfr$
when the word $\bfw = w_1 \cdots w_L$ satisfies the condition: if $i
< j < k$ and $w_i = w_k$ imply $w_j \ge w_i$. Stirling permutations
were introduced by Gessel and Stanley \cite{GS78} for the case $r_1
= \ldots = r_n = 2$; this was generalized to $r_1 = \ldots = r_n =
r$ [which we denote by the shorthand $\bfr = (r^n)$] by Gessel
\cite{Ges78a} and Park \cite{Park_94a,Park_94b}, and to general
multisets $M_\bfr$ by Brenti \cite{Bre89,Bre98} and others
\cite{Jan11,Dzh14}. We refer to Stirling permutations of $M_{(r^n)}$
as \textbfit{$\bm{r}$-Stirling permutations} of order~$n$.

For any word $\bfw = w_1 \cdots w_L$ on a totally ordered alphabet
$\mathbb{A}$, a pair $(i,i+1)$ with $1 \le i \le L-1$ is called a
\textbfit{descent} if $w_i > w_{i+1}$. Let $\euler{n}{k}^{\! (r)}$
denote the number of $r$-Stirling permutations with $k$ descents and
call them the \textbfit{$r$th-order Eulerian numbers}.\footnote{
   Here we follow the convention of Graham {\em et al.}\/ \cite{Graham_94}
   that (when $n \ge 1$)
   $\euler{n}{k}^{\! (r)}$ is nonzero for $0 \le k \le n-1$.
} The \textbfit{$r$th-order Eulerian polynomial} is defined to be
\be
   A_n^{(r)}(x)
   \;=\;
   \sum\limits_{k=0}^n \euler{n}{k}^{\!\! (r)} x^k
   \;.
 \label{def.eulerianr}
\ee The $r$th-order Eulerian numbers satisfy the recurrence
\be
   \euler{n}{k}^{\! (r)}   \;=\;\big[ rn-(r-1)-k \big] \euler{n-1}{k-1}^{\!\! (r)}
      +(k+1) \euler{n-1}{k}^{\!\! (r)}
   \quad\hbox{for } n \ge 1
 \label{eq.eulerian.recurrence}
\ee with initial condition $\euler{0}{k}^{\! (r)} = \delta_{0k}$.

Note that a bijection from $r$-Stirling permutations of order~$n$ to
increasing $(r+1)$-ary trees with $n$ vertices was found by Gessel
\cite{Ges78a} (see \cite{Park_94a}) and independently by Janson {\em
et al.}\/ \cite{Jan11}. In \cite{PSZ18}, by using this bijection to
map statistics between Stirling permutations and trees, authors
proved that the $r$th-order Eulerian polynomial is a special case of
the multivariate Eulerian polynomial on increasing ary trees and
obtained:

\begin{cor}\cite{PSZ18}
   \label{cor.eulerianr.0}
For any integer $r \ge 1$, the $r$th-order Eulerian polynomial
$A_n^{(r)}(x)$ defined in \reff{def.eulerianr} equals the
$r$-Stieltjes--Rogers polynomial $S_n^{(r)}(\balpha)$ and has the
$r$-branched Stieltjes-type continued fraction
\begin{eqnarray}
\sum_{n\geq0}A_n^{(r)}(x)t^n
   & = &
   \cfrac{1}
         {1 \,-\, \alpha_{r} t
            \prod\limits_{i_1=1}^{r}
                 \cfrac{1}
            {1 \,-\, \alpha_{r+i_1} t
               \prod\limits_{i_2=1}^{r}
               \cfrac{1}
            {1 \,-\, \alpha_{r+i_1+i_2} t
               \prod\limits_{i_3=1}^{r}
               \cfrac{1}{1 - \cdots}
            }
           }
         }
\end{eqnarray}
with coefficients $(\alpha_{i})_{i\geq
r}=(\underbrace{1,x,x,\ldots,x}_{r+1},\underbrace{2,2x,2x,\ldots,2x}_{r+1},\underbrace{3,3x,3x,\ldots,3x}_{r+1},\ldots).$
Therefore the sequence $(A_n^{(r)}(x))_{n \ge 0}$ is coefficientwise
Hankel-totally positive in $x$.
\end{cor}

A slightly different definition of the $r$th-order Eulerian
polynomials was used in \cite{PS20}:
\begin{eqnarray}
   E_0^{[r]}(x)  & = &  1  \\
   E_n^{[r]}(x)  & = &  x \,A_n^{(r)}(x)
       \;=\;  \sum\limits_{k=0}^{n-1} \euler{n}{k}^{\!\! (r)} x^{k+1}
           \quad\hbox{for $n \ge 1$}
 \label{eq.reveuler0}
\end{eqnarray}
Then the reversed $r$th-order Eulerian polynomials are defined by
$E_n^{[r]*}(x) = x^n \, E_n^{[r]}(1/x)$. A few years ago, Sokal ever
conjectured:
\begin{conj}
The sequence $(E_n^{[2]*}(x))_{n \ge 0}$ is coefficientwise
Hankel-totally positive in $x$.
\end{conj}
Using the branched continued fraction for $A_n^{(r)}(x)$, in
\cite{PSZ18} authors also obtained:

\begin{cor}\cite{PSZ18}
   \label{cor.eulerianr}
For any integer $r \ge 1$, the reversed $r$th-order Eulerian
polynomial $E_n^{[r]*}(x)$ equals the $r$-Stieltjes--Rogers
polynomial $S_n^{(r)}(\balpha)$ with $(\alpha_{i})_{i\geq
r}=(\underbrace{1,\ldots,1}_{r},x,\underbrace{2,\ldots,2}_{r},\\x,\underbrace{3,\ldots,3}_{r},x,\ldots).$
Therefore the sequence $(E_n^{[r]*}(x))_{n \ge 0}$ is
coefficientwise Hankel-totally positive in $x$.
\end{cor}

For the coefficientwise Hankel-total positivity of
$(A_n^{(r)}(x))_{n\geq0}$ and $(E_n^{[r]*}(x))_{n \ge 0}$, by
Theorem \ref{thm+TP+ERA}, we will give an algebraic proof from the
exponential Riordan array.

Let $F(t)=\sum_{n}A^{(r)}_n(x)\frac{t^n}{n!}$. In \cite{BSV15}, it
was proved
$$F(t)=\frac{h\left(e^{t(1-x)^r}\overline{h}(x)\right)}{x}\frac{1-x}{1-h\left(e^{t(1-x)^r}\overline{h}(x)\right)},$$
where the compositional inverse of the function $h(z)$ is given by
\begin{eqnarray}
\overline{h}(z)&=&z\exp\left(\sum_{k=1}^{r-1}\binom{r-1}{k}\frac{(-z)^k}{k}\right).
\end{eqnarray}
Let
\begin{eqnarray}\label{eq+f+E2}
f(t):=F(t)-1=\frac{h\left(e^{t(1-x)^r}\overline{h}(x)\right)-x}{x\left(1-h\left(e^{t(1-x)^r}\overline{h}(x)\right)\right)}.
\end{eqnarray}
We have
\begin{eqnarray}h\left(e^{t(1-x)^r}\overline{h}(x)\right)&=&\frac{x(1+f(t))}{1+xf(t)},
\end{eqnarray}
which implies
\begin{eqnarray}
xe^{\overline{f}(t)(1-x)^r}\exp\left(\sum_{k=1}^{r-1}\binom{r-1}{k}\frac{(-x)^k}{k}\right)&=&\frac{x(1+t)}{1+xt}\exp\left(\sum_{k=1}^{r-1}\frac{1}{k}\binom{r-1}{k}\left(-\frac{x(1+t)}{1+xt}\right)^k\right).\nonumber\\
\end{eqnarray}
Then we derive
\begin{eqnarray}\frac{1}{\overline{f}'(t)}&=&(1+xt)^{r}(1+t),
\end{eqnarray}
whose coefficient sequence is coefficientwise Toeplitz-totally
positive in $x$. Then by Theorem \ref{thm+TP+ERA} and Theorem
\ref{thm+TP+ERA+branched fraction}, we conclude more results for
total positivity related to $A_n^{(r)}(x)$ and $E_n^{[r]*}(x)$ as
follows.
\begin{prop}\label{prop+r+roder+E+poly}
 Let $f(t)$ be defined by (\ref{eq+f+E2}) and $[\euler{n}{k}^{\!\! (i,j)}]_{n,k}$ be the exponential Riordan array $((1+xf(t))^{i}(1+f(t))^j,f(t))$
for $i\in\{0,1,\ldots,r\}$ and $j\in\{0,1\}$. Then
\begin{itemize}
\item [\rm (i)]
 the triangle $[\euler{n}{k}^{\!\! (i,j)}]_{n,k}$ is coefficientwise totally
positive in $x$;
\item [\rm (ii)]
the row-generating polynomial $(A^{(i,j)}_n(q))_{n}$ of
$[\euler{n}{k}^{\!\! (i,j)}]_{n,k}$ and its reversed polynomial
$(A^{(i,j)*}_n(q))_{n}$ are coefficientwise Hankel-totally positive
in $(x,q)$;
\item [\rm (iii)]
the $r$th-order Eulerian polynomial $(A_n^{(r)}(x))_{n\geq1}$ is
coefficientwise Hankel-totally positive in $x$;
\item [\rm (iv)]
 the sequence $(T_n^{(i,j)}(x))_{n\geq0}$ is coefficientwise Hankel-totally positive
in $x$, where
 $(1+xf(t))^{i}(1+f(t))^j=\sum_{n\geq0}T_n^{(i,j)}(x)\frac{t^n}{n!}$;
 \item [\rm (v)]
the convolution $z_n=\sum_{k=0}^{n}\euler{n}{k}^{\!\!
(i,j)}x_ky_{n-k}$ preserves the Stieltjes moment property in
$\mathbb{R}$ for $x\geq1$;
 \item [\rm (vi)]
we have the $r$-branched Stieltjes-type continued fraction
\begin{eqnarray}
1+\sum_{n\geq1}A_n^{(r)}(x)t^n
   =
   \cfrac{1}
         {1 \,-\, \alpha_{r} t
            \prod\limits_{i_1=1}^{r}
                 \cfrac{1}
            {1 \,-\, \alpha_{r+i_1} t
               \prod\limits_{i_2=1}^{r}
               \cfrac{1}
            {1 \,-\, \alpha_{r+i_1+i_2} t
               \prod\limits_{i_3=1}^{r}
               \cfrac{1}{1 - \cdots}
            }
           }
         }
\end{eqnarray}
with coefficients $(\alpha_{i})_{i\geq
r}=(\underbrace{1,x,x,\ldots,x}_{r+1},\underbrace{2,2x,2x,\ldots,2x}_{r+1},\underbrace{3,3x,3x,\ldots,3x}_{r+1},\ldots).$
\end{itemize}
\end{prop}
Obviously, $A^{(i,j)}_n(0)=T_n^{(i,j)}(x)$ and
$A^{(r,1)}_n(0)=A_{n+1}^{(r)}(x)$. Therefore, we ask the following
question.
\begin{ques}
Let $i\in\{0,1,\ldots,r\}$ and $j\in\{0,1\}$. Find the combinatorial
interpretation of the triangular array $[\euler{n}{k}^{\!\!
(i,j)}]_{n,k}$ and polynomials $A^{(i,j)}_n(q)$ and
$T_n^{(i,j)}(x)$, respectively.
\end{ques}

\subsection{Multivariate Ward polynomials}
The Ward numbers $W_{n,k}$ were studied in \cite{War34}, which
satisfy the recurrence
\begin{eqnarray}
W_{n,k}&=&(n+k-1)W_{n-1,k-1}+kW_{n-1,k}
\end{eqnarray}
for $n\geq k\geq0$ with $W_{0,0}=1$. The row-generating polynomial
of the Ward triangle $W_n(x)=\sum_{k=0}^nW_{n,k}x^k$ is called the
{\it Ward polynomial}. It is known that
$W_n(x)=(1+x)^nA^{(2)}_{n}\left(\frac{1}{1+x}\right)$, which
combines Propositions \ref{prop+recp} and \ref{prop+r+roder+E+poly}
to give that both $W_n(x)$ and $W^{*}_n(x)$ are coefficientwise
Hankel-totally positive in $x$, see \cite{PS20} for a different
proof from the Thon-continued fraction. In fact, Elvey Price and
Sokal \cite{PS20} obtained a Thron-type continued fraction for the
ordinary generating function of a five-variable Ward polynomial. See
\cite[A134991/A181996/A269939]{Slo} for further information on the
Ward numbers and Ward polynomials. There are different combinatorial
interpretations for the Ward polynomials. For example, the Ward
polynomial $W_n(x)$ is the generating polynomial for phylogenetic
trees on $n+1$ labeled leaves in which each internal vertex gets a
weight $x$. More generally, a multivariate Ward polynomial
$\textbf{W}_n(x_1,x_2,\ldots)$ is the generating polynomial for
phylogenetic trees on $n+1$ labeled leaves in which each internal
vertex with $i\geq2$ children gets a weight $x_{i-1}$. We list the
first few $\textbf{W}_n(x_1,x_2,\ldots)$ as follows:
\begin{eqnarray}\textbf{W}_0(x_1,x_2,\ldots)=1,\quad \textbf{W}_1(x_1,x_2,\ldots)=x_1,\quad \textbf{W}_2(x_1,x_2,\ldots)=3x_1^2+x_2.
\end{eqnarray}
Let
\begin{eqnarray}
\mathscr{W}(t,\textbf{x})=\sum_{n\ge0}\textbf{W}_n(x_1,x_2,\ldots)\frac{t^{n+1}}{(n+1)!}=t+\sum_{n\ge2}\textbf{W}_{n-1}(x_1,x_2,\ldots)\frac{t^{n}}{n!}.
\end{eqnarray}
It is known that the compositional inverse of the generic power
series is
\begin{eqnarray}
\mathscr{F}(t,x)=t-\sum_{n\ge2}x_{n-1}\frac{t^{n}}{n!}.
\end{eqnarray}
That is to say
\begin{eqnarray} \label{Eq+mult+War}
\mathscr{W}(t,\textbf{x})=t+\sum_{n\ge2}x_{n-1}\frac{\mathscr{W}(t,\textbf{x})^{n}}{n!}.
\end{eqnarray}
We refer the reader to \cite[pp. 151]{Com74}, \cite{HS89,PS20},
\cite[pp. 181]{Rio68} for more properties of
$\textbf{W}_n(x_1,x_2,\ldots)$. By (\ref{Eq+mult+War}), we derive an
autonomous differential equation
\begin{eqnarray}
\mathscr{W}'(t,\textbf{x})=\frac{1}{1-\sum_{n\ge1}x_{n}\frac{\mathscr{W}(t,\textbf{x})^{n}}{n!}}.
\end{eqnarray}
Hence, by Theorem \ref{thm+TP+ERA+oder+r} (iv), we have:
\begin{prop}
Let $x_1,x_2,\ldots$ be elements of a partially ordered commutative
ring $R$ such that $\frac{1}{1-\sum_{n\ge1}x_{n}\frac{t^{n}}{n!}}$
is a P\'olya frequency ogf of order $r$ in the ring $R$, then the
sequence $(\textbf{W}_n(x_1,x_2,\ldots))_{n\geq1}$ is Hankel-totally
positive of order $r$ in the ring $R$.
\end{prop}

\subsection{Labeled series-parallel networks and Nondegenerate fanout-free functions}
Let $s_n$ be the number of labeled series-parallel networks with $n$
edges. It dates from MacMahon studying Yoke-trains and multipartite
compositions \cite{Mac1891}. It also counts the number of plane
increasing trees on $n$ vertices, where each vertex of outdegree
$k\geq1$ can be in one of $2$ colors. For $n\geq1$, the first few
are $1,2,8,52,472,\ldots$. Gutkovskiy \cite[A006351]{Slo} noted that
it satisfies the recurrence
\begin{eqnarray}
s_1=1,s_n=s_{n-1}+\sum_{k=1}^{n-1}\binom{n-1}{k}s_ks_{n-k}.
\end{eqnarray}
Let $s(t)=\sum_{n\geq1}s_n\frac{t^n}{n!}$. Bala \cite[A006351]{Slo}
gave the autonomous differential equation
\begin{eqnarray}
s'(t)=\frac{1+s(t)}{1-s(t)}.
\end{eqnarray} We refer the reader to Lomnicki
\cite{Lom72}, Moon \cite{Moo87} and \cite[A006351]{Slo} for more
information on $s_n$.

Let $p_n$ be $1,4,32,416,7552,176128,\ldots$ for $n\geq1$, see
\cite[A005172]{Slo}. It counts the number of nondegenerate
fanout-free functions of $n$ variables with the basic logical
operation ``AND" rank $1$ \cite{Ha76} and also is the number of
plane increasing trees on $n$ vertices, where each vertex of
outdegree $k\geq1$ can be colored in $2^{k+1}$ ways (see Bala's
remarks in \cite[A005172]{Slo}). Kruchinin \cite[A005172]{Slo}
obtained for $n>1$ that
\begin{eqnarray}
p_n &=& \sum_{k=1}^{n-1} \sum_{j=1}^k\sum_{i=0}^j
\frac{2^{n-i+j-1}(n+k-1)!}{(k-j)!(n-i+j-1)!i!}\left[
  \begin{array}{ccccc}
    n-i+j-1 \\
   j-i\\
  \end{array}
\right].
\end{eqnarray}
Luschny \cite[A005172]{Slo} also gave
$p_n=2p_{n-1}+\sum_{j=1}^{n-1}\binom{n}{j}p_jp_{n-j}$ for $n>1$. Let
the exponential generating function
$p(t)=\sum_{n\geq1}p_n\frac{t^n}{n!}$. It satisfies the autonomous
differential equation
\begin{eqnarray}p'(t)& =& \frac{1+2p(t)}{1-2p(t)}
\end{eqnarray}
with $p(0)= 0$. We refer the reader to \cite[A042977]{Slo} for more
properties on $p_n$.

In order state some properties related to $s_n$ and $p_n$ for
$n\geq1$, we will consider a generalization of $s(t)$ and $p(t)$.
Let $a$ and $b$ be indeterminates. Define
$S(t)=\sum_{n\geq1}S_n(a,b)\frac{t^n}{n!}$ by
\begin{eqnarray}\label{func+equ+S}
S'(t)&=&\frac{1+a\,S(t)}{1-b\,S(t)}.
\end{eqnarray}
By (\ref{func+equ+S}), we have
\begin{eqnarray}
\frac{1}{\overline{S}'(t)}&=&\frac{1+a\,t}{1-b\,t}.
\end{eqnarray}
Obviously, the coefficient sequence of $\frac{1+a\,t}{1-b\,t}$ is
coefficientwise Toeplitz-totally positive in $a$ and $b$. It follows
from Theorem \ref{thm+TP+ERA} that we immediately obtain:

\begin{thm}\label{thm+networks+S}
Let $a$ and $b$ be indeterminates and
$S(t)=\sum_{n\geq1}S_n(a,b)\frac{t^n}{n!}$ satisfy the equation
(\ref{func+equ+S}). Define an exponential Riordan array
$[S^{(i,j)}_{n,k}(a,b)]_{n,k}:=\left(\frac{(1+a\,S(t))^i}{(1-b\,S(t))^j},S(t)\right)$
for $i\in\{0,1\}$ and $j\in\{0,1\}$. Then we have
\begin{itemize}
\item [\rm (i)]
 the triangle $[S^{(i,j)}_{n,k}(a,b)]_{n,k}$ is coefficientwise totally
positive in $(a,b)$;
\item [\rm (ii)]
the row-generating polynomial sequence $(S^{(i,j)}_n(q))_{n}$ and
its reversed polynomial sequence $(S^{(i,j)*}_n(q))_{n}$ are
coefficientwise Hankel-totally positive in $(a,b,q)$ and
$3$-$(a,b,q)$-log-convex;
\item [\rm (iii)]
the zeroth column sequence $(S^{(i,j)}_{n,0}(a,b))_{n}$ is
coefficientwise Hankel-totally positive in $(a,b)$ and
$3$-$(a,b)$-log-convex;
\item [\rm (iv)]
the sequence $(S_n(a,b))_{n\geq1}$ is coefficientwise Hankel-totally
positive in $(a,b)$ and $3$-$(a,b)$-log-convex;
\item [\rm (v)]
 the sequence $(S^{\circ}_{n}(b))_{n\geq0}$ is coefficientwise Hankel-totally positive in $b$ and
$3$-$b$-log-convex, where
 $\frac{1}{1-b\,S(t)}=\sum_{n\geq0}S^{\circ}_n(b)\frac{t^n}{n!}$;
 \item [\rm (vi)]
the convolution $z_n=\sum_{k=0}^{n}S^{(i,j)}_{n,k}(a,b)x_ky_{n-k}$
preserves the Stieltjes moment property in $\mathbb{R}$ for $a\geq0$
and $b\geq0$.
\end{itemize}
\end{thm}

In particular, taking $a=b=1$ in (\ref{func+equ+S}), then $S(t)$
reduces to $s(t)$ and taking $a=b=2$ in (\ref{func+equ+S}) yields
$p(t)$. So the following is immediate from Theorem
\ref{thm+networks+S}.
\begin{prop}
We have
\begin{itemize}
\item [\rm (i)]
the sequence $(s_{n+1})_{n\geq0}$ is Stieltjes moment and
$3$-log-convex;
\item [\rm (ii)]
 the sequence $(s^{\circ}_{n})_{n\geq0}$ is a Stieltjes moment sequence (of
real numbers) and $3$-log-convex, where
 $\frac{1}{1-s(t)}=\sum_{n\geq0}s^{\circ}_n\frac{t^n}{n!}$;
\item [\rm (iii)]
the sequence $(p_{n+1})_{n\geq0}$ is a Stieltjes moment sequence (of
real numbers) and $3$-log-convex;
\item [\rm (iv)]
 the sequence $(p^{\circ}_{n})_{n\geq0}$ is a Stieltjes moment sequence (of
real numbers) and $3$-log-convex, where
 $\frac{1}{1-2p(t)}=\sum_{n\geq0}p^{\circ}_n\frac{t^n}{n!}$.
\end{itemize}
\end{prop}

It is very interesting to find an answer to the following question.
\begin{ques}
Find the combinatorial interpretation of the triangular array
$[S^{(i,j)}_{n,k}(a,b)]_{n,k}$ and its row-generating polynomial
$S^{(i,j)}_n(q)$, respectively.
\end{ques}

\subsection{An array from the Lambert function} The Lambert $W$ function
was defined in \cite{CGHJK} by $$We^W=x.$$ The $n$-th derivative of
$W$ is given implicitly by
\begin{eqnarray}
\frac{d^nW(x)}{dx^n}&=&\frac{e^{-nW(x)}\beta_n(W(x))}{(1+W(x))^{2n-1}}
\end{eqnarray}
for $n\geq1$, where
$\beta_n(x)=(-1)^{n-1}\sum_{k=0}^{n-1}\beta_{n,k}x^k$ are
polynomials satisfying the recurrence relation
\begin{eqnarray}
\beta_{n+1}(x)=-(nx+3n-1)\beta_n(x)+(1+x)\beta'_n(x)
\end{eqnarray}
for $n\geq1$, and the array $[\beta_{n,k}]_{n\geq1,k\geq0}$
satisfies the recurrence relation
\begin{eqnarray}\label{rec+lab+array}
\beta_{n+1,k}=(3n-k-1)\beta_{n,k}+n\beta_{n,k-1}-(k+1)\beta_{n,k+1}
\end{eqnarray}
for $n,k\geq0$ with $\beta_{1,0}=1$. Kalugin and Jeffrey \cite{KJ11}
proved that each polynomial $(-1)^{n-1}\beta_n(x)$ has all positive
coefficients and $\frac{dW(x)}{dx}$ is a completely monotonic
function. An explicit expression for the coefficients $\beta_{n,k}$
is
\begin{eqnarray}
\beta_{n,k}&=&\sum_{m=0}^k\frac{1}{m!}\binom{2n-1}{k-m}\sum_{i=0}^m\binom{m}{i}(-1)^i(i+n)^{m+n-1}.
\end{eqnarray}
See \cite[A042977]{Slo} for more properties of $\beta_{n,k}$.
Jovovic \cite[A042977]{Slo} gave
\begin{eqnarray}
\sum_{n}(-1)^{n-1}\beta_n(x)\frac{t^n}{n!}&=&\frac{-W(e^x(x-t(1+x)^2))+x}{1+x}.
\end{eqnarray}
Let \begin{eqnarray}\label{def+function+beta}
\beta(t)=\frac{-W(e^x(x-t(1+x)^2))+x}{1+x}.
\end{eqnarray}  We derive that
\begin{eqnarray}
\overline{\beta}(t)=\frac{x+[(1+x)t-x]e^{-(1+x)t}}{(1+x)^2},\quad
\frac{1}{(\overline{\beta}(t))'}=\frac{e^{(1+x)t}}{1-t}.
\end{eqnarray}
It is not hard to prove that $e^{(1+x)t}/(1-t)$ is a P\'olya
frequency function in $x$. Therefore, by Theorem \ref{thm+TP+ERA},
we obtain:

\begin{prop} Let $\beta(t)$ be defined by (\ref{def+function+beta}) and $\bm\beta^{(i,j)}=[\bm\beta^{(i,j)}_{n,k}]_{n,k}$ be the exponential Riordan array
$\left(\frac{e^{i(1+x)\beta(t)}}{(1-\beta(t))^j},\beta(t)\right)$
for $i\in\{0,1\}$ and $j\in\{0,1\}$. Then
\begin{itemize}
\item [\rm (i)]
 the triangle $\bm\beta^{(i,j)}$ is coefficientwise totally positive in $x$;
\item [\rm (ii)]
the row-generating polynomial sequence $(\bm\beta^{(i,j)}_n(q))_{n}$
and its reversed polynomial sequence $(\bm\beta^{(i,j)*}_n(q))_{n}$
are coefficientwise Hankel-totally positive in $(x,q)$ and
$3$-$(x,q)$-log-convex;
\item [\rm (iii)]
$((-1)^{n-1}\beta_n(x))_{n\geq1}$ is coefficientwise Hankel-totally
positive in $x$ and $3$-$x$-log-convex;
\item [\rm (iv)]
$({\beta}^{\circ}_{n}(x))_{n\geq0}$ is coefficientwise
Hankel-totally positive in $x$ and $3$-$x$-log-convex, where
 $\frac{1}{1-\beta(t)}=\sum_{n\geq0}{\beta}^{\circ}_n(x)\frac{t^n}{n!}$;
 \item [\rm (v)]
$({\beta}^{\bullet}_{n}(x))_{n\geq0}$ is coefficientwise
Hankel-totally positive in $x$ and $3$-$x$-log-convex, where
 $e^{(1+x)\beta(t)}=\sum_{n\geq0}{\beta}^{\bullet}_n(x)\frac{t^n}{n!}$;
 \item [\rm (vi)]
the convolution $z_n=\sum_{k=0}^{n}\bm\beta^{(i,j)}_{n,k}x_ky_{n-k}$
preserves the Stieltjes moment property in $\mathbb{R}$ for
$x\geq0$.
\end{itemize}
\end{prop}

It is obvious that $\bm\beta^{(0,0)}=(1,\beta(t))$ and
$\bm\beta^{(1,1)}=(\beta'(t),\beta(t))=[\bm\beta^{(0,0)}_{n+1,k+1}]_{n,k}$.
Therefore we pose the following question.
\begin{ques}
Find the combinatorial interpretation of the triangular arrays
$\bm\beta^{(0,1)}$ and $\bm\beta^{(1,0)}$ and the polynomial
$\bm\beta^{(i,j)}_n(q)$, respectively.
\end{ques}

\subsection{A generalization of Lah numbers}

In \cite{Zhu213}, we studied a generalized Lah triangle
$\mathscr{L}=[\mathscr{L}_{n,k}]_{n,k\geq0}$ satisfying the
following recurrence relation
\begin{eqnarray}\label{rec+four+}
\mathscr{L}_{n,k}=c\,
\mathscr{L}_{n-1,k-1}+\left[ab(n-1)+bk+abd+c\lambda\right]\mathscr{L}_{n-1,k}+b\lambda(k+1)
\mathscr{L}_{n-1,k+1}
\end{eqnarray}
for $n,k\geq1$, where $\mathscr{L}_{0,0}=1$. For $a=b=c=1$ and
$d=\lambda=0$, $\mathscr{L}$ reduces to the well-known signless Lah
triangle $[\binom{n-1}{k-1}\frac{n!}{k!}]_{n,k}$, which counts the
number of partitions of $[n]$ into $k$ lists, where a list means an
ordered subset \cite[A008297]{Slo} and satisfies
\begin{eqnarray}\exp\left(\frac{qt}{1-t}\right)&=&\sum_{n\geq0}
\sum_{k=0}^n\binom{n-1}{k-1}\frac{n!}{k!}q^k \frac{t^n}{n}
\end{eqnarray}
\cite[p.133-134]{Com74}. In addition, the row-generating function
$\sum_{k=0}^n\binom{n-1}{k-1}\frac{n!}{k!}q^k$ is called {\it the
Lah polynomial}. For $2\leq a\leq 4$, $b=c=1$ and $d=\lambda=0$,
$[\mathscr{L}_{n,k}]_{n,k}$ reduces to the triangle \cite[A035342,
A035469, A049029]{Slo} enumerating unordered $n$-vertex $k$-forests
composed of $k$ plane increasing quartic ($a$-ary) trees. In
\cite{Zhu213}, we showed some results concerning total positivity
for the triangle $\mathscr{L}$ under some special conditions. As an
application of Theorem \ref{thm+TP+ERA}, we have the following
generalized result.

\begin{prop}\label{prop+two+term+triangle}
Let $\mathscr{L}$ be the generalized Lah triangle defined by
(\ref{rec+four+}) and
$\mathscr{L}_n(q)=\sum_{k\geq0}\mathscr{L}_{n,k}q^k$. Assume that
$\{a,ad\}\in \mathbb{N}$ and $c>0$. If $0\leq ad\leq a+1=m$, then we
have
\begin{itemize}
\item [\rm (i)]
the generalized Lah triangle $\mathscr{L}$ is coefficientwise
totally positive in $(b,\lambda)$;
 \item [\rm (ii)]
the polynomial sequence $(\mathscr{L}_{n}(q))_{n\geq0}$ and its
reversed polynomial sequence $(\mathscr{L}_{n}^{*}(q))_{n\geq0}$ are
coefficientwise Hankel-totally positive in $(b,\lambda,q)$ and
$3$-$(b,\lambda,q)$-log-convex;
\item [\rm (iii)]$(\mathscr{L}_{n,0})_{n\geq0}$ is coefficientwise Hankel-totally positive in
$(b,\lambda)$ and $3$-$(b,\lambda)$-log-convex;
\item [\rm (iv)]
the convolution $z_n=\sum_{k\geq0}\mathscr{L}_{n,k}x_ky_{n-k}$
preserves Stieltjes moment property in $\mathbb{R}$ for $b\geq0$ and
$\lambda\geq0$;
\item [\rm (v)]
for $d=0$, the $m$-branched Stieltjes-type continued fraction
expression
\begin{eqnarray}
   \sum_{n\geq0}\mathscr{L}_{n}(q)t^n
=
   \cfrac{1}
         {1 \,-\, \alpha_{m} t
            \prod\limits_{i_1=1}^{m}
                 \cfrac{1}
            {1 \,-\, \alpha_{m+i_1} t
               \prod\limits_{i_2=1}^{m}
               \cfrac{1}
            {1 \,-\, \alpha_{m+i_1+i_2} t
               \prod\limits_{i_3=1}^{m}
               \cfrac{1}{1 - \cdots}
            }
           }
         }
%
\end{eqnarray}
with coefficients
$$(\alpha_{i})_{i\geq{m}}=(c(q+\lambda),\underbrace{b,\ldots,b}_{m},c(q+\lambda),\underbrace{2b,\ldots,2b}_{m},c(q+\lambda),\underbrace{3b,\ldots,3b}_{m},\ldots).$$
\item [\rm (vi)]
the $(m-1)$-branched Stieltjes-type continued fraction expression
\begin{eqnarray}
  1+\sum_{n\geq1}bf_{n}t^n
   =
   \cfrac{1}
         {1 \,-\, \alpha_{m-1} t
            \prod\limits_{i_1=1}^{m-1}
                 \cfrac{1}
            {1 \,-\, \alpha_{m-1+i_1} t
               \prod\limits_{i_2=1}^{m-1}
               \cfrac{1}
            {1 \,-\,  - \cdots}
            }
            }
\end{eqnarray}
with
$(\alpha_{i})_{i\geq{m-1}}=(\underbrace{b,\ldots,b}_{m},\underbrace{2b,\ldots,2b}_{m},\underbrace{3b,\ldots,3b}_{m},\ldots),$where
$\sum_{n\geq1}f_n\frac{t^n}{n!}=\frac{(1-abt)^{-\frac{1}{a}}-1}{b}$.
 \end{itemize}
\end{prop}
\begin{proof}
It follows from \cite[Proposition 4.7]{Zhu213} that $\mathscr{L}$ is
the exponential Riordan array $(g(t)e^{\lambda f(t)},f(t))$, where
\begin{eqnarray}g(t)=(1-abt)^{-d},\quad
f(t)=c\left[\frac{(1-abt)^{-\frac{1}{a}}-1}{b}\right].
\end{eqnarray}
Then
\begin{eqnarray}
\bar{f}(t)&=&\frac{1-(1+\frac{bt}{c})^{-a}}{ab},\quad\,1/\bar{f}'(t)=c(1+\frac{b}{c}t)^{a+1}.
\end{eqnarray}
Obviously, the coefficient sequence of $c(1+\frac{b}{c}t)^{a+1}$ is
coefficientwise Toeplitz-totally positive in $b$ and
\begin{eqnarray}
(1+\frac{b}{c}f(t))^{ad}&=&\left[(1-abt)^{-\frac{1}{a}}\right]^{ad}=(1-abt)^{-d}=g(t).
\end{eqnarray}
Hence, by Theorem \ref{thm+TP+ERA}, we have (i)-(iv)

For (v), by taking $q\rightarrow c(q+\lambda)/b$ and $t\rightarrow
bt$, it suffices to prove for the reduced case $b=c=1$ and
$\lambda=0$ that
\begin{eqnarray}
   \sum_{n\geq0}\mathscr{L}_{n}(q)t^n
   & = &
   \cfrac{1}
         {1 \,-\, \alpha_{m} t
            \prod\limits_{i_1=1}^{m}
                 \cfrac{1}
            {1 \,-\, \alpha_{m+i_1} t
               \prod\limits_{i_2=1}^{m}
               \cfrac{1}
            {1 \,-\, \alpha_{m+i_1+i_2} t
               \prod\limits_{i_3=1}^{m}
               \cfrac{1}{1 - \cdots}
            }
           }
         }
%
\end{eqnarray}
with coefficients
$$(\alpha_{i})_{i\geq{m}}=(q,\underbrace{1,\ldots,1}_{m},q,\underbrace{2,\ldots,2}_{m},q,\underbrace{3,\ldots,3}_{m},\ldots).$$
For the reduced case, we have
$$\frac{1}{\bar{f}'(t)}=(1+t)^{m}.$$ Hence we immediately get the continued fraction expansion for the reduced case
by Theorem \ref{thm+TP+ERA+branched fraction}. This completes the
proof.
\end{proof}

We propose the following interesting question.
\begin{ques}
What is the combinatorial interpretation of the triangular array
$[\mathscr{L}_{n,k}]_{n,k\geq0}$ and the polynomial
$\mathscr{L}_n(q)$, respectively ?
\end{ques}

For the reversed generalized Lah polynomial
$\mathscr{L}_n^{*}(q)=q^n\mathscr{L}_n(1/q)$ and the reversed
generalized Lah numbers $\mathscr{L}^*_{n,k}=\mathscr{L}_{n,n-k}$,
the following is immediate from Proposition
\ref{prop+two+term+triangle}.

\begin{prop}
Assume that $\{a,ad\}\in \mathbb{N}$ and $c>0$. If $0\leq ad\leq
a+1=m$, then we have the following results.
\begin{itemize}
\item [\rm (i)]
The reversed generalized Lah triangle
$[\mathscr{L}^*_{n,k}]_{n,k\geq0}$ satisfies the next recurrence
\begin{eqnarray*}
\mathscr{L}^*_{n,k}=c\,
\mathscr{L}^*_{n-1,k}+\left[ab(n-1)+b(n-k)+abd+c\lambda\right]\mathscr{L}^*_{n-1,k-1}+b\lambda(n-k+1)
\mathscr{L}^*_{n-1,k-2}
\end{eqnarray*}
for $n,k\geq1$, where $\mathscr{L}^*_{0,0}=1$.
\item [\rm (ii)]
The exponential generating function of $\mathscr{L}^*_{n}(q)$ can be
written as
\begin{eqnarray}
\sum_{n\geq0}\mathscr{L}^*_{n}(q)\frac{t^n}{n!}&=&(1-abqt)^{-d}\exp\left(\frac{(q\lambda+1)c}{q}\left[\frac{(1-abqt)^{-\frac{1}{a}}-1}{b}\right]\right).
\end{eqnarray}
\item [\rm (iii)]
The sequence $(\mathscr{L}^*_{n}(q))_{n\geq0}$ is coefficientwise
Hankel-totally positive in $(b,\lambda,q)$ and
$3$-$(b,\lambda,q)$-log-convex. In particular,
$(\mathscr{L}^*_{n,n})_{n\geq0}$ is coefficientwise Hankel-totally
positive in $(b,\lambda)$ and $3$-$(b,\lambda)$-log-convex.
\item [\rm (iv)]
The convolution $z_n=\sum_{k\geq0}\mathscr{L}^*_{n,k}x_ky_{n-k}$
preserves Stieltjes moment property of sequences for $b\geq0$ and
$\lambda\geq0$;
\item [\rm (v)]
For $d=0$, we have the $m$-branched Stieltjes-type continued
fraction expansion
\begin{eqnarray}
   \sum_{n\geq0}\mathscr{L}^*_{n}(q)t^n
   =
   \cfrac{1}
         {1 \,-\, \alpha_{m} t
            \prod\limits_{i_1=1}^{m}
                 \cfrac{1}
            {1 \,-\, \alpha_{m+i_1} t
               \prod\limits_{i_2=1}^{m}
               \cfrac{1}
            {1 \,-\, \alpha_{m+i_1+i_2} t
               \prod\limits_{i_3=1}^{m}
               \cfrac{1}{1 - \cdots}
            }
           }
         }
%
\end{eqnarray}
with coefficients
$$(\alpha_{i})_{i\geq{m}}=(c(1+q\lambda),\underbrace{bq,\ldots,bq}_{m},c(1+q\lambda),\underbrace{2bq,\ldots,2bq}_{m},c(1+q\lambda),\underbrace{3bq,\ldots,3bq}_{m},\ldots).$$
\end{itemize}
\end{prop}

\section{Acknowledgements}
The author wants to thank the anonymous referees for many valuable
remarks and suggestions to improve the original manuscript and is
very grateful to one of referees for suggesting some results for the
matrices $M$ and $\widetilde{M}$ in Section $3$.

\end{document}